\documentclass[reqno, 12pt]{amsart}
\pdfoutput=1
\makeatletter
\let\origsection=\section \def\section{\@ifstar{\origsection*}{\mysection}}
\def\mysection{\@startsection{section}{1}\z@{.7\linespacing\@plus\linespacing}{.5\linespacing}{\normalfont\scshape\centering\S}}
\makeatother

\usepackage{amsmath,amssymb,amsthm}
\usepackage{mathrsfs}
\usepackage{mathabx}\changenotsign
\usepackage{dsfont}
\usepackage{bbm}

\usepackage{xcolor}
\usepackage[backref=section]{hyperref}
\usepackage[ocgcolorlinks]{ocgx2}
\hypersetup{
    colorlinks=true,
    linkcolor={red!60!black},
    citecolor={green!60!black},
    urlcolor={blue!60!black},
}

\usepackage[open,openlevel=2,atend]{bookmark}

\usepackage[abbrev,msc-links,backrefs]{amsrefs}
\usepackage{doi}

\renewcommand{\PrintDOI}[1]{\doi{#1}}

\usepackage[T1]{fontenc}
\usepackage{lmodern}
\usepackage[babel]{microtype}
\usepackage[english]{babel}

\usepackage{geometry}
\numberwithin{equation}{section}
\numberwithin{figure}{section}

\usepackage{enumitem}
\def\rmlabel{\upshape({\itshape \roman*\,})}

\def\alabel{\upshape({\itshape \alph*\,})}

\def\nlabel{\upshape({\itshape \arabic*\,})}

\let\polishlcross=\l
\def\l{\ifmmode\ell\else\polishlcross\fi}

\def\tand{\ \text{and}\ }
\def\qand{\quad\text{and}\quad}
\def\qqand{\qquad\text{and}\qquad}

\let\emptyset=\varnothing
\let\setminus=\smallsetminus

\makeatletter
\def\moverlay{\mathpalette\mov@rlay}
\def\mov@rlay#1#2{\leavevmode\vtop{   \baselineskip\z@skip \lineskiplimit-\maxdimen
   \ialign{\hfil$\m@th#1##$\hfil\cr#2\crcr}}}
\newcommand{\charfusion}[3][\mathord]{
    #1{\ifx#1\mathop\vphantom{#2}\fi
        \mathpalette\mov@rlay{#2\cr#3}
      }
    \ifx#1\mathop\expandafter\displaylimits\fi}
\makeatother

\newcommand{\dcup}{\charfusion[\mathbin]{\cup}{\cdot}}

\DeclareFontFamily{U}  {MnSymbolC}{}
\DeclareSymbolFont{MnSyC}         {U}  {MnSymbolC}{m}{n}
\DeclareFontShape{U}{MnSymbolC}{m}{n}{
    <-6>  MnSymbolC5
   <6-7>  MnSymbolC6
   <7-8>  MnSymbolC7
   <8-9>  MnSymbolC8
   <9-10> MnSymbolC9
  <10-12> MnSymbolC10
  <12->   MnSymbolC12}{}
\DeclareMathSymbol{\powerset}{\mathord}{MnSyC}{180}

\usepackage{tikz}
\usetikzlibrary{calc,decorations.pathmorphing}
\usetikzlibrary{arrows,decorations.pathreplacing}
\pgfdeclarelayer{background}
\pgfdeclarelayer{foreground}
\pgfdeclarelayer{front}
\pgfsetlayers{background,main,foreground,front}

\newcommand{\qedge}[7]{

	\ifx\relax#4\relax
		\def\qoffs{0pt}
	\else
		\def\qoffs{#4}
	\fi

	\def\qhedge{
		($#1+#3!\qoffs!-90:#2-#3$) --
		($#2+#1!\qoffs!-90:#3-#1$) --
		($#3+#2!\qoffs!-90:#1-#2$) -- cycle}

	\coordinate (12) at ($#1!\qoffs!90:#2$);
	\coordinate (13) at ($#1!\qoffs!-90:#3$);
	\coordinate (23) at ($#2!\qoffs!90:#3$);
	\coordinate (21) at ($#2!\qoffs!-90:#1$);
	\coordinate (31) at ($#3!\qoffs!90:#1$);
	\coordinate (32) at ($#3!\qoffs!-90:#2$);
	
	\def\nqhedge{
		(13) let \p1=($(13)-#1$), \p2=($(12)-#1$) in
			arc[start angle={atan2(\y1,\x1)}, delta angle={atan2(\y2,\x2)-atan2(\y1,\x1)-360*(atan2(\y2,\x2)-atan2(\y1,\x1)>0)}, x radius=\qoffs, y radius=\qoffs] --
		(21) let \p1=($(21)-#2$), \p2=($(23)-#2$) in
			arc[start angle={atan2(\y1,\x1)}, delta angle={atan2(\y2,\x2)-atan2(\y1,\x1)-360*(atan2(\y2,\x2)-atan2(\y1,\x1)>0)}, x radius=\qoffs, y radius=\qoffs] --
		(32) let \p1=($(32)-#3$), \p2=($(31)-#3$) in
			arc[start angle={atan2(\y1,\x1)}, delta angle={atan2(\y2,\x2)-atan2(\y1,\x1)-360*(atan2(\y2,\x2)-atan2(\y1,\x1)>0)}, x radius=\qoffs, y radius=\qoffs] --
		cycle}

		\ifx\relax#5\relax
		\def\qlwidth{1pt}
	\else
		\def\qlwidth{#5}
	\fi
	
		\ifx\relax#7\relax
		\fill \nqhedge;
	\else
		\fill[#7]\nqhedge;
	\fi

		\ifx\relax#6\relax
		\draw[line width=\qlwidth,rounded corners=\qoffs]\nqhedge;
	\else
		\draw[line width=\qlwidth,#6]\nqhedge;
	\fi
}

\newcommand{\redge}[8]{

		\ifx\relax#5\relax
		\def\qoffs{0pt}
	\else
		\def\qoffs{#5}
	\fi

				\def\rhedge{
		($#1+#4!\qoffs!-90:#2-#4$) -- 
		($#2+#1!\qoffs!-90:#3-#1$) -- 
		($#3+#2!\qoffs!-90:#4-#2$) -- 
		($#4+#3!\qoffs!-90:#1-#3$) -- cycle}

	\coordinate (12) at ($#1!\qoffs!90:#2$);
	\coordinate (14) at ($#1!\qoffs!-90:#4$);
	\coordinate (23) at ($#2!\qoffs!90:#3$);
	\coordinate (21) at ($#2!\qoffs!-90:#1$);
	\coordinate (34) at ($#3!\qoffs!90:#4$);
	\coordinate (32) at ($#3!\qoffs!-90:#2$);
	\coordinate (41) at ($#4!\qoffs!90:#1$);
	\coordinate (43) at ($#4!\qoffs!-90:#3$);
	
	\def\nrhedge{
		(14) let \p1=($(14)-#1$), \p2=($(12)-#1$) in 
			arc[start angle={atan2(\y1,\x1)}, delta angle={atan2(\y2,\x2)-atan2(\y1,\x1)-360*(atan2(\y2,\x2)-atan2(\y1,\x1)>0)}, x radius=\qoffs, y radius=\qoffs] --
		(21) let \p1=($(21)-#2$), \p2=($(23)-#2$) in 
			arc[start angle={atan2(\y1,\x1)}, delta angle={atan2(\y2,\x2)-atan2(\y1,\x1)-360*(atan2(\y2,\x2)-atan2(\y1,\x1)>0)}, x radius=\qoffs, y radius=\qoffs] --
		(32) let \p1=($(32)-#3$), \p2=($(34)-#3$) in 
			arc[start angle={atan2(\y1,\x1)}, delta angle={atan2(\y2,\x2)-atan2(\y1,\x1)-360*(atan2(\y2,\x2)-atan2(\y1,\x1)>0)}, x radius=\qoffs, y radius=\qoffs] --
		(43) let \p1=($(43)-#4$), \p2=($(41)-#4$) in 
			arc[start angle={atan2(\y1,\x1)}, delta angle={atan2(\y2,\x2)-atan2(\y1,\x1)-360*(atan2(\y2,\x2)-atan2(\y1,\x1)>0)}, x radius=\qoffs, y radius=\qoffs] --
		cycle}

		\ifx\relax#6\relax
		\def\rlwidth{1pt}
	\else
		\def\rlwidth{#6}
	\fi
	
		\ifx\relax#8\relax
		\fill \nrhedge;
	\else
		\fill[#8]\nrhedge;
	\fi

		\ifx\relax#7\relax
		\draw[line width=\rlwidth,rounded corners=\qoffs]\nrhedge;
	\else
		\draw[line width=\rlwidth,#7]\nrhedge;
	\fi
	}

\let\epsilon=\varepsilon

\let\rho=\varrho
\let\theta=\vartheta
\let\wh=\widehat

\def\eu{\mathrm e}

\def\EE{{\mathds E}}
\def\NN{{\mathds N}}

\def\PP{{\mathds P}}
\def\RR{{\mathds R}}

\def\Ind{{\mathds 1}}

\newcommand{\cA}{\mathcal{A}}
\newcommand{\cB}{\mathcal{B}}

\newcommand{\cP}{\mathcal{P}}

\newcommand{\cR}{\mathcal{R}}
\newcommand{\cS}{\mathcal{S}}

\newcommand{\ccA}{\mathscr{A}}
\newcommand{\ccB}{\mathscr{B}}

\newcommand{\ccC}{\mathscr{C}}
\newcommand{\ccP}{\mathscr{P}}

\newcommand{\ccW}{\mathscr{W}}

\newcommand{\gS}{\mathfrak{S}}

\newcommand{\gE}{\mathfrak{E}}
\newcommand{\gX}{\mathfrak{X}}
\newcommand{\gP}{\mathfrak{P}}

\newtheoremstyle{note}  {4pt}  {4pt}  {\sl}  {}  {\bfseries}  {.}  {.5em}          {}
\newtheoremstyle{introthms}  {3pt}  {3pt}  {\itshape}  {}  {\bfseries}  {.}  {.5em}          {\thmnote{#3}}
\newtheoremstyle{remark}  {2pt}  {2pt}  {\rm}  {}  {\bfseries}  {.}  {.3em}          {}

\theoremstyle{plain}
\newtheorem{theorem}{Theorem}[section]
\newtheorem{lemma}[theorem]{Lemma}

\newtheorem{prop}[theorem]{Proposition}

\newtheorem{cor}[theorem]{Corollary}

\newtheorem{claim}[theorem]{Claim}

\theoremstyle{note}
\newtheorem{dfn}[theorem]{Definition}
\newtheorem{setup}[theorem]{Setup}

\theoremstyle{remark}

\newtheorem{question}[theorem]{Question}

\usepackage{lineno}
\newcommand*\patchAmsMathEnvironmentForLineno[1]{
\expandafter\let\csname old#1\expandafter\endcsname\csname #1\endcsname
\expandafter\let\csname oldend#1\expandafter\endcsname\csname end#1\endcsname
\renewenvironment{#1}
{\linenomath\csname old#1\endcsname}
{\csname oldend#1\endcsname\endlinenomath}}
\newcommand*\patchBothAmsMathEnvironmentsForLineno[1]{
\patchAmsMathEnvironmentForLineno{#1}
\patchAmsMathEnvironmentForLineno{#1*}}
\AtBeginDocument{
\patchBothAmsMathEnvironmentsForLineno{equation}
\patchBothAmsMathEnvironmentsForLineno{align}
\patchBothAmsMathEnvironmentsForLineno{flalign}
\patchBothAmsMathEnvironmentsForLineno{alignat}
\patchBothAmsMathEnvironmentsForLineno{gather}
\patchBothAmsMathEnvironmentsForLineno{multline}
}

\def\Ubad{U_\text{\rm bad}}
\usepackage{scalerel}

\makeatletter
\newcommand{\overrighharpoonup}[1]{\ThisStyle{%
 \vbox {\m@th\ialign{##\crcr
 \rightharpoonupfill \crcr
 \noalign{\kern-\p@\nointerlineskip}
 $\hfil\SavedStyle#1\hfil$\crcr}}}}

\def\rightharpoonupfill{%
$\SavedStyle\m@th\mkern+0.8mu\cleaders\hbox{$\shortbar\mkern-4mu$}\hfill\rightharpoonuptip\mkern+0.8mu$}

\def\rightharpoonuptip{%
 \raisebox{\z@}[2pt][1pt]{\scalebox{0.55}{$\SavedStyle\rightharpoonup$}}}

\def\shortbar{%
 \smash{\scalebox{0.55}{$\SavedStyle\relbar$}}}
\makeatother

\let\seq=\overrighharpoonup

\def\sa{\seq a}
\def\sb{\seq b}
\def\su{\seq u}
\def\sw{\seq w}
\def\sx{\seq x}

\def\zetas{\zeta_{\star}}
\def\thetas{\theta_{\star}}
\def\zetass{\zeta_{\star\star}}
\def\thetass{\theta_{\star\star}}

\def\copr{\;\middle|\;}
\def\coprn{\mid}
\def\bcopr{\mathrel{\big|}}

\begin{document}
\dedicatory{Dedicated to Endre Szemer\'edi on the occasion of his $80^{\text{th}}$ birthday}

\title[Minimum pair degree for Hamiltonian cycles in $4$-uniform hypergraphs]{Minimum pair degree condition for tight Hamiltonian cycles in $4$-uniform hypergraphs}
\author[J.~Polcyn]{Joanna Polcyn}
\address{Adam Mickiewicz University, Faculty of Mathematics and Computer Science, Pozna\'n, Poland}
\email{joaska@amu.edu.pl}
\email{rucinski@amu.edu.pl}

\author[Chr.~Reiher]{Christian Reiher}
\address{Fachbereich Mathematik, Universit\"at Hamburg, Hamburg, Germany}
\email{christian.reiher@uni-hamburg.de}
\email{schacht@math.uni-hamburg.de}
\email{bjarne.schuelke@uni-hamburg.de}

\author[V.~R\"{o}dl]{Vojt\v{e}ch R\"{o}dl}
\address{Department of Mathematics, Emory University, Atlanta, USA}
\email{vrodl@emory.edu}
\thanks{The third author is supported by NSF grant DMS 1764385.}

\author[A.~Ruci\'nski]{Andrzej Ruci\'nski}

\thanks{The fourth author is supported by Polish NSC grant 2018/29/B/ST1/00426.}

\author[M.~Schacht]{Mathias Schacht}
\thanks{The fifth author is supported by the ERC (PEPCo 724903)}

\author[B. Sch\"ulke]{Bjarne Sch\"ulke}

%\subjclass[2020]{Primary: 05C65. Secondary: 05C45}
\subjclass[2010]{Primary: 05C65. Secondary: 05C45}
\keywords{Hamiltonian cycles, Dirac's theorem, hypergraphs}

\begin{abstract}
We show that every 4-uniform hypergraph with $n$ vertices and minimum pair degree 
at least $(5/9+o(1))n^2/2$ contains a tight Hamiltonian cycle. 
This degree condition is asymptotically optimal. 
\end{abstract}

\maketitle

\section{Introduction}
We study hypergraph generalisations of Dirac's theorem for graphs. 
For hypergraphs several extensions were considered and Endre Szemer\'edi has been an integral 
part and driving force for these developments. All but the last author already
had the pleasure to collaborate with and learn from Endre, while working on related (and unrelated) problems.

\subsection{Background and main result}
G.\,A.~Dirac~\cite{dirac} showed that every (finite) 
graph $G=(V,E)$ on at least $3$ vertices with minimum degree $\delta(G)\geq |V|/2$ contains a
Hamiltonian cycle. This result is clearly best possible, as exemplified by slightly off-balanced 
complete bipartite graphs. Several hypergraph extensions were suggested and considered in the 
literature. Here we focus on \emph{tight}
Hamiltonian cycles in uniform hypergraphs and we briefly review the relevant notation.

For an integer~$k\ge2$, a \emph{$k$-uniform hypergraph}  
is a pair~$(V,E)$, where the \emph{vertex set} $V$ is a finite set and the \emph{edge set}
$E\subseteq V^{(k)}=\{U\subseteq V\colon |U|=k\}$ is some collection of $k$-element subsets of $V$.
A \emph{tight Hamiltonian cycle} in a $k$-uniform hypergraph $H=(V,E)$ is given by a cyclic ordering 
of~$V$ such that every $k$ consecutive vertices (in the cyclic ordering) span a hyperedge from $E$.  
As usual for $k=2$, we recover the notion of finite, simple graphs and Hamiltonian cycles.

For $k>2$, large part of the research concerns sufficient minimum degree conditions in hypergraphs that 
guarantee the existence of tight Hamiltonian cycles (see, e.g., the surveys~\cites{RR-survey,Zhao-sur} 
and the references therein for a more thorough discussion). 
For a set of vertices $S\subseteq V$, the \emph{degree} in~$H$ 
is defined by 
\[
	d_H(S)=\big|\{e\in E\colon S\subseteq e\}\big| 
\]
and for an integer $j$ with $1\leq j\leq k$ the \emph{minimum $j$-degree} is defined by 
\[
	\delta_j(H)=\min\{d_H(S)\colon S\in V^{(j)}\}\,.
\]
The minimum $1$-degree $\delta_1(H)$ is often called minimum \emph{vertex degree}
and sometimes (in particular, in the context of graphs) we may
omit the subscript. Moreover, for $j=2$ we often refer to $\delta_2(H)$ as the minimum \emph{pair degree}.

Lower bounds on the minimum $j$-degree bear more information and restrictions for larger values 
of~$j$ and, in fact, sufficient minimum $(k-1)$-degree conditions for tight Hamiltonian cycles were considered first in the 
literature. This line of research was initiated by Katona and Kierstead~\cite{KK99}. In joint work with 
the two senior authors, Endre~\cites{RRS08} established the following asymptotically optimal result for $k$-uniform hypergraphs (see also~\cite{rrs3} for an earlier result for 
$3$-uniform hypergraphs and~\cite{3} for a sharp version of that result).
\begin{theorem}[R$^2$Sz, 2008]
\label{thm:apxDirac}
	For every integer $k\geq 3$ and $\alpha>0$, there exists an integer~$n_0$ such that every
	$k$-uniform hypergraph $H$ on $n\ge n_0$ vertices with 
	$\delta_{k-1}(H)\geq \bigl(\tfrac12+\alpha\bigr)n$ contains a tight Hamiltonian cycle.
	\qed
\end{theorem}
Theorem~\ref{thm:apxDirac} can be viewed as an approximate generalisation of Dirac's theorem from graphs to hypergraphs and, in fact, the lower bound constructions, that show the optimality of this result, exhibit a similar bipartite structure.

Given the `monotonicity' of the degree conditions, as a next step it seems natural to consider an extension of Theorem~\ref{thm:apxDirac} with a minimum $(k-2)$-degree condition. For such an extension we have to restrict ourselves to $k$-uniform hypergraphs for $k\geq 3$. Improving a series of partial results by several authors, 
for $3$-uniform hypergraphs the following asymptotically optimal result was recently obtained by Endre 
in collaboration with the middle four authors~\cite{old}. 

\begin{theorem}[R$^3$SSz, 2019]
	\label{thm:main3}
	For every $\alpha>0$, there exists an integer $n_0$ such that every
	$3$-uniform hypergraph $H$ on $n\ge n_0$ vertices with 
	$\delta_1(H)\geq \bigl(\tfrac59+\alpha\bigr)\tfrac{n^2}{2}$ contains a tight Hamiltonian cycle.
	\qed
\end{theorem}
Again there are lower bound constructions, showing that the number $5/9$ in Theorem~\ref{thm:main3} is best possible. 
In fact, three structurally different examples can be found in~\cite{old}*{Example~1.2}. Here we 
extend Theorem~\ref{thm:main3} to $4$-uniform hypergraphs with a minimum pair degree condition
and establish the following result.

\begin{theorem}
	\label{thm:main}
	For every $\alpha>0$ there exists an integer $n_0$ such that every
	\mbox{$4$-uniform} hypergraph~$H$ on $n\ge n_0$ vertices with 
	$\delta_2(H)\geq \bigl(\tfrac59+\alpha\bigr)\tfrac{n^2}{2}$ contains a tight Hamiltonian cycle.
\end{theorem}
Theorem~\ref{thm:main} is also asymptotically best possible as the following examples
of Han and Zhao~\cite{HZ16} show:
\begin{enumerate}[label=\alabel]
\item\label{it:ex-a} 
	For simplicity let $n=|V|$ be divisible by three and consider a partition $X\dcup Y=V$ with $|X|=2n/3$.
	Let $H$ be the $4$-uniform 
	hypergraph $H=(V,E)$ with $e\in V^{(4)}$ being an edge of $H$ if, and only if,
	\begin{equation}\label{eq:count-ex-i}
		|e\cap X|\neq 2\,.
	\end{equation} 
	It is easy to check that $H$ satisfies $\delta_2(H)\geq \bigl(\tfrac59-o(1)\bigr)\tfrac{n^2}{2}$. 
	
	Suppose for the sake of contradiction that~$H$ contains a tight Hamiltonian cycle~$C$. 
	Since every vertex of $C$ is 
	contained in precisely four edges of $C$, we have
	\[
		\sum_{f\in E(C)}|f\cap X|=4\,|X|\,.
	\]
	Hence, the average intersection of an edge of $C$ with~$X$ is~$8/3$.
	In particular, there exist two edges $f$ and $f'$ in $C$ such that 
	\[
		|f\cap X|\leq 2=\lfloor\tfrac{8}{3}\rfloor
		\qqand
		|f'\cap X|\geq 3=\lceil\tfrac{8}{3}\rceil
	\]
	and the definition of $H$ implies that  
	$f$ shares at most one vertex with $X$. 
	
	On the other hand, the sizes of the intersections in $X$ of two 
	consecutive edges in~$C$ (in the induced cyclic order) can differ by at most one. Consequently, 
	the lack of edges in~$E$ intersecting $X$ in exactly two vertices makes the occurrence of the 
	edges~$f$ and~$f'$ in $C$ impossible.
\item\label{it:ex-b} The same construction with~\eqref{eq:count-ex-i} replaced by $|e\cap X|\neq 3$ 
	yields another hypergraph exemplifying a matching lower bound for 
	Theorem~\ref{thm:main} by a similar argument.
\end{enumerate}

The type of construction used in~\ref{it:ex-a} and~\ref{it:ex-b} above 
generalises to arbitrary uniformities~$k\geq 3$. In fact, if $3\divides k$ then this gives rise to 
three structurally different lower bound constructions and if $3\ndivides k$ then two hypergraphs arise  (see~\cite{HZ16}*{Corollary~1.6} for details). Those examples show that 
the optimal minimum  $(k-2)$-degree for tight Hamiltonian cycles in $k$-uniform hypergraphs on $n$ vertices is at least 
$\bigl(\tfrac59-o(1)\bigr)\tfrac{n^2}{2}$. 

The results discussed so far address special cases of the following more general problem:
Given integers $k>r\ge 1$, determine the infimal real number $\alpha^{(k)}_r\in [0, 1]$ with 
the property that every $k$-uniform hypergraph $H=(V, E)$ satisfying the minimum $r$-degree 
condition $\delta_r(H)\ge \bigl(\alpha^{(k)}_r+o(1)\bigr)|V|^{k-r}/(k-r)!$ contains a 
Hamiltonian cycle. 
Thus Dirac's theorem and Theorem~\ref{thm:apxDirac} assert $\alpha^{(k)}_{k-1}=1/2$ for $k\ge 2$, 
while the Theorems~\ref{thm:main3} and~\ref{thm:main} entail $\alpha^{(3)}_1=\alpha^{(4)}_2=5/9$.
These results might indicate that $\alpha^{(k)}_r$ might be determined by the difference 
$d=k-r$, which leads to the following question.

\begin{question}
	Given a positive integer $d$, does there exist a constant $\beta_d$ such that 
	$\alpha^{(k)}_{k-d}=\beta_d$ holds for every $k\ge d+1$?	
\end{question}
We are not aware of any counterexample and for $d=1$ 
Theorem~\ref{thm:apxDirac} states that $\beta_1=1/2$. Moreover, very recently 
Theorems~\ref{thm:main3} and~\ref{thm:main} were extended for arbitrary $k\geq 5$ in~\cite{TYH}
and $\beta_2=5/9$ was established.
The lower bounds on $\alpha^{(k)}_r$ obtained by Han and Zhao~\cite{HZ16} might be optimal 
for all $k>r\ge 1$. In this case, all numbers $\beta_d$ would exist and the next problem 
would be to decide whether $\beta_3=5/8$. 

\subsection{Overview and organisation}
The proof of Theorem~\ref{thm:main} is based on the \emph{absorption method}.
This method has been introduced more than a decade ago in~\cites{rrs3} (see also the survey~\cite{Sz-surv} of Endre Szemer\'edi) and
since then it has turned out to be a versatile tool for solving a  
variety of problems concerning the existence of spanning structures in graphs 
and hypergraphs.
Proofs 
based on the absorption method usually decompose the problem at hand into several 
more manageable subproblems. In results on Hamiltonian cycles in hypergraphs
such as Theorem~\ref{thm:main} most of the effort is usually directed towards showing 
a {\it connecting lemma}, an {\it absorbing lemma}, and a {\it covering lemma}.

The complexity of the first two ingredients has evolved over time. For instance,
in the proof of Theorem~\ref{thm:apxDirac} for $k=3$ in~\cite{rrs3},
the connecting lemma just said that every pair 
of vertices can be connected to any other pair of vertices by means of a relatively short tight
path. An analogous result is not available when proving  Theorem~\ref{thm:main3}
(see~\cite{old}).
Instead, one defines a sufficiently broad class of so-called {\it connectable pairs} 
of vertices and, roughly speaking, the connecting lemma of~\cite{old} asserts that any 
such connectable pair can be reached from any other connectable pair by means of a short 
tight path. This idea will be reused below, so we shall define a notion of connectable triples 
in $4$-uniform hypergraphs of large pair degree and our connecting lemma 
(Proposition~\ref{lem:con} below) claims that any two such triples can be connected by means 
of a short tight path.    
 
As for the absorbing lemma (see Proposition~\ref{prop:absorbingP}), one needs to establish 
the existence of a so-called 
{\it absorbing path} $P_A$ capable of absorbing any ``small'' set of left-over 
vertices~$Z$. More precisely, no matter which small set~$Z$ of vertices needs attention 
in the end of the argument there always is a path with vertex set $V(P_A)\cup Z$
which starts and ends with the same vertices as $P_A$ itself. Such a path $P_A$ 
is usually constructed by taking several small building blocks called {\it absorbers}
and connecting them by appealing to the connecting lemma. Proving the existence of 
suitable absorbers has often been among the main difficulties in applying the absorption method. 
Recently, the first two authors, while studying a related problem, observed that in many cases this problem 
can be reduced to a degenerate Tur\'an-type problem~\cite{PR}. In fact, ignoring for a moment the issue that the 
absorbers need to be connectable into a tight path, their existence is a direct consequence 
of a classical extremal result of Erd\H os~\cite{E64}, for the small price that the size of $Z$ needs to satisfy 
an additional divisibility assumption (see \S\,\ref{sec:abs_intro} for more details).

Finally, the covering lemma (see Proposition~\ref{prop:1625}) asserts, in particular, 
that the minimum pair degree condition considered in Theorem~\ref{thm:main} 
ensures the existence of an almost perfect path cover.
Then the connecting lemma allows us to connect the paths from the cover together with $P_A$.
In fact, there even exists a cycle~$C$ containing paths from the cover and the absorbing path $P_A$
for which the (small) set $Z=V(H)\setminus V(C)$ of remaining vertices satisfies the aforementioned divisibility condition. 
Now, to complete the proof of Theorem~\ref{thm:main} one just needs to absorb the vertices outside $C$ into 
the absorbing path. 

As mentioned above, the proof of Theorem~\ref{thm:main} presented here reuses some ideas and results from~\cite{old} and we collect the relevant material in the next section. Sections~\ref{conn}\,--\,\ref{sec:long_path} establish the 
connecting lemma, absorbing lemma, covering lemma, and the so-called \emph{reservoir lemma}, 
which ensures that the short tight paths used for the connections are always vertex disjoint from the rest. 
In Section~\ref{sec:main-pf} we then present the somewhat standard proof of Theorem~\ref{thm:main} based on these lemmata. 
 
\section{Preliminaries}
\subsection{Notation}
Besides graphs, we mainly consider 3-uniform and 4-uniform hypergraphs,
and here we briefly recall some relevant definitions.
For simplicity, if there is no danger of confusion 
we sometimes omit parentheses, braces, and commas and denote edges $\{x,y\}$, $\{x,y,z\}$, or~$\{x,y,z,w\}$
in graphs and $3$- and $4$-uniform hypergraphs by $xy$, $xyz$, or $xyzw$, respectively. 

\subsubsection*{Walks, paths, and cycles}
We shall only consider tight walks, paths, and cycles and for simplicity we omit the word tight from now on.
The length of a walk, a path, or a cycle is measured by its number of edges. 

For $3$-uniform hypergraphs 
a \emph{walk~$W$} of length $\l\geq 0$ is given by a sequence $(x_1,\dots,x_{\l+2})$ of vertices such that 
$e\in E(W)$ if and only if $e=x_ix_{i+1}x_{i+2}$ for some $i\in[\l]$. 
The ordered pairs $(x_1,x_2)$ and $(x_{\l+1},x_{\l+2})$
are the \emph{end-pairs} of~$W$ and we say $W$ is a $(x_1,x_2)$-$(x_{\l+1},x_{\l+2})$-walk.  
This definition of end-pairs is not symmetric and implicitly fixes a direction 
of~$W$ and sometimes we may refer to $(x_1,x_2)$ and $(x_{\l+1},x_{\l+2})$ as \emph{starting pair} 
and  \emph{ending pair}, respectively. The vertices $x_3,\dots,x_\l$ are the 
{\it inner vertices} of $W$ and in the context of walks 
we count the inner vertices with their multiplicities, i.e., for $\l\geq 2$ a walk of length~$\l$ has~$\l-2$ 
inner vertices. We often identify a walk 
with the sequence of its vertices $x_1x_2\dots x_{\l+2}$ and refer to it as a $x_1x_2$-$x_{\l+1}x_{\l+2}$-walk.

A walk $W$ is a \emph{path} if all the vertices $x_1,\dots,x_{\l+2}$ are distinct and it is a \emph{cycle} if 
the vertices $x_1,\dots,x_{\l}$ are distinct and $x_{\l+1}=x_1$ and $x_{\l+2}=x_2$.
 
These definitions extend to $4$-uniform hypergraphs in a straightforward way. In this context a walk of length $\l$ 
is given by a sequence of $\l+3$ vertices and the \emph{end-triples} are the subsequences of the first and the 
last three vertices.

\subsubsection*{Links of vertices and pairs}
We recall that the \emph{link graph} of a vertex 
$v$ of a $3$-uniform hypergraph~$H$ is defined to be the graph $H_v$ with the 
same vertex set as $H$ and with 
\[
	E(H_v)=\{xy\colon vxy\in E(H)\}\,.
\] 
Similarly, for a $4$-uniform hypergraph~$H$ the link $H_v$ of a vertex $v$ 
is a $3$-uniform hypergraph on the same vertex set 
with $E(H_v)=\{xyz\colon vxyz\in E(H)\}$.
Moreover, for an unordered pair of distinct vertices $u$ and $v$ the \emph{link of the pair $uv$}
is the graph $H_{uv}$ with vertex set~$V(H_{uv})=V(H)$
and edge set
\[
	E(H_{uv})=\{xy\colon uvxy\in E(H)\}\,.
\]

\subsection{Robust subgraphs}
\label{subsec:2239}
Both in the $3$-uniform predecessor~\cite{old} of this work and here 
the connecting lemma is deduced from certain connectivity properties of $2$-uniform 
link graphs. In the present subsection we discuss the graph theoretic result we shall
require for this purpose. We begin with the key notion in this regard 
(cf.~\cite{old}*{Definition~2.2}).   

\begin{dfn} \label{dfn:robust}
		Given $\beta>0$ and $\ell\in\NN$ a graph~$R$ is said to 
		be {\it $(\beta, \l)$-robust} 
		if for any two distinct vertices~$x$ and~$y$ of~$R$ the number of $x$-$y$-paths of 
		length $\ell$ is at least $\beta |V(R)|^{\ell-1}$.
	\end{dfn}
	
The main point is that graphs whose density is larger than $5/9$ possess 
sufficiently dense robust subgraphs containing more than two thirds of the 
vertices. The following result along those lines is a slight strengthening  
of~\cite{old}*{Proposition~2.3} and below we shall only indicate how the arguments 
in~\cite{old} can be modified so as to yield the present version. 

\begin{prop} \label{prop:robust}
	Given $\alpha$, $\mu >0$, there exist $\beta>0$ and an odd integer $\ell\ge 3$
	such that for sufficiently large~$n$, every $n$-vertex graph $G=(V, E)$ 
	with $|E|\ge \left(\frac{5}{9}+\alpha\right)\frac{n^2}{2}$
	contains a $(\beta,\l)$-robust induced subgraph $R\subseteq G$ satisfying
		\begin{enumerate}[label=\rmlabel]
			\item\label{it:rc1} $|V(R)|\ge \bigl(\frac23+\frac\alpha 2\bigr)n$,
			\item\label{it:rc2} $e_G\big(V(R),V\setminus V(R)\big)\leq 
				\mu n^2$, 			
			\item\label{it:rc3} and $e(R)\ge 
				\left(\frac{5}{9}+\frac{\alpha}{2}\right)\frac{n^2}{2}-\frac{(n-|V(R)|)^2}{2}
				\ge\left(\frac49+\frac23\alpha\right)\frac{n^2}{2}$.
		\end{enumerate}
\end{prop}

\begin{proof}
	We may assume $\alpha\leq4/9$, for otherwise there are no $n$-vertex
	graphs~$(V, E)$ satisfying $|E|\ge \left(\frac{5}{9}+\alpha\right)\frac{n^2}{2}$
	and there is nothing to show.
	The proof of~\cite{old}*{Lemma~3.2} shows for every fixed $\mu'\le\alpha/72$
	that every graph $G=(V, E)$ on $n\gg 1/\mu'$ vertices such that 
	$|E|\ge \left(\frac{5}{9}+\alpha\right)\frac{n^2}{2}$ has 
	an induced subgraph $R$ satisfying~\ref{it:rc1}, 
	\begin{equation}\label{eq:2350}
		e_G\big(V(R),V\setminus V(R)\big) < 4\mu' n^2\,, 
	\end{equation}
	and the first estimate in~\ref{it:rc3} which, moreover, has a property 
	called $\mu'$-inseparability (see~\cite{old}*{Definition~3.1}). 
	For the purposes of~\cite{old} it was enough to apply this fact to $\mu'=\alpha/72$
	itself, but here it will be more convenient to set $\mu'=\min\{\mu/4, \alpha/72\}$, which 
	causes~\eqref{eq:2350} to imply~\ref{it:rc2}. The second estimate in~\ref{it:rc3}
	is an immediate consequence of~\ref{it:rc1} and of $\alpha\leq 4/9<2/3$.
		
	It remains to show that $R$ is indeed $(\beta, \ell)$-robust for some constants
	$\beta$ and $\ell$ that only depend on $\alpha$ and $\mu$ but not on $n$. As the 
	proof of~\cite{old}*{Proposition~2.3} shows, this follows from the $\mu'$-inseparability 
	of $R$ combined with the fact that~\ref{it:rc1} and~\ref{it:rc3} allow us to 
	bound the density of $R$ from below. In fact, it is enough to let
	$\ell$ be the least odd integer such that 
	\[
		\ell >8\left(\frac{1}{\mu'}\right)^{2}+1
		\quad \text{ and to set } \quad 
		\beta =\frac 1{72}\left(\frac{\mu'}2\right)^{6\ell}\,. \qedhere
	\]
\end{proof}

The next result will assist us (indirectly via Lemma~\ref{L35}) in Section~\ref{sec:Abpa} when we wish to 
ensure that the end-triples of our absorbers are connectable. Notice that 
the assumptions on $R$ are like clause~\ref{it:rc1} and the special case 
$\mu=\alpha/4$ of  
clause~\ref{it:rc2} of Proposition~\ref{prop:robust}.

\begin{lemma}\label{lem:L36}
	Given $\alpha>0$, let $G=(V,E)$ and $G'=(V,E')$ be two graphs on the same $n$-element 
	vertex set, each with at least $(5/9+\alpha)n^2/2$ edges. Let $R$ be a  
	subgraph of $G$ induced by a set $U=V(R)\subseteq V$ with 
	$|U|\ge 2n/3$ that satisfies 
	$e_{G}(U, V\setminus U)\le \alpha n^2/4$.
	Then
	\begin{equation}\label{eq:uzzy}
		\big|\{(u,v)\in U^2\colon uv\in E\cap E' 
		\tand d_R(v)>n/3\}\big|
		\ge
		\frac34\alpha n^2\,.
	\end{equation}
\end{lemma}

\begin{proof} Let $Z=\{z\in U\colon d_R(z)>n/3\}$.
	We shall show
	\begin{equation}\label{eq:xyxy}
		\big|\{xy\in E\cap E'\colon 
			x,y\in U\mbox{ and } \{x,y\}\cap Z\ne\emptyset\}\big|
		\ge
		\frac34\alpha n^2\,.
	\end{equation}
 	Since every (unordered) edge~$xy$ counted here corresponds to one 
	or two ordered pairs~$(u, z)$ counted on the left side of~\eqref{eq:uzzy}
	(depending on whether only one or both of $x$, $y$ are in~$Z$), this will
	imply the desired estimate~\eqref{eq:uzzy}. For the proof of~\eqref{eq:xyxy} 
	we let $\eta\in[2/3, 1]$
	and $\tau\in[0, 1]$ be defined by 
	\[
		|U|=\eta n\qqand|Z|=\tau n\,.
	\] 
	We consider two cases depending on the value of~$\tau$.

	\smallskip
	{\it First Case. We have $\tau\ge 2/3$.}\\
	Owing to
	\[
		\big|\{(x,y)\in V^2\colon xy\in E \cap E'\}\big|
		\ge
		2\,|E|+2\,|E'|-n^2
		\ge
		\big(\tfrac{1}{9}+2\alpha\big)n^2
	\]
	and $|V\setminus Z|^2=(1-\tau)^2n^2\le n^2/9$ we have
	\[
		\big|\{(x,y)\in V^2\colon xy\in E\cap E' \mbox{ and }
			\{x,y\}\cap Z\neq\emptyset\}\big|
		\ge
		2\alpha n^2\,.
	\]
	Recall that $Z\subseteq U$. So if $\{x,y\}\cap Z\neq\emptyset$, 
	but $\{x,y\}\not\subseteq U$, then one of the vertices $x$, $y$ is in $U$ while the other 
	one is in $V\setminus U$, whence
	\[
		\big|\{(x,y)\in V^2\colon xy\in E\,,
			\{x,y\}\cap Z\neq\emptyset, \tand
			\{x,y\}\not\subseteq U\}\big|
		\le
		2\,e_{G}(U, V\setminus U)
		\le \frac{\alpha}{2} n^2\,.
	\]
	Consequently, the number of (unordered) edges $xy$ considered on the 
	left-hand side of~\eqref{eq:xyxy} is at least 
	$\frac12(2\alpha n^2-\alpha n^2/2)=3\alpha n^2/4$, as desired. 

	\smallskip
	{\it Second Case. We have $\tau < 2/3$.}\\
	Notice that 
	\[
		2\,e(R)
		\ge
		2\big(e(G)-e_{G}(U, V\setminus U)-e_{G}(V\setminus U)\big)
		\ge
		\Big(\frac{5}{9}+\alpha-\frac\alpha2-(1-\eta)^2\Big)n^2\,.
	\]
	Together with $(1-\eta)(2/3-\eta)\le 0$ this yields
	\begin{equation}\label{eq:1541}
		2\,e(R)
		\ge
		\Big(\frac{5}{9}+\frac\alpha2-(1-\eta)^2+(1-\eta)\Big(\frac{2}{3}-\eta\Big)\Big)n^2
		=
		\Big(\frac{2}{9}+\frac{\alpha}{2}+\frac{\eta}{3}\Big)n^2\,.
	\end{equation}
	On the other hand, the definition of $Z$ leads to
	\begin{equation}\label{eq:1544}
		2\,e(R)
		=
		\sum_{z\in Z}d_R(z)+\sum_{z\in U\setminus Z}d_R(z)
		\le
		\sum_{z\in Z}d_R(z)+\frac{(\eta-\tau)}3n^2\,.
	\end{equation}
	Comparing~\eqref{eq:1541} and~\eqref{eq:1544} we deduce
	\[
		\sum_{z\in Z}d_R(z)
		\ge
		\Big(\frac{2}{9}+\frac{\alpha}{2}+\frac{\eta}{3}-\frac{\eta-\tau}{3}\Big)n^2
		=
		\Big(\frac{2}{9}+\frac{\alpha}{2}+\frac{\tau}{3}\Big)n^2\,.
	\]
	By the assumption of the case we have 
	\[
		\frac{\tau}{3} n^2
		>
		\frac{\tau^2}{2} n^2
		\ge
		\binom{|Z|}2
	\]
	and this shows that 
	\begin{equation*}\label{eq:1550}
		\sum_{z\in Z}d_R(z)
		\ge
		\Big(\frac{2}{9}+\frac{\alpha}{2}\Big)n^2+\binom{|Z|}2\,,
	\end{equation*}
	which in turn implies 
	\[
		\big|\{xy\in E\colon x, y\in U\tand \{x,y\}\cap Z\neq\emptyset\}\big|
		\ge
		\sum_{z\in Z}d_R(z)-\binom{|Z|}{2}
		\ge
		\Big(\frac{2}{9}+\frac{\alpha}{2}\Big)n^2\,.
	\]
	Finally, the sieve formula yields
	\begin{multline*}
		\big|\{xy\in E\cap E'\colon x, y\in U\tand \{x,y\}\cap Z\neq\emptyset\}\big|
		\ge
		\Big(\frac49+\alpha\Big)\frac{n^2}{2}+\Big(\frac59+\alpha\Big)\frac{n^2}{2}-\binom{n}{2}
		\ge
		\alpha n^2\,,
	\end{multline*}
	which is more than what we need for establishing~\eqref{eq:xyxy}. 
\end{proof}

\subsection{Connectable pairs and bridges in 3-uniform hypergraphs} 
\label{subsec:2341}

In this subsection we discuss the $3$-uniform connecting lemma from~\cite{old}
together with some related results. Roughly speaking, this lemma asserts that 
in any sufficiently large $3$-uniform hypergraph~$H=(V,E)$ 
with $\delta_1(H)\ge (5/9+\alpha)|V|^2/2$ any two pairs of vertices possessing 
a special property called connectability can be connected by  many short 
paths. 
The definition of our connectability notion presupposes that for every 
vertex $v\in V(H)$, one has fixed a robust subgraph of its link graph 
as obtained by Proposition~\ref{prop:robust}. We collect these assumptions 
in the following setup.

\begin{setup}\label{setup:1746}
	Suppose that $\alpha\in (0, 1/3)$, that $\mu$, $\beta>0$, that $\ell\ge 3$ is an odd integer, 
	that $H=(V, E)$ is 
	a sufficiently large $3$-uniform hypergraph 
	with $\delta_1(H)\ge (5/9+\alpha)|V|^2/2$, and that for every vertex $v\in V$,
	Proposition~\ref{prop:robust} located a 
	$(\beta, \ell)$-robust induced subgraph $R_v\subseteq H_v$ of its 
	link graph satisfying
		\begin{enumerate}[label=\rmlabel]
			\item\label{it:1652a} $|V(R_v)|\ge \bigl(\frac23+\frac\alpha 2\bigr)|V|$,
			\item\label{it:1652b} $e_{H_v}\big(V(R_v),V\setminus V(R_v)\big)\le 
				\mu |V|^2$, 			
			\item\label{it:1652c} and $e(R_v)\ge 
				\left(\frac{5}{9}+\frac{\alpha}{2}\right)\frac{|V|^2}{2}-\frac{(|V|-|V(R_v)|)^2}{2}
				\ge\left(\frac49+\frac23\alpha\right)\frac{|V|^2}{2}$.
		\end{enumerate}
\end{setup}

We remark that for most part of this section condition~\ref{it:1652b} with $\mu=\alpha/4$ of this setup suffices.
In fact, the results in~\cite{old} were obtained for this restricted version of the setup 
and below we (mostly) recapitulate and apply them in this form. A stronger form, with 
a smaller value of $\mu$, of 
Proposition~\ref{prop:robust}\,\ref{it:rc2} will be useful in Lemma~\ref{L35} below and for 
the construction of the absorbing path in Section~\ref{sec:Abpa}.
The following notion of connectable pairs is taken
from~\cite{old}*{Definition 2.5}.

\begin{dfn}
	Given Setup~\ref{setup:1746} and $\zeta>0$, an unordered pair $xy$ of distinct
	vertices of~$H$ is said to be {\it $\zeta$-connectable} if the set
	\[
		U_{xy}=\{v\in V\colon xy\in E(R_v)\}
	\]
	satisfies $|U_{xy}|\ge \zeta |V|$. An ordered pair $(x, y)$ 
	is {\it $\zeta$-connectable} if its underlying unordered pair $xy$ is. 
\end{dfn}

We are now ready to state the $3$-uniform connecting lemma from~\cite{old}*{Proposition 2.6}.
\begin{prop}[Connecting lemma for $3$-uniform hypergraphs] \label{lem:con3} 
	Given Setup~\ref{setup:1746} (with $\mu=\alpha/4$) and $\zeta>0$, there 
	exists $\theta=\theta(\alpha, \beta, \ell, \zeta)>0$
	such that the following holds.
	
	If $(a, b)$, $(x, y)$ are two disjoint $\zeta$-connectable 
	pairs of vertices of $H$, then the number of $ab$-$xy$-paths in $H$ with 
	$3\ell+1$ inner vertices is at least $\theta|V|^{3\ell+1}$. \qed 
\end{prop}

For later use we also state the following simple fact (see~\cite{old}*{Fact 4.1}).

\begin{lemma}\label{F41}
	Given Setup~\ref{setup:1746} (with $\mu=\alpha/4$) and $\zeta>0$, there are at most $\zeta |V|^3$ 
	triples $(x,y,z)\in V^3$ such that $xy\in E(R_z)$, but $xy$ 
	is not $\zeta$-connectable in $H$. \qed
\end{lemma}

To extend these notions and results to $4$-uniform hypergraphs we need a new 
$3$-uniform concept.

\begin{dfn}\label{def:bridge3}
	Given Setup~\ref{setup:1746} and $\zeta>0$, a triple $(x,y,z)\in V^3$ is called 
	a {\it $\zeta$-bridge in~$H$} if $xyz\in E$ and $xy$ and $yz$ are 
	both $\zeta$-connectable in $H$. We say a path $x_1x_2\dots x_{j-1}x_j$
	starts (resp.\ ends) with a $\zeta$-bridge, if $x_1x_2x_3$ (resp.\ $x_{j-2}x_{j-1}x_j$)
	is a $\zeta$-bridge.
\end{dfn}

It will be useful to estimate the number of bridges in a dense 3-uniform hypergraph.

\begin{lemma}\label{NB3} 
	Given Setup~\ref{setup:1746} (with $\mu=\alpha/4$) and $\zeta>0$, the number of 
	triples $(x,y,z)\in V^3$ with $xyz\in E$ 
	that fail to be a $\zeta$-bridge in $H$ is at most 
	$(2/9+\alpha/2+2\zeta)|V|^3$.
	In particular, if $\zeta<\alpha/4$, then
	there are more than $|V|^3/3$ $\zeta$-bridges in $H$.
\end{lemma}

\begin{proof}
	Starting with $A=\{(x,y,z)\in V^3\colon xyz \in E\}$ we note that the minimum degree 
	assumption yields $|A|\ge (5/9+\alpha)|V|^3$.
	We consider four exceptional subsets of $A$, namely 
	\begin{align*}
		P_1&=\{(x,y,z)\in A\colon xy\not\in E(R_z)\}\,,&
		Q_1&=\{(x,y,z)\in A\setminus P_1\colon \text{ 
		$xy$ is not $\zeta$-connectable}\}\,,\\
		P_2&=\{(x,y,z)\in A\colon yz\not\in E(R_x)\}\,,& 
		Q_2&=\{(x,y,z)\in A\setminus P_2\colon \text{
		$yz$ is not $\zeta$-connectable}\}\,.
	\end{align*}
	Notice that every triple in $A\setminus(P_1\cup Q_1\cup P_2\cup Q_2)$ is a $\zeta$-bridge 
	in $H$. 
	Lemma~\ref{F41} yields the upper bounds $|Q_1|$, $|Q_2|\le \zeta |V|^3$.
	Moreover, the first two clauses of Setup~\ref{setup:1746} 
	and~$\alpha\leq 1/3$ lead to
	\begin{align*}
		|P_1|
		&\le 
		\sum_{z\in V}\Big(2\,e_{H_z}(V(R_z), V\setminus V(R_z))+(n-|V(R_z)|)^2\Big) \\
		&\le 
		\bigg(\frac{\alpha}{2}+\Big(\frac{1}{3}-\frac{\alpha}{2}\Big)^2\bigg)|V|^3\\
		&\le 
		\Big(\frac{1}{9}+\frac{\alpha}{4}\Big)|V|^3\,.
	\end{align*}
	The same upper bound applies to $|P_2|$. These upper bounds on $|Q_1|$, $|Q_2|$
	and $|P_1|$, $|P_2|$ yield the desired upper bound on 
	$|P_1\cup Q_1\cup P_2\cup Q_2|$ for the first part of the lemma.
	
	The second part is a direct consequence, since $A\setminus (P_1\cup Q_1\cup P_2\cup Q_2)$
	is a subset of all $\zeta$-bridges in $H$ and 
	\[
		\big|A\setminus (P_1\cup Q_1\cup P_2\cup Q_2)\big|
		\geq 
		\Big(\frac{5}{9}+\alpha\Big)|V^3|-\Big(\frac{2}{9}+\frac{\alpha}{2}+2\zeta\Big)|V^3|
		>
		\frac{|V|^3}{3}
	\]
	as long as $\zeta<\alpha/4$.
\end{proof}

The next lemma implies that every two $3$-uniform hypergraphs $H$ and $H'$ on the same vertex set $V$
with minimum vertex degree $(5/9+o(1))|V|^2/2$ have the property that~$H'$ contains many 
bridges of~$H$ as edges. (For technical reasons it will be convenient to allow that the vertex sets of $H$ and $H'$ differ slightly.) Note that the lower bound on the number of bridges 
in Lemma~\ref{NB3} falls short of implying such an assertion. In fact, the proof of the following lemma will
rely on the structural properties of hypergraphs and bridges. 

\begin{lemma}\label{L35}
Given Setup~\ref{setup:1746} with $\mu=\frac{\alpha^3}{18}$ for a $3$-uniform hypergraph $H=(V,E)$ with $|V|=n$, 
let $\zeta\in(0,\alpha^2/9)$, and let $H'=(V',E')$ 
be a $3$-uniform hypergraph with $\delta_1(H')\geq (5/9+\alpha)n^2/2$ and 
$|V\triangle V'|\leq \alpha n/18$.
Then the number of $\zeta$-bridges $(x,y,z)\in V^3$ in $H$ such that $xyz\in E'$ 
is at least $\alpha n^3/2$.
\end{lemma}

\begin{proof} Let $H=(V,E)$ and $H'=(V',E')$ satisfy the assumptions of the lemma. 
In particular, for every vertex $v\in V$ we fixed a robust subgraph $R_v\subseteq H_v$. 
We consider the following set of triples
\[
	T
	=
	\big\{(x,y,z)\in V^3\colon 
		xy\in E(H_z)\cap E(H'_z)\,,\
		x,y\in V(R_{z})\,,
		\tand
		d_{R_z}(x) > (\tfrac{1}{3}-\tfrac{\alpha}{54})n\big\}\,.
\]
We shall appeal to Lemma~\ref{lem:L36} for a lower bound on $|T|$. For that we have to restrict to the subhypergraphs and 
subgraphs induced on $W=V\cap V'$. We
consider 
\begin{multline*}
	T[W]
	=
	\big\{(x,y,z)\in W^3\colon 
		xy\in E(H_z[W])\cap E(H'_z[W])\,,\\
		x,y\in V(R_{z})\cap W\,, 
		\tand
		d_{R_z[W]}(x) > |W|/3\big\}\,.
\end{multline*}
Note that the bound on the symmetric difference $V\triangle V'$ 
guarantees a minimum vertex degree of at least $(5/9+8\alpha/9)n^2/2$ for $H[W]$ and $H'[W]$. Moreover,
for every $z\in W$
we have
\[
	|V(R_z)\cap W|
	\geq
	\Big(\frac{2}{3}+\frac{\alpha}{2}\Big)n-\frac{\alpha}{18}n
	\geq 
	\frac{2}{3}|W|
\]
and
\[
	e_{H_z}\big(V(R_z)\cap W, W\setminus V(R_z)\big)
	\leq 
	e_{H_z}\big(V(R_z), V\setminus V(R_z)\big)
	\leq 
	\mu n^2
	=
	\frac{\alpha^3}{18}n^2
	\leq 
	\frac{\alpha}{4} |W|^2\,.
\]
Consequently, for every $z\in W$ we can apply Lemma~\ref{lem:L36} to
 $G=H_{z}[W]$, $G'=H'_{z}[W]$, and $R=R_z[W]\subseteq H_z[W]$
and~\eqref{eq:uzzy} tells us that
\begin{equation*}
	\big|T[W]\big|
	\ge 
	\frac34\cdot\frac{8}{9}\alpha |W|^3
	\geq
	\frac34\cdot\frac{8}{9}\alpha \Big(1-\frac{\alpha}{18}\Big)^3n^3
	\geq
	\frac{5}{8}\alpha n^3\,.
\end{equation*}
The definitions of $T$ and $T[W]$ imply $T[W]\subseteq T$ and, hence, we arrive at
\begin{equation}\label{eq:35T} 
	|T|\geq \frac{5}{8}\alpha n^3\,.
\end{equation}
We shall bound the sizes of the following `bad' subsets of~$T$
\begin{align*}
S_1&=\{(x,y,z)\in T\colon \text{$xy$ is not a $\zeta$-connectable pair in~$H$}\}\\
\tand\quad 
S_2&=\{(x,y,z)\in T\colon \text{$yz$ is not a $\zeta$-connectable pair in~$H$}\}\,.
\end{align*}
By definition of $T$, every triple $(x,y,z)\in T$ corresponds to an edge in $E\cap E'$ and by 
definition of $S_1$ and $S_2$, every triple in $T\setminus(S_1\cup S_2)$ is a $\zeta$-bridge in~$H$. 
Hence in view of~\eqref{eq:35T}, the conclusion of Lemma~\ref{L35} will follow from the estimates
\begin{equation}\label{eq:35S12}
	|S_1|\leq \zeta n^3
	\qqand
	|S_2|\leq \Big(\zeta+\frac{\alpha^3}{18}+\frac{\alpha^2}{8}\Big)n^3
\end{equation}
combined with $\alpha<1/3$ (cf.\ Setup~\ref{setup:1746}) and $\zeta<\alpha^2/9$.

The desired upper bound on the size of $S_1$ is a direct consequence of Lemma~\ref{F41}. In fact 
by definition of $T$, for every $(x,y,z)\in T$ we have $xy\in E(H_z)$ and  $x$, $y\in V(R_z)$. Since 
$R_z\subseteq H_z$ is an induced subgraph, it follows that $xy\in E(R_z)$ and Lemma~\ref{F41} applies.

In order to prove the second inequality of~\eqref{eq:35S12} we note that 
$xy\in E(H_z)$ is equivalent to $yz\in E(H_x)$ and thus we can apply the same argument as above to the subset 
\[
	S'_2=\{(x,y,z)\in S_2\colon y\in V(R_x)\tand z\in V(R_x)\}
\]
and Lemma~\ref{F41} tells us $S'_2\leq \zeta|V|^3$. Next we bound the size of $S_2\setminus S'_2$ by splitting it into the sets
\[
	S''_2=\{(x,y,z)\in S_2\colon y\not\in V(R_x)\tand z\in V(R_x)\}
	\ \tand\
	S'''_2=\{(x,y,z)\in S_2\colon z\not\in V(R_x)\}\,.
\]
Summarising the discussion above, we note that the proof of~\eqref{eq:35S12} reduces to showing that
\begin{equation}\label{eq:35S2p}
	\big|S''_2\big|\leq \frac{\alpha^3}{18}n^3\qqand
	\big|S'''_2\big|\leq \frac{\alpha^2}{8}n^3\,.
\end{equation}
For the bound on $|S''_2|$ we appeal for every $x\in V$ 
to part~\ref{it:1652b} of Setup~\ref{setup:1746} for~$R_x\subseteq H_x$. For every vertex $x\in V$ there are at most 
$\mu n^2=\alpha^3n^2/18$ pairs $(y,z)\in (V\setminus V(R_x))\times V(R_x)$ with $yz\in E(H_x)$. Since the definition of~$T\supseteq S''_2$ ensures $xy\in E(H_z)$ and, hence, 
by symmetry also  $yz\in E(H_x)$, the desired bound on $|S''_2|$ stated in~\eqref{eq:35S2p} follows.

For the bound on $|S'''_2|$ we consider the set of pairs  
\[
	P=\{(x,z)\in V^2
		\colon d_{R_z}(x)>(\tfrac{1}{3}-\tfrac{\alpha}{54})n\tand z\not\in V(R_x)\}
\]
and we observe that the definitions of $T$ and $S'''_2$ yield
\[
	\big|S'''_2\big|\leq |P|\cdot n\,.
\]
For the bound on $P$ we consider an arbitrary vertex $x\in V$. Since 
\[
	d_{R_z}(x)\leq d_{H_z}(x)=d_H(x,z)=d_{H_x}(z)
\]
we are interested in the number of vertices $z\not\in V(R_x)$ with $d_{H_x}(z)>(\tfrac{1}{3}-\tfrac{\alpha}{54})n$.
Owing to part~\ref{it:1652a} of Setup~\ref{setup:1746} for~$R_x\subseteq H_x$ there are at least 
\[
	d_{H_x}(z)-\big|V\setminus V(R_x)\big|
	>
	\Big(\frac{1}{2}-\frac{1}{54}\Big)\alpha n
	=
	\frac{26}{54}\alpha n
\]
edges of $E_{H_x}\big(V(R_x),V\setminus V(R_x)\big)$ incident to such a vertex~$z$. Therefore, part~\ref{it:1652b} 
of Setup~\ref{setup:1746} implies that for every fixed~$x$ there are at most 
\[
	\frac{e_{H_x}\big(V(R_x),V\setminus V(R_x)\big)}{26\alpha n/54}
	\leq
	\frac{\alpha^3n^2/18}{26\alpha n/54}
	\leq 
	\frac{\alpha^2}{8}n
\]
choices of~$z$. Consequently, $|P|\leq \alpha^2 n^2/8$ and the bound on $|S'''_2|$ from~\eqref{eq:35S2p} follows.
This concludes the proof of Lemma~\ref{L35}.
\end{proof}

An interesting feature of Proposition~\ref{lem:con3} caused by the proof strategy
pursued in~\cite{old} is that the number of inner vertices in the connections
it provides is necessarily congruent to~$1$ modulo $3$. In \S\ref{subsec:1727}
below it will be convenient to employ connections whose numbers of inner vertices are 
in other residue classes modulo $3$. As the following result shows, such 
connections can be accomplished by going ``via bridges''.

\begin{cor}\label{all3}
	Given Setup~\ref{setup:1746} (with $\mu=\alpha/4$) and $\zeta>0$, there 
	exist three  integers $\ell_1, \ell_2, \ell_3\le 12\ell$
	with~$\ell_i\equiv i\pmod{3}$ for all $i\in [3]$ and $\theta=\theta(\alpha,\beta,\l,\zeta)>0$
	such that the following holds.
	
	If $(a, b)$, $(x, y)$ are two disjoint $\zeta$-connectable 
	pairs of vertices of $H$, then for every~\mbox{$i\in[3]$}, the number of $ab$-$xy$-paths in $H$ 
	with $\ell_i$ inner vertices is 
	at least $\theta|V|^{\ell_i}$.
\end{cor}

\begin{proof}
	We set
	\[
		\l_1=3\l+1\,,\qquad
		\l_2=6\l+5\,,\qquad
		\l_3=9\l+9\,,\qqand
		\theta=\frac{\theta^3_1}{25}\,,
	\]
	where $\theta_1$ is provided by  Proposition~\ref{lem:con3}. Since $\l\geq 3$, we have $\l_1\leq\l_2\leq\l_3\leq 12\l$. 
	We already know that $\ell_1$ has the desired property by Proposition~\ref{lem:con3}
	and we shall verify the corollary for $\ell_2$ and $\ell_3$.
	
	Starting with the argument for $\ell_2$, we let any two disjoint $\zeta$-connectable
	pairs $(a, b)$ and $(x, y)$ be given. Notice that if $(u, v, w)$ is a $\zeta$-bridge,
	$abPuv$ is an $ab$-$uv$-path with $\ell_1$ inner vertices, and $vwQxy$ is a 
	$vw$-$xy$-path with $\ell_1$ inner vertices, then $abPuvwQxy$ is an $ab$-$xy$-walk
	with $\ell_1+3+\ell_1=\ell_2$ inner vertices. 

	By Lemma~\ref{NB3} for sufficiently large $|V|$
	there are $|V|^3/4$ possibilities to choose the bridge $(u, v, w)$ in such a way 
	that $\{a, b, x, y\}\cap\{u, v, w\}=\varnothing$ and for every such choice of 
	the bridge in the middle, Proposition~\ref{lem:con3} delivers $\theta_1|V|^{\ell_1}$
	possibilities for $P$ as well as~$\theta_1|V|^{\ell_1}$ possibilities for $Q$.
	So altogether the number of $ab$-$xy$-walks with $\ell_2$ inner vertices is at 
	least $\theta^2_1|V|^{\ell_2}/4$. Since at most~$O(|V|^{\ell_2-1})$ of them fail to 
	be paths due to containing the same vertex multiple times, this proves 
	that $\ell_2$ has the desired property for $\theta<\theta^2_1/5$ and sufficiently 
	large~$|V|$. 
	
	For $\l_3$ we can repeat the same argument once more and 
	get the same conclusion for the choice of $\theta=\theta^3_1/25$ above.
\end{proof}

\subsection{Path covers in 3-uniform hypergraphs}
\label{subsec:1727}

Preparing the proof of the $4$-uniform covering lemma in Section~\ref{sec:long_path}
we shall now prove the following $3$-uniform covering principle. 

\begin{prop}\label{prop:1742}
	For all $\alpha$, $\xi\in (0, 1/3)$ there is an infinite arithmetic 
	progression ${P\subseteq 3\NN}$ such that the following holds.
	
	Given Setup~\ref{setup:1746} (with $\mu=\alpha/4$), a collection $\ccB\subseteq V^3$ 
	of $\xi$-bridges in $H$ with $|\ccB|\ge \xi |V|^3$, and $M\in P$, we can cover all 
	but at most $\xi |V|+M$ 
	vertices of $H$ by vertex-disjoint paths of length $M$ each of which starts and ends
	with a bridge from $\ccB$.  
\end{prop} 

Let us remark that while the vertex set $V$ in this statement is assumed to be much larger
than $\alpha^{-1}$ and $\xi^{-1}$, the quantification in Proposition~\ref{prop:1742} allows
to consider $M$ to be a function of $|V|$. In the application
we have in mind,~$M$ will be about $\Theta(\sqrt{|V|})$. 
Before we come to the proof of Proposition~\ref{prop:1742} itself, we would like to give 
a brief overview. First of all, in~\cite{old} we proved (somewhat implicitly) a similar 
result, where $M$ is a constant and~$|V|$ is very large. Moreover, there everything related 
to $\ccB$ is omitted, but instead of this one can 
demand that the end-pairs of the constructed paths should be $\zeta_{\star\star}$-connectable
for a sufficiently small constant $\zeta_{\star\star}\ll \alpha, \xi$ (see Lemma~\ref{lem:2140}
below). The idea here for obtaining longer paths (say of length $\sqrt{|V|}$) is that in the 
beginning of the proof we put a small reservoir 
set aside, so that in the end we can connect many short paths into a smaller number of longer 
ones. To this end we require a somewhat standard reservoir lemma (see Lemma~\ref{lem:1743}
and Figure~\ref{fig:res}). The length of the longer paths we obtain in this manner 
depends linearly on the number
of short paths we connect, and hence the possible such lengths form an arithmetic progression 
$P$. Now we still need to ensure that the paths we construct start and end with bridges 
from $\ccB$. This is achieved by putting sufficiently many such bridges aside that are 
vertex-disjoint among themselves and to the reservoir. At the end of the proof we will then
be able to connect the selected bridges to our paths by making further uses of the reservoir.

\begin{figure}[t]
	\begin{tikzpicture}[scale=0.9]
	
	\coordinate (a11) at (0,0);
	\coordinate (a12) at (0,0);
	\coordinate (a13) at (0,0);
	
	\foreach \i in {0, ..., 5}{
		\foreach \j in {0,1,2}{
			\coordinate (a\i\j) at (\j,\i);
			\fill (a\i\j) circle (2pt);
		}}
	
	\foreach \i in {0, ..., 5}{
		\qedge{(a\i0)}{(a\i1)}{(a\i2)}{7pt}{1.5pt}{red!70!black}{red!50!white,opacity=0};
		}
	
	\draw [ultra thick] (-1.5,-1.2) rectangle (13.5,8.5);
	\draw [ultra thick] (-.5,-.7) rectangle (2.5,5.7);
	\draw [ultra thick] (-.5,6.2) rectangle (12.5,8);
	\draw [ultra thick] (3,4.2) rectangle (12.5,5.2);
	\draw [ultra thick] (3,2.5) rectangle (12.5,3.5);
	\draw [ultra thick] (3,-.7) rectangle (10,.3);
	
	\node [rotate=90] at (8,1.5){{\Huge $\dots$}};
	\node at (-1,4){{\Huge $X$}};
	\node at (.5,7.1){{\Huge $\mathcal{R}$}};
	\node at (11,-.3){{\Large $\mathscr{C}_{\lambda+1}$}};
	\node at (13,2.9){{\Large $\mathscr{C}_{2}$}};
	\node at (13,4.6){{\Large $\mathscr{C}_{1}$}};
	\node at (14.1,4.2){{\Huge $V$}};
		
	\coordinate (a) at (3.2,-.2);
	
	\foreach \i in {0,1,2,3}{
		\coordinate (a0\i) at ($(a)+(\i*2.5,0)$);
		\coordinate (a1\i) at ($(a)+(\i*2.5,3.2)$);
		\coordinate (a2\i) at ($(a)+(\i*2.5,4.9)$);
		\draw [line width=3pt, color=red!70!black] (a1\i) -- ($(a1\i)+(1.6,0)$);
		\draw [line width=3pt, color=red!70!black] (a2\i) -- ($(a2\i)+(1.6,0)$);
		}
	
	\foreach \i in {0,1,2}{
		\draw [line width=3pt, color=red!70!black] (a0\i) -- ($(a0\i)+(1.6,0)$);
	}
	
	\draw [line width= 3pt, color = red!70!black] ($(a20)+(1.55,0)$) to [out = 90, in = 180]
	 +(.5,2.7)to [out = 0, in = 90] ($(a21)+(.05,0)$);
	
	\draw [line width= 3pt, color = red!70!black] ($(a21)+(1.55,0)$) to [out = 90, in = 180]
	+(.5,2.7)to [out = 0, in = 90] ($(a22)+(.055,0)$);
	
	\draw [line width= 3pt, color = red!70!black] ($(a22)+(1.55,0)$) to [out = 90, in = 180]
	+(.5,2.7)to [out = 0, in = 90] ($(a23)+(.055,0)$);

	\end{tikzpicture}
	\caption {Collections of small paths, reservoir $\mathcal{R}$, and 
		some bridges from~$\mathscr{B}$ form the set $X$.}
	\label{fig:res}
\end{figure}
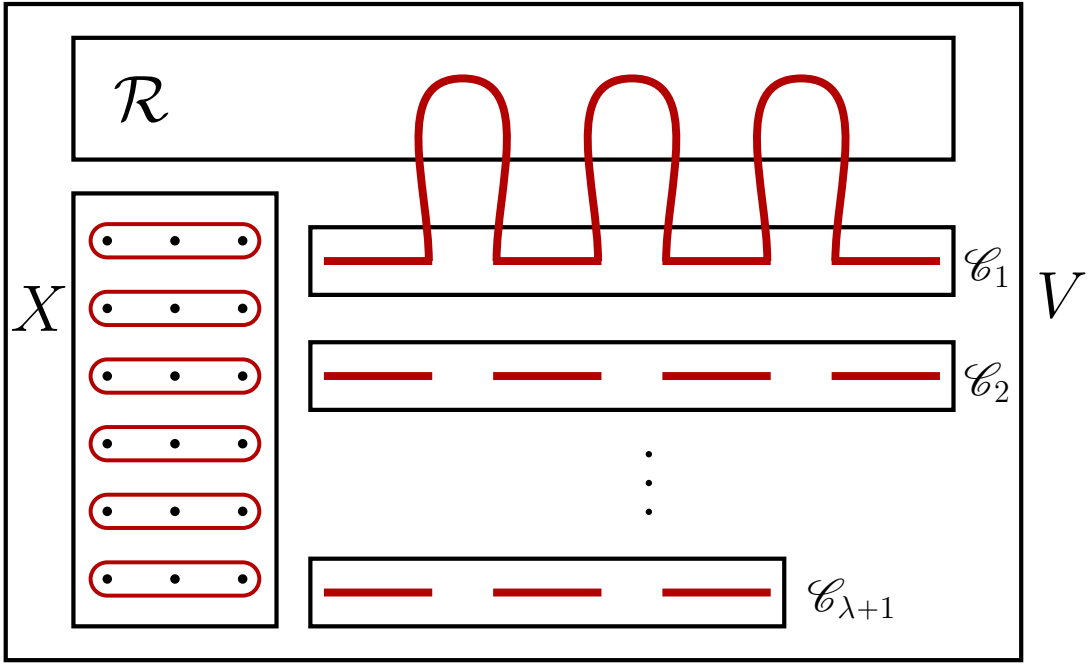

\begin{lemma}\label{lem:1743}
	For all $\alpha\in (0, 1/3)$ and $\theta_\star$, $\zeta_{\star\star}>0$ there 
	exists $\theta_{\star\star}>0$ with the following property.
	
	Given Setup~\ref{setup:1746} (with $\mu=\alpha/4$) and the 
	integers $\ell_1, \ell_2, \ell_3\le 12\ell$ 
	provided by Corollary~\ref{all3}, there is a reservoir set $\cR\subseteq V$ with 
	$\frac 12 \theta_\star^2 |V| \le |\cR|\le \theta_\star^2 |V|$ such that for 
	every $\cR'\subseteq \cR$ with $|\cR'|\le \theta_{\star\star}|\cR|$, 
	every $i\in [3]$, and any two 
	disjoint $\zeta_{\star\star}$-connectable pairs $(x, y)$ and $(z, w)$, there is 
	a $xy$-$zw$-path with $\ell_i$ inner vertices all of which belong to $\cR\setminus\cR'$.
\end{lemma}

\begin{proof}
	Without loss of generality we may assume $\zeta_{\star\star}<1/4$. We 
	fix an auxiliary constant $\eta$ and choose $\theta_{\star\star}$ appropriately
	to obey the hierarchy  
	$\theta_{\star\star}\ll \eta \ll \theta_\star, \zeta_{\star\star}, \alpha$. 
	Consider a random subset $\cR\subseteq V$ including every vertex 
	$v\in V$ independently with probability $\frac 34\theta_\star^2$. It follows from 
	Corollary~\ref{all3}
	along the lines of  the proof of~\cite{old}*{Proposition~2.7} that such a set a.a.s.\ 
	has the desired size $\frac 12 \theta_\star^2 |V| \le |\cR|\le \theta_\star^2 |V|$
	and possesses the property that for all disjoint $\zeta_{\star\star}$-connectable 
	pairs $(x, y)$ and $(z, w)$ and all $i\in [3]$, the number of $xy$-$zw$-paths with 
	$\ell_i$ inner vertices all of which belong to $\cR$ is at least $\eta |\cR|^{\ell_i}$.
	Fix a reservoir set $\cR\subseteq V$ with this property. Now, if in addition to 
	the pairs $(x, y)$, $(z, w)$, and to $i\in [3]$, also a set $\cR'\subseteq \cR$ 
	with $|\cR'|\le \theta_{\star\star}|\cR|$ is given, we know that at most 
	$\ell_i |\cR'||\cR|^{\ell_i-1}\le 12\ell\theta_{\star\star} |\cR|^{\ell_i}
	<\eta |\cR|^{\ell_i}$ of these paths can contain an inner vertex from~$\cR'$, meaning that 
	the desired path with inner vertices only from $\cR\setminus \cR'$ exists.
\end{proof}

\begin{lemma}\label{lem:2140}
	For all $\alpha\in (0, 1/3)$ and $\theta_\star$ with $0\ll\theta_{\star}\ll\alpha$ 
	there is a $\zeta_{\star\star}\in (0, \theta_\star)$
	such that for every sufficiently large $M\in\NN$ with $M\equiv 2\pmod{3}$ 
	the following holds. 
	
	Given Setup~\ref{setup:1746} (with $\mu=\alpha/4$), a reservoir set $\cR\subseteq V$ 
	as provided by Lemma~\ref{lem:1743}, and a set $X\subseteq V\setminus \cR$ 
	with $|X|\le \theta_\star|V|$, 
	one can cover all but at most $2\theta_\star^2|V|$ vertices of~$H-(\cR\cup X)$ by disjoint 
	$M$-vertex paths whose end-pairs are $\zeta_{\star\star}$-connectable. 
\end{lemma}

\begin{proof}
	This is implicit in~\cite{old}*{Section~7}, where an almost spanning path in $H$
	is constructed that avoids the absorbing path. More precisely,~\cite{old}*{Lemma~7.1}
	asserts that for a certain set~$X$ called~$V(P_A)$ there, there is a path $Q\subseteq H-X$
	satisfying 
	\begin{equation}\label{eq:2058}
		|V(H)\setminus (\cR\cup X\cup V(Q))|\le \theta_\star^2 |V|
	\end{equation} 
	and two further statements that are immaterial for our present concerns. 
	The only property of $X$ used in the proof of~\cite{old}*{Lemma~7.1}
	is that it consists of no more than $\theta_\star |V|$ vertices and thus we can 
	repeat the entire proof with an arbitrary such set. In the beginning of the proof 
	we fixed a sufficiently large $M\in 3\NN+2$ and below we will assume, in particular, 
	that $M\ge \theta_\star^{-2}$.
	
	Next we recall that $Q$ is constructed so as to contain many subpaths belonging to 
	the set
	\[
		\ccP=\bigl\{P\subseteq H-(X\cup \cR) \colon P \text{ is an $M$-vertex path
			whose end-pairs are $\zeta_{\star\star}$-connectable}\bigr\}\,.
	\]
	In fact, there is a set $\ccC\subseteq \ccP$ of mutually vertex-disjoint paths
	such that $Q$ starts and ends with a path from $\ccP$ and between any two ``consecutive''
	members of $\ccC$ appearing in~$Q$ there is either at most one vertex or there are only 
	vertices from $\cR$ (cf.\ clauses (b) and (c) in the definition of candidates in the proof of Lemma~7.1 in~\cite{old}).
	This property of~$Q$ guarantees
	\[
		 \Big|V(Q)\setminus \bigr(\bigcup\nolimits_{P\in \ccC}V(P) \cup \cR\bigl)\Big|
		 \le 
		 \frac{|V|}{M}
		 \le 
		 \theta_\star^2 |V|\,,
	\]
	which combined with~\eqref{eq:2058} yields
	\begin{multline*}
		\Big|V(H)\setminus \Bigl(\cR \cup X \cup \bigcup\nolimits_{P\in \ccC}V(P)\Bigr)\Big|
		\le
		\big|V(H)\setminus (\cR\cup X\cup V(Q))\big| \\
		+
		\Big|V(Q)\setminus \Bigl(\bigcup\nolimits_{P\in \ccC}V(P)\cup \cR\Bigr)\Big| 
		\le
		2\theta_\star^2 |V| \,.
	\end{multline*}
	In other words, $\ccC$ is the desired collection of paths. 
\end{proof}

\begin{proof}[Proof of Proposition~\ref{prop:1742}]
	Given $\alpha$, $\xi\in(0,1/3)$ we apply Lemma~\ref{lem:2140} with $\alpha$ 
	and $\theta_\star\ll \xi, \alpha$ and 
	obtain $\zeta_{\star\star}\in (0, \theta_\star)$. With this value of $\zeta_{\star\star}$
	we appeal to Lemma~\ref{lem:1743}, thus getting some $\theta_{\star\star}>0$. Next we pick 
	some $M\gg \theta_\star^{-1}, \theta_{\star\star}^{-1}$ with $M\equiv 2\pmod{3}$ which is so 
	large that the conclusion of Lemma~\ref{lem:2140} holds. Finally we take 
	$n_0\gg M, \ell, \theta_{\star}^{-1}, \theta_{\star\star}^{-1}$ so large that 
	we can apply the Lemmata~\ref{lem:1743} and~\ref{lem:2140} when $|V|\ge n_0$. 
	
	We shall prove that the infinite arithmetic progression  
	\[
		P=\bigl\{M'\in\NN\colon M'>n_0 \text{ and } M'\equiv 9\ell+15\pmod{M+1+3\ell}\bigr\}
	\]
	has the desired property. Since $9\ell+15$ and $M+1+3\ell$ are divisible by $3$, so 
	are all members of $P$. Now let Setup~\ref{setup:1746}, a collection $\ccB\subseteq V^3$
	of $\xi$-bridges with $|\ccB|\ge \xi |V|^3$ as well as a natural number $M'\in P$ be given. 
	We are to cover all but at most $\xi |V|+M'$ vertices of~$H$ by vertex-disjoint paths 
	consisting of $M'$ vertices which start and end with a $\xi$-bridge from $\ccB$. 
	If $|V|\le M'$ we can just take the empty collection of paths,
	so we may assume $|V| > M' > n_0$ from now on. Let $\cR\subseteq V$
	be a reservoir set as obtained from Lemma~\ref{lem:1743}. Consider a maximal 
	sequence $b_1, \ldots, b_r$ of $\xi$-bridges from $\ccB$ such that $\cR$ and these bridges 
	are mutually vertex-disjoint. Since the reservoir and the selected bridges 
	together involve $|\cR|+3r$ vertices, we have $3(|\cR|+3r)|V|^2\ge |\ccB|\ge \xi |V|^3$, 
	whence 
	\[
		r\ge \frac{\xi |V|-3|\cR|}{9} \ge \frac{(\xi-3\theta_\star^2)|V|}{9}>
		\theta_\star |V|\,.
	\]
	In particular, we can choose $x=\lfloor \theta_\star |V|/3\rfloor$ bridges in $\ccB$
	that are vertex-disjoint both from each other and from the reservoir. Define $X\subseteq V$
	to be the set of the $3x$ vertices occurring in such a list of $\xi$-bridges. 
	
	By Lemma~\ref{lem:2140} there is a collection~$\ccC$ of disjoint $M$-vertex paths in $H-(X\cup \cR)$
	covering all but at most $2\theta_\star^2 |V|$ vertices of $V(H)\setminus (X\cup \cR)$
	which start and end 
	with $\zeta_{\star\star}$-connectable pairs. Due to $M'> n_0\gg M, \ell$ the natural 
	number $k$ defined by 
	\[
		M'=(M+1+3\ell)k+(9\ell+15)
	\]
	satisfies $k\ge \sqrt{n_0}$. Take an arbitrary partition 
	\[
		\ccC=\ccC_1\dcup\ldots\dcup \ccC_\lambda\dcup \ccC_{\lambda+1}
	\]
	such that $|\ccC_1|=\ldots=|\ccC_\lambda|=k>|\ccC_{\lambda+1}|$. 
	For every $j\in [\lambda]$, we want to connect the $k$ paths in $\ccC_j$ by means 
	of $k-1$ connections through the reservoir to a path $P_j$. For each of these connections,
	we want to use $3\ell+1$ vertices from $\cR$, so we will have 
	\[
		v(P_j)=kM+(k-1)(3\ell+1)=M'-(12\ell+16)
	\]
	for every $j\in [\lambda]$. Altogether, these connections require at most
	\[
		(3\ell+1)|\ccC|\le (3\ell+1)\frac{|V|}{M}\le \frac{\theta_{\star\star}}{2}|\cR|
	\]
	vertices from the reservoir, so there is no problem in choosing them one by one. 
	
	Our strategy to continue is that for every $j\in [\lambda]$ we want to connect 
	the ends of the path $P_j$ to two of the $\xi$-bridges that have been put aside 
	into the set $X$. These connections are to be made through the reservoir and 
	for one of them we want to use $3\ell+1$ inner vertices, while the other one
	is supposed to use $\ell_3=9\ell+9$ inner vertices. Thereby each path $P_j$
	gets extended to a path $Q_j$ with
	\[
		v(Q_j)=v(P_j)+(9\ell+9)+(3\ell+1)+6=M'\,.
	\]
	There are indeed sufficiently many bridges contributing to $X$ for this plan, because
	\[
		2\lambda 
		\le 
		\frac{2|V|}k 
		\le 
		\frac{2|V|}{\sqrt{n_0}}
		\le 
		\frac{\theta_\star |V|}4
		<
		x\,.
	\]
	In fact, we even have 
	\[
		(12\ell+10)\lambda 
		\le 
		\frac{(12\ell+10)|V|}k 
		\le 
		\frac{|V|}{\sqrt[3]{n_0}}
		\le 
		\frac{\theta_{\star\star}|\cR|}2\,,
	\]
	which shows that the reservoir stays sufficiently intact while we are constructing 
	the paths $Q_1, \ldots, Q_\lambda$. Finally, the number of vertices that these 
	paths fail to cover is at most 
	\begin{align*}
		\Big|V\setminus \bigcup\nolimits_{P\in \ccC}V(P)\Big|
		+
		\Big|\bigcup\nolimits_{P\in \ccC_{\lambda+1}} V(P)\Big|
		& \le
		|\cR|+|X|+2\theta_\star^2 |V| + M' \\
		& \le 
		(3\theta_\star^2+\theta_\star) |V|+ M'
		\le
		\xi |V| +M'\,. \qedhere
	\end{align*}
\end{proof}

\section{Connecting lemma}\label{conn}
In this section we establish appropriate extensions 
of Proposition~\ref{lem:con3} and Corollary~\ref{all3} for $4$-uniform hypergraphs 
(see Proposition~\ref{lem:con} and Corollary~\ref{all4} below). In particular,
from now on $H$ is a 4-uniform hypergraph.

\subsection{Connectable triples in 4-uniform hypergraphs}
\label{subsec:conn_trip_4-unif}
Given a $4$-uniform hypergraph $H=(V,E)$ with minimum pair degree 
$\delta_2(H)\geq (5/9+\alpha)|V|^2/2$ we observe that the link~$H_{uv}$ of a pair  
of vertices $u$, $v\in V$ is a graph with edge density at least $5/9+\alpha$. Consequently,
Proposition~\ref{prop:robust} provides the existence of a robust subgraph in every joint link 
and we collect this information in the following setup. 

\begin{setup}\label{setup:2335}
	Suppose that $\alpha\in(0,1/3)$, $\beta>0$, that $\ell\geq 3$ is an odd integer, that $H=(V, E)$
	is a sufficiently large $4$-uniform hypergraph with $|V|=n$ 
	and $\delta_2(H)\ge (5/9+\alpha)n^2/2$, and that for every $\{u,v\}\in V^{(2)}$ 
	we have fixed a $(\beta, \ell)$-robust 
	subgraph $R_{uv}\subseteq H_{uv}$ of its link graph given 
	by Proposition~\ref{prop:robust} applied with $\mu=\alpha^3/18$.
\end{setup}

Let us remark that in this situation the vertices $u$ and $v$ are isolated in $H_{uv}$,
for which reason they cannot belong to the robust subgraph $R_{uv}$. 
Similarly, the vertex $v$ is isolated in the ($3$-uniform) link hypergraph $H_v$. 
So to make the results of~\S\ref{subsec:2341} applicable 
it turns out to be more convenient to work with the $3$-uniform hypergraph 
\[
	\overline{H}_v=H_v-v
\]
obtained from $H_v$ by removing the vertex $v$. Clearly 
this hypergraph has $n-1$ vertices and it satisfies the minimum degree condition 
$\delta_1(\overline{H}_v)\ge (5/9+\alpha)n^2/2\ge (5/9+\alpha)|V(\overline{H}_v)|^2/2$.

Moreover, $\overline{H}_v$ together with the family of 
graphs 
\[
	\bigl\{R_{uv}\colon u\in V(\overline{H}_v)\bigr\}
\]
exemplifies Setup~\ref{setup:1746}. 
Thus, whenever Setup~\ref{setup:2335}, a constant $\zeta>0$,
and a vertex $v\in V$ are given, we can speak of $\zeta$-connectable pairs in 
$\overline{H}_v$ and the notion of a $\zeta$-bridge in $\overline{H}_v$ is defined
as well. 

We continue with the definition of connectable triples in $4$-uniform hypergraphs, which pivots
on bridges in the $3$-uniform links of vertices.

\begin{dfn}\label{def:connectable4}
	Given Setup~\ref{setup:2335} and $\zeta>0$, a triple $(x,y,z)\in V^3$ is 
	said to be {\it $\zeta$-connectable in $H$} if the set
	\[
		U_{xyz}=\{v\in V\colon (x,y,z) \mbox{ is a $\zeta$-bridge in }\overline{H}_v\}
	\]
	satisfies $|U_{xyz}|\ge \zeta|V|$.
\end{dfn}

In general, changing the ordering of $x$, $y$, and $z$ can affect whether a 
triple $(x, y, z)$ is $\zeta$-connectable. It is easy to see, however, that 
reversing the ordering cannot have such an effect, i.e., $(z, y, x)$ 
is $\zeta$-connectable if and only if $(x, y, z)$ is. 

\begin{prop}[Connecting lemma] \label{lem:con} 
	Given Setup~\ref{setup:2335} and $\zeta>0$, there is 
	$\theta>0$ such that
	if $(a,b,c)$ and $(x,y,z)$ are disjoint, $\zeta$-connectable triples in $H$, 
	then the number of $abc$-$xyz$-paths in $H$ with $8\ell+10$ inner vertices 
	is at least $\theta n^{8\ell+10}$.	
\end{prop}

\begin{proof}[Proof of Proposition \ref{lem:con}]
	By monotonicity we may suppose that $\zeta<\frac 1{48}$. 
	Let $\theta_3$ denote the constant obtained by applying Proposition~\ref{lem:con3}
	to $\alpha$, $\beta$, $\ell$, and $\zeta^3$. We shall prove 
	that 
	\begin{equation}\label{eq:1253}
		\theta=\frac12\zeta^{3\ell+6}\theta_3^{2\ell+4}
	\end{equation}
	has the desired property. To this end we fix two disjoint $\zeta$-connectable 
	triples $(a, b, c)$ and~$(x, y, z)$. Consider the set $T$ of all sequences
	\[
		(u, \seq p, \seq q,  \seq r, w)\in V^{6\l+8}
	\]
	with
	\[
		\seq{p}=(p_1, \ldots, p_{3\ell+1})\,,\quad
		\seq{q}=(q_1,q_2,q_3,q_4)\,,
		\qand 
		\seq{r}=(r_{3\ell+1}, \ldots, r_1)
	\]
	such that the following six conditions hold:
	\begin{enumerate}[label=\nlabel]
		\item\label{it:Q1} $u\neq w$ and $u\in U_{abc}$, $w\in U_{xyz}$,
		\item\label{it:Q2} $\seq q$ spans a walk of length $3$ 
			in the robust subgraph $R_{uw}$ of the link graph $H_{uw}$,
		\item\label{it:Q3} $q_1q_2$ is $\zeta^3$-connectable in $\overline{H}_u$,
		\item\label{it:Q4} $q_3q_4$ is $\zeta^3$-connectable in $\overline{H}_w$,
		\item\label{it:Q5} $(b, c, \seq{p},q_1,q_2)$ spans a  $3$-uniform  path of length $3\l+1$ 
			in the link $\overline{H}_u$,
		\item\label{it:Q6} $(q_3,q_4,\seq{r}, x, y)$ spans a  $3$-uniform  path of length $3\l+1$
			in the link $\overline{H}_w$.
	\end{enumerate}
	We establish the following lower bound on the size of set $T$ defined above. 
	\begin{claim}\label{Q} 
		We have $|T|\ge\zeta^3\theta_3^2 n^{6\ell+8}$.
	\end{claim}
	\begin{proof} 
		Our first step is to show that the set 
		\[
			S=\{(u,\seq q, w)\in V^6\colon u\neq w, u\in U_{abc}, w\in U_{xyz}, 
			\text{ and } \seq q \mbox{ spans a walk of length $3$ in } R_{uw}\}
		\]
		of all sextuples satisfying~\ref{it:Q1} and~\ref{it:Q2} satisfies 
		\begin{equation}\label{eq:1056}
			|S|\ge\tfrac1{12}\zeta^2n^6\,.
		\end{equation}
		In fact, in view of Definition \ref{def:connectable4} the $\zeta$-connectability
		of $(a, b, c)$ and $(x, y, z)$ ensures that there are $\zeta n\cdot(\zeta n-1)$ 
		possibilities to 
		choose the pair $(u,w)$. Thus for the proof of~\eqref{eq:1056} it suffices to show 
		that for every pair $(u, w)\in U_{abc}\times U_{xyz}$ the number of $3$-edge walks
		in~$R_{uw}$ is at least $cn^4$ for some $c>1/12$. A result of Blakley and 
		Roy~\cite{BR} (asserting the validity of Sidorenko's conjecture for paths)
		combined with Proposition~\ref{prop:robust}\,\ref{it:rc3}
		entails that the number of these walks is indeed at least 
		\[
			\frac{(2e(R_{uv}))^3}{v(R_{uv})^2}
			\ge
			\frac{(4n^2/9)^3}{n^2}
			=
			\frac{4^3}{9^3}n^4
			>
			\frac{n^4}{12}\,.
		\]
		Thereby~\eqref{eq:1056} is proved and we proceed by estimating the set 
		\[
			S^\star=\{(u,\seq q,w)\in S\colon  
			\text{$q_1q_2$ is $\zeta^3$-connectable in $\overline{H}_{u}$
			and $q_3q_4$ is $\zeta^3$-connectable in $\overline{H}_{w}$}\}
		\]
		of all sextuples satisfying~\ref{it:Q1}\,--\,\ref{it:Q4}. By two successive 
		applications of Lemma~\ref{F41} we shall show
		\begin{equation}\label{eq:1143}
			|S\setminus S^\star|\le 2\zeta^3 n^6\,.
		\end{equation}
		Indeed, for every fixed triple $(u, q_3, q_4)\in V^3$, Lemma~\ref{F41} 
		applied to $\overline{H}_u$ and $\zeta^3$ (in place 
		of~$H$ and $\zeta$) tells us that there are at most $\zeta^3 n^3$ 
		triples $(q_1, q_2, w)$ with $q_1q_2\in E(R_{uw})$ for which~$q_1q_2$ 
		fails to be $\zeta^3$-connectable in $\overline{H}_{u}$.
		Similarly, for every fixed triple $(q_1, q_2, w)\in V^3$, Lemma~\ref{F41} applied to 
		$\overline{H}_w$ and $\zeta^3$ tells us that there are at most $\zeta^3 n^3$ 
		triples $(u, q_3, q_4)$ with $q_3q_4\in E(R_{uw})$ for which~$q_3q_4$
		fails to be $\zeta^3$-connectable in~$\overline{H}_{w}$. So altogether we have 
		$|S\setminus S^\star|\le 2n^3\cdot \zeta^3 n^3$, which proves~\eqref{eq:1143}.
		
		As a direct consequence of~\eqref{eq:1056},~\eqref{eq:1143}, and $\zeta<\frac 1{48}$
		we obtain
		\begin{equation}\label{eq:1154}
			|S^\star|\ge \frac1{12}\zeta^2n^6-2\zeta^3 n^6\ge 2\zeta^3n^6\,.
		\end{equation}
		Now by Proposition~\ref{lem:con3} and the definition of $S^\star$, for every 
		sextuple $(u,\seq q, w)$ there are at least $\theta_3(n-1)^{3\ell+1}$ 
		sequences $\seq p$ as demanded by~\ref{it:Q5} and there is at 
		least the same number of sequences $\seq r$ as 
		required by~\ref{it:Q6}. Consequently, we have 
		\[
			|T|
			\ge 
			|S^\star|\left(\theta_3(n-1)^{3\ell+1}\right)^2
			\overset{\eqref{eq:1154}}{\ge}
			\zeta^3\theta_3^2n^{6\ell+8}
		\]
		for sufficiently large $n$ and this concludes the proof of Claim~\ref{Q}.
	\end{proof}

	Now consider an auxiliary 3-partite 3-uniform hypergraph $A$ with vertex classes $M$, $U$,
	and $W$, where $M=V^{6\ell+6}$, while $U$ and $W$ are two copies of $V$. We represent 
	the vertices in $M$ as sequences
	\[
		\seq{m}
		=
		(p_1, \ldots, p_{3\ell+1},q_1,q_2,q_3,q_4,r_{3\ell+1}, \ldots, r_1)
		=
		(\seq p, \seq q, \seq r)\,.
	\]
	The edges of $A$ 
	are defined to be the triples $\{u, \seq{m}, w\}$ with $\seq{m}\in M$, $u\in U$, 
	$w\in W$, and $(u, \seq{m}, w)\in T$. Thus Claim~\ref{Q} implies 
	\begin{equation}\label{eq:1244}
		e(A)=|T|\ge \zeta^3\theta_3^2 n^{6\l+8}= \zeta^3\theta_3^2 |M||U||W|\,.
	\end{equation}
	For every vertex $\seq{m}\in M$ we consider its (ordered) bipartite link graph
	\[
		A_{\seq{m}}=\{(u, w)\in U\times W\colon \seq{m}uw\in E(A)\}\,.
	\]
	A standard convexity argument yields 
	\begin{equation}\label{eq:1251}
		\sum_{\seq{m}\in M} |A_{\seq{m}}|^{\ell+2}
		\ge
		|M|\left(\frac{e(A)}{|M|}\right)^{\ell+2}
		\overset{\eqref{eq:1244}}{\ge}
		\zeta^{3\ell+6}\theta_3^{2\ell+4} n^{8\ell+10}
		\overset{\eqref{eq:1253}}{=}
		2\theta n^{8\ell+10}\,.
	\end{equation}

	\begin{figure}[t]
	\begin{tikzpicture}[scale=1]
	
		\def\an{12.8};			
		\def\ra{2.5cm};        
		\def\s{1.2};  		
	
		\coordinate (c6) at (-7*\an:\ra);
		\coordinate (b6) at (-5*\an:\ra);
		\coordinate (b5) at (-3*\an:\ra);
		\coordinate (b2) at (3*\an:\ra);
		\coordinate (b1) at (5*\an:\ra);
		\coordinate (w) at (-\an:\ra);
		\coordinate (u) at (\an:\ra);
		\coordinate (a7) at (7*\an:\ra);
						
		\coordinate (a) at ($(a7)+( -10*\s, 0)$);
		\coordinate (b) at ($(a7)+(-9*\s, 0)$);
		\coordinate (c) at ($(a7)+(-8*\s, 0)$);
		\coordinate (a1) at ($(a7)+(-7*\s, 0)$);
		\coordinate (a2) at ($(a7)+(-6*\s, 0)$);
		\coordinate (a3) at ($(a7)+(-5*\s,0)$);
		\coordinate (a4) at ($(a7)+(-4*\s, 0)$);
		\coordinate (a5) at ($(a7)+(-3*\s, 0)$);
		\coordinate (a6) at ($(a7)+(-\s, 0)$);
		\coordinate (a56) at ($(a5)!.5!(a6)$);

		\coordinate (z) at ($(c6)+(-10*\s, 0)$);
		\coordinate (y) at ($(c6)+(-9*\s, 0)$);
		\coordinate (x) at ($(c6)+(-8*\s, 0)$);
		\coordinate (c1) at ($(c6)+(-7*\s, 0)$);
		\coordinate (c2) at ($(c6)+(-4*\s, 0)$);
		\coordinate (c12) at ($(c1)!.5!(c2)$);
		\coordinate (c3) at ($(c6)+(-3*\s, 0)$);
		\coordinate (c4) at ($(c6)+(-2*\s, 0)$);
		\coordinate (c5) at ($(c6)+(-\s, 0)$);
				
		\def\p{.9};     
		
		\coordinate (w1) at ($(x)!.5!(c1)+(0,\p)$);
		\coordinate (w2) at ($(c2)!.5!(c3)+(0,\p)$);
		\coordinate (w3) at ($(c5)!.4!(c6)+(0,\p)$);
		
		\coordinate (u1) at ($(c)!.5!(a1)+(0,-\p)$);
		\coordinate (u2) at ($(a3)!.5!(a4)+(0,-\p)$);
		\coordinate (u3) at ($(a6)!.4!(a7)+(0,-\p)$);

		\redge{(w1)}{(x)}{(y)}{(z)}{7.5pt}{1.5pt}{yellow!80!black}{yellow,opacity=0.2};
		\redge{(w1)}{(c1)}{(x)}{(y)}{7.5pt}{1.5pt}{yellow!80!black}{yellow,opacity=0.2};
		\redge{(w2)}{(c4)}{(c3)}{(c2)}{7.5pt}{1.5pt}{yellow!80!black}{yellow,opacity=0.2};
		\redge{(w2)}{(c5)}{(c4)}{(c3)}{7.5pt}{1.5pt}{yellow!80!black}{yellow,opacity=0.2};
		
		\redge{(w3)}{(c5)}{(c4)}{(c3)}{7.5pt}{1.5pt}{yellow!80!black}{yellow,opacity=0.2};
		\redge{(w3)}{(c6)}{(c5)}{(c4)}{7.5pt}{1.5pt}{yellow!80!black}{yellow,opacity=0.2};
		\redge{(w3)}{(b6)}{(c6)}{(c5)}{7.5pt}{1.5pt}{yellow!80!black}{yellow,opacity=0.2};
		\redge{(w3)}{(b5)}{(b6)}{(c6)}{7.5pt}{1.5pt}{yellow!80!black}{yellow,opacity=0.2};
		
		\redge{(w)}{(b5)}{(b6)}{(c6)}{7.5pt}{1.5pt}{yellow!80!black}{yellow,opacity=0.2};
		\redge{(u)}{(w)}{(b5)}{(b6)}{7.5pt}{1.5pt}{yellow!80!black}{yellow,opacity=0.2};
		\redge{(b2)}{(u)}{(w)}{(b5)}{7.5pt}{1.5pt}{yellow!80!black}{yellow,opacity=0.2};
		\redge{(b1)}{(b2)}{(u)}{(w)}{7.5pt}{1.5pt}{yellow!80!black}{yellow,opacity=0.2};
		\redge{(a7)}{(b1)}{(b2)}{(u)}{7.5pt}{1.5pt}{yellow!80!black}{yellow,opacity=0.2};
		
		\redge{(u3)}{(a7)}{(b1)}{(b2)}{7.5pt}{1.5pt}{yellow!80!black}{yellow,opacity=0.2};
		\redge{(u3)}{(a6)}{(a7)}{(b1)}{7.5pt}{1.5pt}{yellow!80!black}{yellow,opacity=0.2};
		
		\redge{(u2)}{(a3)}{(a4)}{(a5)}{7.5pt}{1.5pt}{yellow!80!black}{yellow,opacity=0.2};
		\redge{(u2)}{(a2)}{(a3)}{(a4)}{7.5pt}{1.5pt}{yellow!80!black}{yellow,opacity=0.2};
		\redge{(u2)}{(a1)}{(a2)}{(a3)}{7.5pt}{1.5pt}{yellow!80!black}{yellow,opacity=0.2};
		
		\redge{(u1)}{(a)}{(b)}{(c)}{7.5pt}{1.5pt}{yellow!80!black}{yellow,opacity=0.2};
		\redge{(u1)}{(b)}{(c)}{(a1)}{7.5pt}{1.5pt}{yellow!80!black}{yellow,opacity=0.2};
		\redge{(u1)}{(c)}{(a1)}{(a2)}{7.5pt}{1.5pt}{yellow!80!black}{yellow,opacity=0.2};
		\redge{(u1)}{(a1)}{(a2)}{(a3)}{7.5pt}{1.5pt}{yellow!80!black}{yellow,opacity=0.2};

		\foreach \i in {c6,b6,b5,w,u,b2,b1,a,b,c,x,y,z}{
			\fill  (\i) circle (2pt);}
		\foreach \i in {1,...,7}{
			\fill  (a\i) circle (2pt);}
		\foreach \i in {1,...,5}{
			\fill  (c\i) circle (2pt);}
		\foreach \i in {1,...,3}{
			\fill  (w\i) circle (2pt);
			\fill  (u\i) circle (2pt);}

		\node at (c12) {\Huge $\dots$};
		\node at (a56) {\Huge $\dots$};
		
		\def\x{.42};  
		\def\y{.63}; 
		
		\foreach \i in {x,y,z}{
			\node[anchor=base] at ($(\i)+(0,-\y)$) {$\i$};}
		
		\foreach \i in {a,b,c}{
			\node[anchor=base] at ($(\i)+(0,\x)$) {$\i$};}
		
		\node[anchor=base] at ($(c1)+(0,-\y)$) {$r_1$};
		\node[anchor=base] at ($(c2)+(0,-\y)$) {$r_{3\ell -3}$};
		\node[anchor=base] at ($(c3)+(0,-\y)$) {$r_{3\ell - 2}$};
		\node[anchor=base] at ($(c4)+(0,-\y)$) {$r_{3\ell - 1}$};
		\node[anchor=base] at ($(c5)+(0,-\y)$) {$r_{3\ell}$};
		\node[anchor=base] at ($(c6)+(0,-\y)$) {$r_{3\ell +1}$};
		
		\node[anchor=base] at ($(a1)+(0,\x)$) {$p_1$};
		\node[anchor=base] at ($(a2)+(0,\x)$) {$p_2$};
		\node[anchor=base] at ($(a3)+(0,\x)$) {$p_3$};
		\node[anchor=base] at ($(a4)+(0,\x)$) {$p_4$};
		\node[anchor=base] at ($(a5)+(0,\x)$) {$p_5$};
		\node[anchor=base] at ($(a6)+(0,\x)$) {$p_{3\ell}$};
		\node[anchor=base] at ($(a7)+(0,\x)$) {$p_{3\ell +1}$};
	
		\node at ($(b6)+(-5*\an:.57)$) {$q_4$};
		\node at ($(b5)+(-3*\an:.58)$) {$q_3$};
		\node[anchor=east] at ($(w)-(0.8,0)$) {$w_{\ell+2}$};
		\node[anchor=east] at ($(u)-(0.8,0)$) {$u_{\ell+2}$};
		\node at ($(b2)+(3*\an:.55)$) {$q_2$};
		\node at ($(b1)+(5*\an:.52)$) {$q_1$};
		
		\node[anchor=base] at ($(w1)+(0,\x)$) {$w_1$};
		\node[anchor=base] at ($(w2)+(0,\x)$) {$w_\ell$};
		\node[anchor=base] at ($(w3)+(0,\x)$) {$w_{\ell+1}$};
		
		\node[anchor=base] at ($(u1)+(0,-\y)$) {$u_1$};
		\node[anchor=base] at ($(u2)+(0,-\y)$) {$u_2$};
		\node[anchor=base] at ($(u3)+(0,-\y)$) {$u_{\ell+1}$};
	\end{tikzpicture}
	\caption{Connecting $(a,b,c)$ and $(x,y,z)$.}
	\label{fig:con}
\end{figure}

	As we will check below, if $\seq{m}\in M$ 
	and $(u_1, w_1), \ldots, (u_{\ell+2}, w_{\ell+2})\in A_{\seq{m}}$, then 
	\[
		abcu_1p_1p_2p_3u_2\ldots u_{\l+1}p_{3\l+1}q_1q_2u_{\l+2} 
		w_{\l+2}q_3q_4r_{3\l+1}w_{\l+1}r_{3\l}r_{3\l-1}r_{3\l-2}w_{\l}\dots w_1xyz
	\]
	is an $abc$-$xyz$-walk in $H$ with $8\ell+10$ inner vertices (see Figure~\ref{fig:con}). By~\eqref{eq:1251}
	this argument produces at least $2\theta n^{8\ell+10}$
	such walks and, as at most $O(n^{8\ell+9})$ of them can fail to be paths, this 
	will conclude the proof of Proposition~\ref{lem:con}.

	It remains to verify that any four consecutive vertices in the above sequence form 
	an edge of~$H$. Recall that $(u_i, \seq{m}, w_i)\in T$ for every $i\in [\ell+2]$.
	So~\ref{it:Q1} implies $abcu_1\in E$ and $w_1xyz\in E$, respectively.
	Since $bcp_1\dots p_{3\l+1}q_1q_2$ is a 3-uniform path in each of the link 
	hypergraphs $\overline{H}_{u_1}, \ldots, \overline{H}_{u_{\ell+2}}$ 
	by~\ref{it:Q5}, 
	we have 
	\[
		bcu_1p_1,\; cu_1p_1p_2,\; u_1p_1p_2p_3,\; p_1p_2p_3u_2,\;\dots,\; 
		p_{3\ell+1}q_1q_2u_{\ell+2}\in E
	\]
	and a similar argument utilising~\ref{it:Q6} establishes 
	\[
		w_{\l+2}q_3q_4r_{3\l+1},\;q_3q_4r_{3\l+1}w_{\l+1},\dots,\; r_1w_1xy\in E\,. 
	\]
	It remains to note, that by~\ref{it:Q2} we have
	\[
		q_1q_2u_{\l+2}w_{\l+2},\;q_2u_{\l+2}w_{\l+2}q_3,\;u_{\l+2}w_{\l+2}q_3q_4\in E\,,
	\]
	which completes the proof of Proposition~\ref{lem:con}.
\end{proof}

\subsection{Other residue classes}
The almost spanning cycle to be constructed in Section~\ref{sec:main-pf} will be obtained 
from an almost spanning path cover with the help 
of the connecting lemma. The number of inner vertices appearing in the last connection 
will determine in which residue class modulo four the number of left-over vertices will lie. 
As the nature of our absorbing mechanism requires that the number of left-over
vertices should be divisible by four, it will be useful to strengthen
the connecting lemma as follows. 

\begin{cor}\label{all4}
	Given Setup~\ref{setup:2335} and $\zeta>0$, there 
	exist natural numbers $\ell_1, \ell_2, \ell_3, \ell_4\le 50\ell$
	with~$\ell_i\equiv i\pmod{4}$ for all $i\in [4]$ and $\theta=\theta(\alpha,\beta,\l,\zeta)>0$
	such that the following holds.
	
	If $(a, b, c)$, $(x, y, z)$ are disjoint $\zeta$-connectable 
	triples of vertices of $H$, then for every~$i\in[4]$ the number of $abc$-$xyz$-paths in $H$ 
	with $\ell_i$ inner vertices is 
	at least $\theta|V|^{\ell_i}$.
\end{cor}
   
The proof will be established in almost the same way as Corollary~\ref{all3}, the main difference
being that, instead of bridges, we utilise connectable
triples to build connecting paths whose 
number of inner vertices is incongruent to $2$ modulo $4$ (cf.\ Proposition~\ref{lem:con}). 
For the proof we first observe that there are many connectable triples in the $4$-uniform 
hypergraph~$H=(V,E)$ under consideration. 

\begin{lemma}\label{NCT} 
	Given Setup~\ref{setup:2335} and $\zeta\in (0, \alpha/4)$, the number 
	of $\zeta$-connectable triples in $H$ is at least $(1/3-2\zeta)|V|^3$.
\end{lemma}

\begin{proof}
	Let $N$ be the number of $\zeta$-connectable triples in $H=(V,E)$. We will estimate the 
	number
	\[
		\Pi=\big|\{(v,e)\colon e \mbox{ is a $\zeta$-bridge in } \overline{H}_v\}\big|\,.
	\]
	in two different ways. First, Lemma~\ref{NB3} tells us that for every vertex $v\in V$ 
	there are at least $(n-1)^3/3$ different $\zeta$-bridges in $\overline{H}_v$, which yields  
	$\Pi\ge n(n-1)^3/3\ge (1/3-\zeta)n^4$ for sufficiently large $n$. Second, we have 
	\[
		\Pi\le N\cdot n+n^3\cdot \zeta n\,,
	\]
	since every $\zeta$-connectable triple $e$ participates in at most $n$ 
	pairs $(v, e)\in \Pi$, while every triple $e$ that fails to be $\zeta$-connectable 
	can be a $\zeta$-bridge in at most $\zeta n$ link hypergraphs. 

	By comparing our estimates on $\Pi$ we obtain 
	\[
		N\ge (1/3-\zeta)n^3-\zeta n^3=(1/3-2\zeta)n^3\,,
	\]
	as promised.
\end{proof}

\begin{proof}[Proof of Corollary~\ref{all4}] 
	By monotonicity we may suppose that $\zeta<\alpha/4<1/12$ and let $\theta_1>0$ be given by 
	Proposition~\ref{lem:con}. We set
	\[
		\l_1=32\l+49\,,\quad
		\l_2=8\l+10\,,\quad
		\l_3=16\l+23\,,\quad
		\l_4=24\l+36\,,\qand
		\theta=\frac{\theta^4_1}{7^3}\,,
	\]
	It follows from $\l\geq 3$ that $\l_2\leq \l_3\leq \l_4\leq \l_1\leq 50\l$
	and Proposition~\ref{lem:con} directly asserts the conclusion of Corollary~\ref{all4} for 
	$i=2$. 	
	
	For $i=3$ we use the following argument. 
	Given disjoint $\zeta$-connectable triples $(a, b, c)$ and $(x, y, z)$ 
	Lemma~\ref{NCT} delivers for sufficiently large $n$ at least $n^3/6$ 
	different $\zeta$-connectable triples $(u, v, w)$ in $H$ 
	with $\{a, b, c, x, y, z\}\cap \{u, v, w\}=\varnothing$ . For each of them, 
	Proposition~\ref{lem:con} provides $\theta_1n^{\ell_2}$ $abc$-$uvw$-paths of the form 
	$abcPuvw$, where $P$ consists of $\ell_2$ vertices. Similarly, there are $\theta_1n^{\ell_2}$
	$uvw$-$xyz$-paths of the form $uvwQxyz$, where $Q$ consists of $\ell_2$ vertices as well.
	Altogether, this yields $\theta_1^2n^{\ell_3}/6$ 
	$abc$-$xyz$-walks of the form $abcPuvwQxyz$
	Since at most $O(n^{\ell_3-1})$ of them fail to be a path due to some overlap between $P$ 
	and $Q$, the corollary follows for $i=3$.
	
	For $i=0$ and $i=1$ we argue similarly, exploiting $\ell_4=\ell_2+\ell_3+3$
	and $\ell_1=\ell_3+\ell_3+3$, respectively.    
\end{proof}
  
\subsection{Bridges in 4-uniform hypergraphs}
We conclude this section with some results that will be helpful in Section~\ref{sec:Abpa}.
The following is a $4$-uniform analogue of Lemma~\ref{F41}.

\begin{lemma}\label{F41analog}
	Given Setup~\ref{setup:2335} and $\zeta>0$, there are at most $\zeta |V|^4$ 
	quadruples $(a, b, c, d)\in V^4$ such that $(a, b, c)$ is a $\zeta$-bridge in $\overline{H}_d$, 
	but $(a, b, c)$ is not $\zeta$-connectable in $H$.
\end{lemma}

\begin{proof} 
	It follows from Definition~\ref{def:connectable4} that for every triple 
	$(a, b, c)\in V^3$ that fails to be $\zeta$-connectable in~$H$, there are 
	at most $\zeta |V|$ choices of $d$ such that $(a, b, c)$ is a $\zeta$-bridge in $\overline{H}_d$. 
	Consequently, there are at most $\zeta |V|^4$ quadruples with the properties under consideration. 
\end{proof}
 Similarly to the notion of bridges in $3$-uniform hypergraphs, which was defined 
 by containing connectable pairs (cf.\ Definition~\ref{def:bridge3}), 
 we define $4$-uniform bridges in terms of connectable triples.
 \begin{dfn}\label{def:bridge}
 	Given Setup~\ref{setup:2335} and $\zeta>0$, a quadruple $(a, b, c, d)\in V^4$ 
	is called a~{\it $\zeta$-bridge in $H$} if $abcd\in E$ and $(a, b, c)$ 
	and $(b, c, d)$ are both  $\zeta$-connectable triples in $H$.
\end{dfn}

It will later become important for us that there are plenty of bridges in $H$.
The argument in the proof of the following lemma is very
similar to that in the proof of Lemma~\ref{NB3}.
\begin{lemma}\label{NB4} 
	Given Setup~\ref{setup:2335} and $\zeta>0$ there are at least $(1/9-7\zeta)|V|^4$ $\zeta$-bridges in~$H$. 
\end{lemma}

\begin{proof}
	Let $A=\{(a, b, c, d)\in V^4\colon abcd \in E\}$ be the set of all orderings of the edges of~$H$. 
	Obviously, the minimum pair degree condition imposed on $H$ implies 
	\[
		|A|
		\ge 
		\left(\frac59+\alpha\right)|V|^3(|V|-1)
		\ge 
		\left(\frac59+\alpha-\zeta\right)|V|^4\,.
	\]
	We consider two exceptional subsets of $A$, namely
	\begin{align*}
		P_1&=\{(a, b, c, d)\in A\colon (a, b, c) \mbox{ is not a $\zeta$-bridge in }
			\overline{H}_d \} \\
		\text{and}\quad
		Q_1&=\{(a, b, c, d)\in A\setminus P_1 \colon \text{$(a, b, c)$ is not $\zeta$-connectable in $H$}\}\,.
	\end{align*}
	It follows directly from Lemma~\ref{F41analog}, that 
	\[
		|Q_1|\leq \zeta |V|^4\,.
	\]
	Moreover, by Lemma~\ref{NB3} every $d\in V$
	contributes at most $(2/9+\alpha/2+2\zeta)(|V|-1)^3$ quadruples to $P_1$, 
	which yields the upper bound 
	\[
		|P_1|\leq \Big(\frac{2}{9}+\frac{\alpha}{2}+2\zeta\Big)|V|^4\,.
	\]
	By symmetry we obtain the same bounds for the sets 
	\begin{align*}
		P_2&=\{(a, b, c, d)\in A\colon (b, c, d) \mbox{ is not a $\zeta$-bridge in }
			\overline{H}_a\} \\
		\text{and}\quad
		Q_2&=\{(a, b, c, d)\in A\setminus P_2 \colon \text{$(b, c, d)$ is not $\zeta$-connectable in $H$}\}\,.
	\end{align*}
	Since every quadruple in $A\setminus (P_1\cup Q_1\cup P_2\cup Q_2)$ is 
	a $\zeta$-bridge in $H$, the lemma follows.
\end{proof}
  
\section{Reservoir lemma}\label{sec:reservoir}
The connecting lemma for $4$-uniform hypergraphs from Section~\ref{conn} allows us to connect 
paths that start and end with a connectable triple. However, in the process of building longer paths, we
must not interfere with the paths already constructed. For that we put aside a randomly selected small 
\emph{reservoir} of vertices~$\cR$.
Moreover, due to the divisibility restriction of the absorbing path lemma (see Proposition~\ref{prop:absorbingP}), 
we need to guarantee short connections by paths of lengths in all residue classes modulo four.
The existence of such a reservoir set is given by the following proposition.

\vbox{
\begin{prop}[Reservoir lemma]\label{prop:reservoir}
	Given Setup~\ref{setup:2335} and constants $\zetas$, $\zetass>0$,
	let integers $\ell_1, \ell_2, \ell_3, \ell_4\le 50\ell$ and 
	$\thetas=\theta(\alpha,\beta,\l,\zetas)$ and $\thetass=\theta(\alpha,\beta,\l,\zetass)$
	be provided by Corollary~\ref{all4}.
	Then there exists a subset $\cR\subseteq V$ such that
	\begin{enumerate}[label=\rmlabel]
	\item\label{it:R1} $\thetas^2|V|/2\leq|\cR|\leq \thetas^2|V|$
	\item\label{it:R2} and for all disjoint, $\zetass$-connectable triples $(a, b, c)$, $(x, y, z)$ in $H$ 
		and every $i\in[4]$, there are $\thetass|\cR|^{\l_i}/2$
		$abc$-$xyz$-paths with $\ell_i$ inner vertices, 
		which all belong to $\cR$.
	\end{enumerate}
\end{prop}}

We often refer to the set~$\cR$ given by Proposition~\ref{prop:reservoir} as 
the \emph{reservoir set}. 

\begin{proof}
	The existence of such a reservoir set $\cR$ is established by a
	standard probabilistic argument.
	For that we set 
	\[
		p=\frac34\thetas^2
		\qqand
		C=\Big(\frac{4}{3}\Big)^{\frac{1}{50\l}}
	\]
	and we consider a random subset $\cR\subseteq V$ with elements included
	independently with probability $p$.
	Observe that~$|\cR|$ is binomially distributed with expectation $p |V|$ and 
	Chebyshev's inequality implies that a.a.s.\ 
	\begin{equation}\label{eq:sizeR}
		\frac{p}{C}|V|
		\le
		|\cR|
		\le 
		C p |V|\,.
	\end{equation}
	In particular, our choice of $C$ shows that a.a.s.\ the set $\cR$ satisfies 
	part~\ref{it:R1} of Proposition~\ref{prop:reservoir}. 
	
	For part~\ref{it:R2} we recall that for every pair of disjoint, $\zetass$-connectable triples 
	$(a,b,c)$, $(x,y,z)\in V^3$, Corollary~\ref{all4} guarantees for every $i\in[4]$ at least 
	$\thetass |V|^{\l_i}$ $abc$-$xyz$-paths with~$\l_i$ inner vertices.	
	Let
	$X=X(i,(a,b,c),(x,y,z))$ be the random variable counting the number of such $abc$-$xyz$-paths
	with all $\l_i$ inner vertices in $\cR$.
	Clearly,
	\begin{equation}\label{eq:EEX}
		\EE X 
		\geq
		p^{\l_i}\cdot\thetass|V|^{\l_i}\,.
	\end{equation}
	Since including or not including a particular vertex into $\cR$ 
	affects the random variable~$X$ by no more than $\l_i |V|^{\l_i-1}$, 
	the Azuma--Hoeffding inequality (see, e.g.,~\cite{JLR00}*{Corollary~2.27})
	asserts
	\begin{align}
		\PP\big(X\leq \tfrac{2}{3}\thetass(p |V|)^{\l_i}\big)
		&\overset{\eqref{eq:EEX}}{\leq}
		\PP\big(X\leq \tfrac{2}{3}\EE X\big)\nonumber\\
		&\overset{\phantom{\eqref{eq:EEX}}}{\leq}
		\exp\left(-\frac{(\EE X)^2}{18\cdot |V|\cdot (\l_i|V|^{\l_i-1})^2}\right)
		=\exp\big(-\Omega(|V|)\big)\,.\label{eq:RAzuma}
\end{align}
In view of~\eqref{eq:sizeR} and $\l_i\leq 50\l$ our choice of $C$ implies that a.a.s.\
\begin{equation}\label{Rless}
	\frac{2}{3}\thetass (p|V|)^{\l_i}
	\ge
	\frac12\thetass|\cR|^{\l_i}.
\end{equation}
	Since there are at most $4|V|^6$ choices for the triples $(a,b,c)$, $(x,y,z)$, and for $i$, 
	the union bound combined with~\eqref{eq:RAzuma} and~\eqref{Rless}
	shows that a.a.s.\ the set $\cR$ satisfies part~\ref{it:R2} of Proposition~\ref{prop:reservoir}.
	Consequently,
	a reservoir set $\cR$ with all required properties exists.
\end{proof}

In the proof of Theorem~\ref{thm:main} in Section~\ref{sec:main-pf} 
we will repeatedly connect connectable triples
through the reservoir~$\cR$ provided by Proposition~\ref{prop:reservoir}. 
Whenever such a connection is made some of the vertices of the reservoir are used and 
the part of the reservoir that may still be used for further connections shrinks. Although $\Omega(|V|)$ such
connections will be needed, we shall be able to keep an appropriate 
version of property~\ref{it:R2} of the reservoir intact throughout this process.
We prepare for this situation by the following corollary.

\begin{cor}\label{lem:use-reservoir}
	Given Setup~\ref{setup:2335} and $\zetas$, $\zetass>0$, let 
	integers $\ell_1, \ell_2, \ell_3, \ell_4\le 50\ell$ and 
	$\thetas=\theta(\alpha,\beta,\l,\zetas)$, $\thetass=\theta(\alpha,\beta,\l,\zetass)$
	be provided by Corollary~\ref{all4}.
	Moreover, let $\cR\subseteq V$ be a reservoir set provided by Proposition~\ref{prop:reservoir}.
	Then for every subset $\cR'\subseteq \cR$ of size at most $\frac{\thetas^2\thetass}{400\l}|V|$ the following holds.
	
	For all disjoint, $\zetass$-connectable triples $(a, b, c)$, $(x, y, z)$ in $H$ 
	and every $i\in[4]$, there is some $abc$-$xyz$-path with $\ell_i$ inner vertices, 
	which all belong to $\cR\setminus\cR'$.	
\end{cor}

\begin{proof}
	It follows from the lower bound in Proposition~\ref{prop:reservoir}\,\ref{it:R1} and the bound on $|\cR'|$ 
	that 
	\[
		|\cR'|\leq \frac{\thetass}{200\l}|\cR|\,.
	\]
	Moreover, every  vertex in $\cR'$ is an inner vertex in at most $\l_i|\cR|^{\l_i-1}$ different
	$abc$-$xyz$-paths in $H$ with all $\l_i$ inner vertices  belonging to $\cR$. Consequently,
	it follows from Proposition~\ref{prop:reservoir}\,\ref{it:R2} and $\l_i\leq 50\l$ that there are at least 
	\[
		\frac{\thetass}{2}|\cR|^{\l_i} - |\cR'|\cdot\l_i|\cR|^{\l_i-1} 
		\geq 
		\frac{\thetass}{4}|\cR|^{\l_i}
	\]
	such paths with all inner vertices from $\cR\setminus\cR'$.
\end{proof}

\section{Absorbing path lemma}
\label{sec:Abpa}

\subsection{Outline and main ideas}
\label{sec:abs_intro}
In this section we establish the existence of an \emph{absorbing path~$P_A$}, which at the end of the proof of Theorem~\ref{thm:main}
will allow us to `absorb' an arbitrary (but not too large) set $Z$ of left-over vertices 
with a size divisible by four.

\begin{prop}[Absorbing path lemma]\label{prop:absorbingP} 
	Given Setup~\ref{setup:2335}, there is some $\zeta_0=\zeta_0(\alpha)>0$ such that 
	for every $\zetas\in(0,\zeta_0)$ and for	
	$\thetas=\theta(\alpha,\beta,\l,\zetas)$
	provided by Proposition~\ref{lem:con} the following holds. 
	For every set 
	$\cR\subseteq V$ of size at most $\thetas^2|V|$,
	there exists a path~$P_A\subseteq H-\cR$ satisfying
	\begin{enumerate}[label=\rmlabel]
	\item\label{it:PA1} $ |V(P_A)|\le \thetas |V|$,
	\item\label{it:PA2} the end-triples of $P_A$ are $\zetas$-connectable, 
	\item\label{it:PA3} and for every subset $Z\subseteq V\setminus V(P_A)$
		with $|Z|\le 2\thetas^2n$ and $|Z|\equiv0 \pmod 4$, there exists a path $Q\subseteq H$ with the same end-triples as $P_A$
		and $V(Q)=V(P_A)\cup Z$.
	\end{enumerate}
\end{prop}

The absorbing path $P_A$ will be built by connecting many so-called \emph{absorbers} (see Definition~\ref{absorber}).
Similarly as in~\cite{old}, the absorbers used here consist of two parts. Roughly speaking, the first part allows us to
``swap'' any given vertex $a$ with a different vertex~$x$, which then can be absorbed by the second part 
of the absorber. In other words, we can move from an arbitrary vertex $a$, which we may need to absorb, to another vertex~$x$ that enjoys better properties. For the first part this can be easily achieved if $a$ and $x$ share a $3$-uniform path with six vertices 
in their joint link $H_a\cap H_x$ (see Figure~\ref{fig:Abs-part1}). Note that our degree assumption on $H$ implies that
the $3$-uniform link of every vertex has density at least $5/9>1/2$ and, hence, the joint link of any two vertices 
has positive density and the existence of the $6$-vertex paths follows from~\cite{E64}.

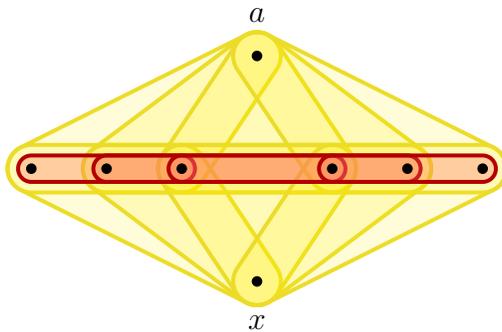
\begin{figure}[ht]
	\begin{tikzpicture}
		
		\def\s{1};
		
		\foreach \i in {1,2, 3}{
			\coordinate (b\i) at ($(\i*\s, 0)-(4*\s,0)$);
			}
		
		\foreach \i in {4, 5, 6}{
			\coordinate (b\i) at ($(\i*\s,0)-(3*\s,0)$);
			}
			
		\coordinate (a) at (0,1.5*\s);
		\coordinate (x) at (0,-1.5*\s);
		
		\def\w{9pt};
				
	\redge{(a)}{(b3)}{(b2)}{(b1)}{\w}{1.5pt}{yellow!90!black}{yellow,opacity=0.15};
	\redge{(a)}{(b4)}{(b3)}{(b2)}{\w}{1.5pt}{yellow!90!black}{yellow,opacity=0.15};
	\redge{(a)}{(b5)}{(b4)}{(b3)}{\w}{1.5pt}{yellow!90!black}{yellow,opacity=0.15};
	\redge{(a)}{(b6)}{(b5)}{(b4)}{\w}{1.5pt}{yellow!90!black}{yellow,opacity=0.15};
	
	\redge{(x)}{(b1)}{(b2)}{(b3)}{\w}{1.5pt}{yellow!90!black}{yellow,opacity=0.15};
	\redge{(x)}{(b2)}{(b3)}{(b4)}{\w}{1.5pt}{yellow!90!black}{yellow,opacity=0.15};
	\redge{(x)}{(b3)}{(b4)}{(b5)}{\w}{1.5pt}{yellow!90!black}{yellow,opacity=0.15};
	\redge{(x)}{(b4)}{(b5)}{(b6)}{\w}{1.5pt}{yellow!90!black}{yellow,opacity=0.15};
	
	\qedge{(b1)}{(b2)}{(b3)}{5pt}{1.5pt}{red!70!black}{red!60!white,opacity=0.3};
	\qedge{(b2)}{(b3)}{(b4)}{5pt}{1.5pt}{red!70!black}{red!60!white,opacity=0.3};
	\qedge{(b3)}{(b4)}{(b5)}{5pt}{1.5pt}{red!70!black}{red!60!white,opacity=0.3};
	\qedge{(b4)}{(b5)}{(b6)}{5pt}{1.5pt}{red!70!black}{red!60!white,opacity=0.3};
				
		\foreach \i in {1,...,6}{
			\fill (b\i) circle (2pt);
			}
		\fill (a) circle (2pt);
		\fill (x) circle (2pt);
		
	\node at ($(a)	+(0,.55)$) {$a$};
	\node at ($(x)	-(0,.55)$) {$x$};
	
	\end{tikzpicture}
	\caption{Both $a$ and $x$ form a $4$-uniform path together with the $3$-uniform path in $H_a\cap H_x$.}
	\label{fig:Abs-part1}
\end{figure}

Having replaced $x$ with $a$ we need to ensure that $x$ itself can be absorbed. For that, in the context of $4$-uniform hypergraphs, 
one usually showed that many vertices~$x$ have the property that their links contain a $3$-uniform path on six vertices 
with the additional property that its vertices span a $4$-uniform path in $H$. In particular, these six vertices form 
a path on its own and can absorb $x$ in the middle, building a $7$-vertex path with the 
same end-triples
(see e.g.~\cite{old}, where this strategy was implemented for $3$-uniform hypergraphs). 

While working on a related problem, the first two authors~\cite{PR} suggested 
a different approach for the second part of the absorber. For that we note that every complete $4$-partite $4$-uniform 
hypergraph~$K^{(4)}_{s,s,s,s}$ contains a path on $4s$ vertices.  However, any four consecutive vertices in that path 
are crossing in the $K^{(4)}_{s,s,s,s}$ and removing them gives rise to a copy of $K^{(4)}_{s-1,s-1,s-1,s-1}$, which again contains a 
spanning path on the remaining $4(s-1)$ vertices. Moreover, if $s\geq 3$ then these paths have the same end-triples.
Actually it suffices already to start with a $K^{(4)}_{3,3,3,2}$ and we will follow that route. Again 
the existence of~$K^{(4)}_{3,3,3,2}$'s in $4$-uniform hypergraphs of positive density follows from~\cite{E64}. However, 
due to the $4$-partiteness, with this absorption mechanism we can only absorb four tuples of vertices $(x_1,\dots,x_4)$, 
which in turn implies that we have to start with four vertices $a_1,\dots,a_4$ at the beginning. This is the reason for the divisibility condition 
on $|Z|$ in part~\ref{it:PA3} of Proposition~\ref{prop:absorbingP}. 

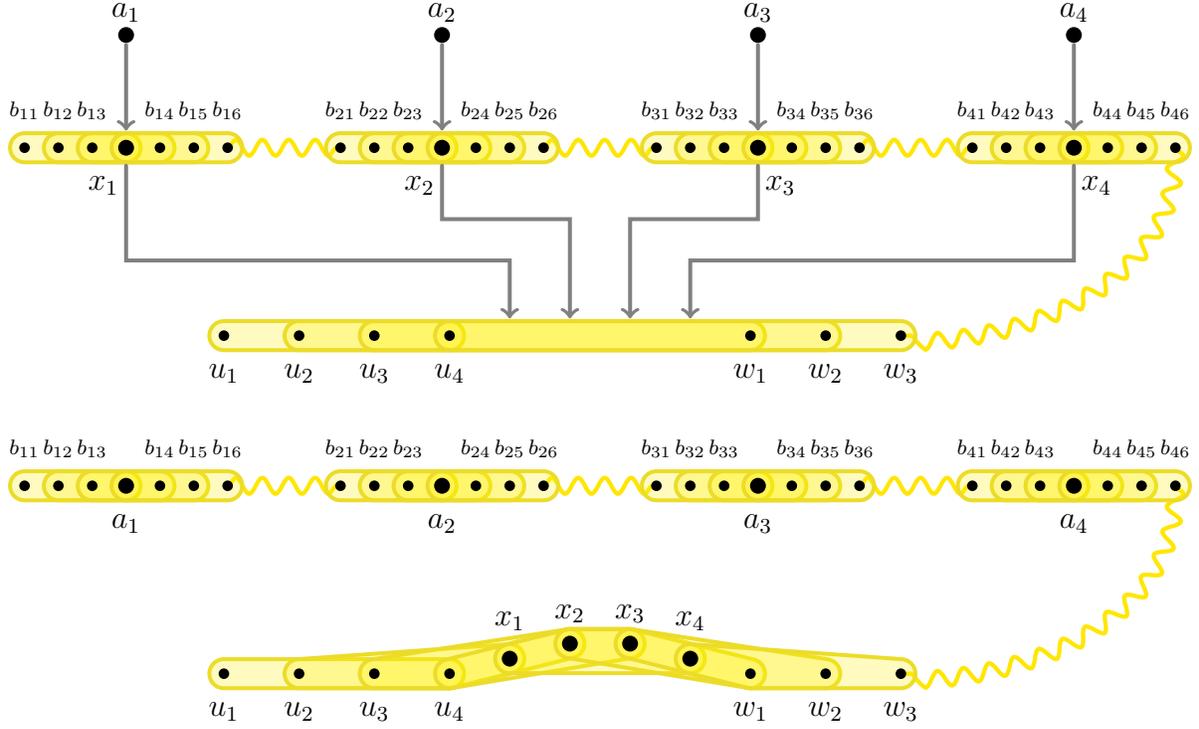
\begin{figure}[t]
	\begin{tikzpicture}[scale=1]
	
	\def \dd{2.1};
	\def \md{.45};
	\def \co{yellow!90!black};

	\coordinate (a1) at (-3*\dd, 0);
	\coordinate (a2) at (-\dd,0);
	\coordinate (a3) at (\dd,0);
	\coordinate (a4) at (3*\dd,0);
	
	\foreach \i in {1,2,3,4}{
		\coordinate (x\i) at ($(a\i)-(0,1.5)$);
		\foreach \j in {1,...,7}{
			\coordinate (b\i\j) at ($(x\i)+(\j*\md,0)-(4*\md,0)$);
		}}
	
	\foreach \i in {1,2,3,4}{
		\foreach \j in {1,2,3}{
			\node at (b\i\j) at ($(b\i\j)+(0,.5)$){\tiny $b_{\i\j}$};}
		\node at (b\i5) at ($(b\i5)+(0,.5)$){\tiny $b_{\i4}$};
		\node at (b\i6) at ($(b\i6)+(0,.5)$){\tiny $b_{\i5}$};
		\node at (b\i7) at ($(b\i7)+(0,.5)$){\tiny $b_{\i6}$};
	}
	
	\coordinate (y1) at (-5,-4);
	\coordinate (y2) at (-4,-4);
	\coordinate (y3) at (-3,-4);
	\coordinate (y4) at (-2,-4);
	\coordinate (z1) at (2,-4);
	\coordinate (z2) at (3,-4);
	\coordinate (z3) at (4,-4);
	
	\coordinate (v1) at ($(y4)!.2!(z1)$);
	\coordinate (v2) at ($(y4)!.4!(z1)$);
	\coordinate (v3) at ($(y4)!.6!(z1)$);
	\coordinate (v4) at ($(y4)!.8!(z1)$);
	
	\foreach \i in {1,2,3,4}{
		\redge{(b\i1)}{(b\i2)}{(b\i3)}{(x\i)}{5.5pt}{1.5pt}{\co}{yellow,opacity=0.25}
		\redge{(b\i2)}{(b\i3)}{(x\i)}{(b\i5)}{5.5pt}{1.5pt}{\co}{yellow,opacity=0.25}
		\redge{(b\i3)}{(x\i)}{(b\i5)}{(b\i6)}{5.5pt}{1.5pt}{\co}{yellow,opacity=0.25}
		\redge{(x\i)}{(b\i5)}{(b\i6)}{(b\i7)}{5.5pt}{1.5pt}{\co}{yellow,opacity=0.25}
		}
	
	\redge{(y1)}{(y2)}{(y3)}{(y4)}{5.5pt}{1.5pt}{\co}{yellow,opacity=0.25}
	\redge{(y2)}{(y3)}{(y4)}{(z1)}{5.5pt}{1.5pt}{\co}{yellow,opacity=0.25}
	\redge{(y3)}{(y4)}{(z1)}{(z2)}{5.5pt}{1.5pt}{\co}{yellow,opacity=0.25}
	\redge{(y4)}{(z1)}{(z2)}{(z3)}{5.5pt}{1.5pt}{\co}{yellow,opacity=0.25}

	\begin{pgfonlayer}{front}
	
		\foreach \i in {1,2,3} {
		\fill (z\i) circle (2pt);
		\node at ($(z\i) -(0,.5)$) {$w_{\i}$};}
	
		\foreach \i in {1,2,3,4}{
			\fill (y\i) circle (2pt);
			\fill (a\i) circle (3pt);
			\fill (x\i) circle (3pt);
			\foreach \j in {1,...,7}{
				\fill (b\i\j) circle (2pt);}
			\draw[black!50!white, line width=1.5pt, line cap=round, shorten <=3.8pt,shorten >=6.5pt,->] (a\i) -- (x\i);
		
		\node at ($(a\i)+(0,.3)$) {$a_{\i}$};
		\node at ($(y\i) -(0,.5)$) {$u_{\i}$};
	}
		\node at ($(x1) -(.3,.5)$) {$x_1$};
		\node at ($(x2) -(.3,.5)$) {$x_2$};
		\node at ($(x3) -(-.3,.5)$) {$x_3$};
		\node at ($(x4) -(-.3,.5)$) {$x_4$};

	\end{pgfonlayer}

	\draw[black!50!white, line width=1.5pt, line cap=round, shorten <=6.8pt,shorten >=6.5pt,->] (x1) -- +(0,-1.5)-- ($(v1)+(0,1)$) --(v1);
	\draw[black!50!white, line width=1.5pt, line cap=round, shorten <=6.8pt,shorten >=6.5pt,->] (x2) -- +(0,-.95)-- ($(v2)+(0,1.55)$) --(v2);
	\draw[black!50!white, line width=1.5pt, line cap=round, shorten <=6.8pt,shorten >=6.5pt,->] (x3) -- +(0,-.95)-- ($(v3)+(0,1.55)$) --(v3);
	\draw[black!50!white, line width=1.5pt, line cap=round, shorten <=6.8pt,shorten >=6.5pt,->] (x4) -- +(0,-1.5)-- ($(v4)+(0,1)$) --(v4);

	\begin{pgfonlayer}{background}
	\draw[ultra thick, decorate, decoration={snake,segment length=10,amplitude=3}, yellow!95!red]
			(b17)--(b21);
	\draw[ultra thick, decorate, decoration={snake,segment length=10,amplitude=3}, yellow!95!red]
			(b27)--(b31);
	\draw[ultra thick, decorate, decoration={snake,segment length=10,amplitude=3}, yellow!95!red]
			(b37)--(b41);
	
	\draw[ultra thick, decorate, decoration={snake,segment length=10,amplitude=3}, yellow!95!red]
			(b47) to[out=270,in=0] (z3);
	\end{pgfonlayer}

	\end{tikzpicture}

\bigskip 

	\begin{tikzpicture}[scale=1]
	
	\def \dd{2.1};
	\def \md{.45};
	\def \co{yellow!90!black};

	\coordinate (a1) at (-3*\dd, 0);
	\coordinate (a2) at (-\dd,0);
	\coordinate (a3) at (\dd,0);
	\coordinate (a4) at (3*\dd,0);
	
	\foreach \i in {1,2,3,4}{
		\coordinate (x\i) at ($(a\i)-(0,1.5)$);
		\foreach \j in {1,...,7}{
			\coordinate (b\i\j) at ($(x\i)+(\j*\md,0)-(4*\md,0)$);
		}}
	
	\foreach \i in {1,2,3,4}{
		\foreach \j in {1,2,3}{
			\node at (b\i\j) at ($(b\i\j)+(0,.5)$){\tiny $b_{\i\j}$};}
		\node at (b\i5) at ($(b\i5)+(0,.5)$){\tiny $b_{\i4}$};
		\node at (b\i6) at ($(b\i6)+(0,.5)$){\tiny $b_{\i5}$};
		\node at (b\i7) at ($(b\i7)+(0,.5)$){\tiny $b_{\i6}$};
		}
	
	\coordinate (y1) at (-5,-4);
	\coordinate (y2) at (-4,-4);
	\coordinate (y3) at (-3,-4);
	\coordinate (y4) at (-2,-4);
	\coordinate (z1) at (2,-4);
	\coordinate (z2) at (3,-4);
	\coordinate (z3) at (4,-4);
	
	\coordinate (v1) at ($(y4)!.2!(z1)+(0,.2)$);
	\coordinate (v2) at ($(y4)!.4!(z1)+(0,.4)$);
	\coordinate (v3) at ($(y4)!.6!(z1)+(0,.4)$);
	\coordinate (v4) at ($(y4)!.8!(z1)+(0,.2)$);
	
	\foreach \i in {1,2,3,4}{
		\redge{(b\i1)}{(b\i2)}{(b\i3)}{(x\i)}{5.5pt}{1.5pt}{\co}{yellow,opacity=0.25}
		\redge{(b\i2)}{(b\i3)}{(x\i)}{(b\i5)}{5.5pt}{1.5pt}{\co}{yellow,opacity=0.25}
		\redge{(b\i3)}{(x\i)}{(b\i5)}{(b\i6)}{5.5pt}{1.5pt}{\co}{yellow,opacity=0.25}
		\redge{(x\i)}{(b\i5)}{(b\i6)}{(b\i7)}{5.5pt}{1.5pt}{\co}{yellow,opacity=0.25}
		}
	
	\redge{(y1)}{(y2)}{(y3)}{(y4)}{5.5pt}{1.5pt}{\co}{yellow,opacity=0.25}
	\redge{(y4)}{(y3)}{(y2)}{(v1)}{5.5pt}{1.5pt}{\co}{yellow,opacity=0.25}
	\redge{(y4)}{(y3)}{(v2)}{(v1)}{5.5pt}{1.5pt}{\co}{yellow,opacity=0.25}
	\redge{(y4)}{(v1)}{(v2)}{(v3)}{5.5pt}{1.5pt}{\co}{yellow,opacity=0.25}
	\redge{(v1)}{(v2)}{(v3)}{(v4)}{5.5pt}{1.5pt}{\co}{yellow,opacity=0.25}
	\redge{(v2)}{(v3)}{(v4)}{(z1)}{5.5pt}{1.5pt}{\co}{yellow,opacity=0.25}
	\redge{(z2)}{(z1)}{(v4)}{(v3)}{5.5pt}{1.5pt}{\co}{yellow,opacity=0.25}
	\redge{(z3)}{(z2)}{(z1)}{(v4)}{5.5pt}{1.5pt}{\co}{yellow,opacity=0.25}

	\begin{pgfonlayer}{front}
	
		\foreach \i in {1,2,3} {
		\fill (z\i) circle (2pt);
		\node at ($(z\i) -(0,.5)$) {$w_{\i}$};}
	
		\foreach \i in {1,2,3,4}{
			\fill (y\i) circle (2pt);
			\fill (v\i) circle (3pt);
			\fill (x\i) circle (3pt);
			\foreach \j in {1,...,7}{
				\fill (b\i\j) circle (2pt);}
		\node at ($(y\i) -(0,.5)$) {$u_{\i}$};
		\node at ($(x\i) -(0,.5)$) {$a_{\i}$};
	}
	\node at ($(v1) +(0,.5)$) {$x_{1}$};
	\node at ($(v4) +(0,.5)$) {$x_{4}$};
	\node at ($(v2) +(0,.4)$) {$x_{2}$};
	\node at ($(v3) +(0,.4)$) {$x_{3}$};
		
	\end{pgfonlayer}

	\begin{pgfonlayer}{background}
	
	\draw[ultra thick, decorate, decoration={snake,segment length=10,amplitude=3}, yellow!95!red]
			(b17)--(b21);
	\draw[ultra thick, decorate, decoration={snake,segment length=10,amplitude=3}, yellow!95!red]
			(b27)--(b31);
	\draw[ultra thick, decorate, decoration={snake,segment length=10,amplitude=3}, yellow!95!red]
			(b37)--(b41);
	
	\draw[ultra thick, decorate, decoration={snake,segment length=10,amplitude=3}, yellow!95!red]
			(b47) to[out=270,in=0] (z3);

	\end{pgfonlayer}

				\end{tikzpicture}
				\caption{Absorber for $(a_1,\dots, a_4)$ before and after absorption.}
				\label{fig:absorberafter}
			\end{figure}

As a result for any given $(a_1,\dots,a_4)\in V^4$ our absorbers will consist of $35$ vertices, which split into five $7$-vertex paths
(see Figure~\ref{fig:absorberafter}). 
Four of the paths are of the form $b_{i1}b_{i2}b_{i3}x_ib_{i4}b_{i5}b_{i6}$ for $i\in[4]$, where $b_{i1}b_{i2}b_{i3}b_{i4}b_{i5}b_{i6}$
is a $3$-uniform path in the joint link $H_{a_i}\cap H_{x_i}$ (cf.\ first part of the absorber outlined above). 
The fifth path  $u_{1}\dots u_4w_{1}w_{2}w_{3}$ is given by the vertices of a $K^{(4)}_{2,2,2,1}$, which together with 
$x_1,\dots,x_4$ span a $K^{(4)}_{3,3,3,2}$. In order to connect these paths into one absorbing path $P_A$, we shall also require 
that the end-triples of these paths are connectable (see Lemmata~\ref{b6} and~\ref{manyelves} below).

\subsection{Proof of the absorbing path lemma}
Roughly speaking, the following lemma shows that
the joint $3$-uniform link $H_a\cap H_x$ of (almost) all pairs of vertices $a$, $x\in V$ 
contains $\Omega(|V|^3)$ connectable triples. Consequently, a result of Erd\H os~\cite{E64} 
implies that the joint link contains $\Omega(|V|^6)$ $3$-uniform paths on six vertices with connectable end-triples 
(in fact all triples will be connectable), which shows the abundant existence of the first part of our absorbers 
for every vertex~$a\in V$ (see Lemma~\ref{b6} below).

\begin{lemma}\label{coro} 
	Given Setup~\ref{setup:2335} and $\zeta>0$, there is a set $X\subseteq V$ of size $|X|\le \sqrt\zeta n$ 
	such that for all $a\in V$ and every $x\in V\setminus X$ there are at least $(\alpha/3-\sqrt\zeta )|V|^3$ 
	triples $(b,b',b'')\in V^3$ with $bb'b''\in E(H_a)\cap E(H_x)$ and $(b,b',b'')$ being $\zeta$-connectable in $H$.
\end{lemma}
\begin{proof} 
The lemma is trivially true for $\zeta\geq\alpha^2/9$ and, hence, we may assume $\zeta<\alpha^2/9$.
First, we define the set $X$. For a vertex $v\in V$ let 
\[
	\cB(v)=\{(b,b',b'')\in V^{3}\colon \text{$(b,b',b'')$ is a $\zeta$-bridge in $\overline{H}_v$, but 
		it is not $\zeta$-connectable in $H$}\}
\]
and we note that Lemma~\ref{F41analog} asserts
\[
	\sum_{v\in V}|\cB(v)|\leq \zeta n^4\,.
\] 
We define
\[
	X=\{v\in V\colon  |\cB(v)|\geq \sqrt{\zeta}n^3\}\,.
\]
and $|X|\leq\sqrt\zeta n$ follows.

It is left to show that $V\setminus X$ has the desired property. For that 
let $a\in V$ and $x\in V\setminus X$. 
An application of Lemma~\ref{L35} with $H=\overline{H}_x$ and $H'=\overline{H}_a$ yields 
for sufficiently large $n$ at least 
\[
	\frac{\alpha}{2}(n-1)^3
	>
	\frac{\alpha}{3}n^3
\]
triples $(b,b',b'')\in V^3$ such that 
\begin{equation}
\label{eq:coro1}
	bb'b''\in E(H_a)\cap E(H_x) \qand 
	\text{$(b,b',b'')$ is a $\zeta$-bridge in $H_x$.}
\end{equation}
Since $x\not\in X$, we have $|\cB(x)|<\sqrt{\zeta}n^3$
and, therefore, all but at most $\sqrt{\zeta}n^3$ of the triples 
$(b,b',b'')$ satisfying~\eqref{eq:coro1} are $\zeta$-connectable.
\end{proof}

Lemma~\ref{coro} combined with a result of Erd\H os from~\cite{E64} implies the following.

\begin{lemma}\label{b6}
Given Setup~\ref{setup:2335} and $\zeta\in(0,\alpha^2/36)$, there is 
some $\xi'=\xi'(\alpha)>0$ and a set~$X\subseteq V$ of size at most $\sqrt{\zeta}n$ 
such that the following holds.

For all $a\in V$ and $x\in V\setminus X$,  
there are at least $\xi' n^6$ sextuples $(b_1,\dots,b_6)\in V^6$ such that
\begin{enumerate}[label=\rmlabel]
	\item\label{it:Abs-S1} $b_1b_2\dots b_6$ is a $3$-uniform path in $H_a\cap H_x$
	\item\label{it:Abs-S2} and $(b_1,b_2,b_3)$, $(b_4,b_5,b_6)$ are $\zeta$-connectable in~$H$.
\end{enumerate}
\end{lemma}

\begin{proof} Let $X$ be given by Lemma~\ref{coro}
and fix two vertices $a\in V$ and $x\in V\setminus X$. We consider the auxiliary $3$-partite $3$-uniform hypergraph $B=(U\dcup U'\dcup U'',E_B)$ 
whose vertex classes are three disjoint copies of $V$ and 
edges $bb'b''\in E_B$ with $b\in U$, $b'\in U'$, and $b''\in U''$
correspond to $\zeta$-connectable triples $(b,b',b'')$ of $H$ with  $bb'b''\in E(H_a)\cap E(H_x)$.
Lemma~\ref{coro} and $\zeta<\alpha^2/36$ tell us that~$|E_B|\geq \alpha n^3/6$ 
and it follows from~\cite{E64} that~$B$ contains any complete $3$-partite $3$-uniform hypergraph of fixed size. 
In particular, there is a copy of~$K^{(3)}_{2,2,2}$ in $B$ and by the so-called supersaturation phenomenon (see, e.g.,~\cite{ES83})
there are at least 
$2\xi' n^6$ such copies for some constant $\xi'=\xi'(\alpha)$. Each such copy of $K^{(3)}_{2,2,2}$ contains a walk $b_1b_2\dots b_6$ in
$H_a\cap H_x$ satisfying~\ref{it:Abs-S2} and, consequently, there are at least $2\xi' n^6-O(n^5)\geq \xi' n^6$ paths satisfying~\ref{it:Abs-S1}
and~\ref{it:Abs-S2}.
\end{proof}

Next we focus on the second part of the absorbers.

\begin{lemma}\label{manyelves}
	There is some $\xi''>0$ such that for every $\zeta\in(0,1/126)$ the following holds.
	Given Setup~\ref{setup:2335}, there are $\xi''n^{11}$ $11$-tuples $(u_1,\dots,u_4,x_1,\dots,x_4,w_1,w_2,w_3)\in V^{11}$
	so that
	\begin{enumerate}[label=\rmlabel]
	\item\label{it:Abs-E1} $u_1\dots u_4x_1\dots x_4w_1w_2w_3$
		and $u_1\dots u_4w_1w_2w_3$ are $4$-uniform paths in $H$
	\item\label{it:Abs-E2} and $(u_1,u_2,u_3)$, $(w_1,w_2,w_3)$ are $\zeta$-connectable triples in~$H$.
	\end{enumerate}
\end{lemma}
The proof of Lemma~\ref{manyelves} is very similar to the proof of Lemma~\ref{b6}. However, instead of an auxiliary $3$-uniform hypergraph 
of connectable triples in the shared link of two vertices, we shall consider a $4$-uniform hypergraph of bridges.
\begin{proof}
	We consider the $4$-partite $4$-uniform hypergraph $B=(V_1\dcup V_2\dcup V_3\dcup V_4,E_B)$ 
	whose vertex classes are four disjoint copies of $V$ and whose
    edges $v_1v_2v_3v_4\in E_B$ with $v_i\in V_i$ for $i\in[4]$
	correspond to $\zeta$-bridges $(v_1,v_2,v_3,v_4)$ of $H$.
	By Lemma~\ref{NB4} and our choice of $\zeta$, there are at least 
	\[
		\Big(\frac{1}{9}-7\zeta\Big)n^4
		>
		\frac{n^4}{18}
	\] 
	$\zeta$-bridges in $H$ 
	and, hence, $|E_B|\geq n^4/18$ edges. Similar as in the proof of Lemma~\ref{b6}, this implies  
	that there are at least $2\xi'' n^{11}$ copies of the complete $4$-partite $4$-uniform 
	hypergraph~$K^{(4)}_{3,3,3,2}$ in $B$ for some universal constant $\xi''>0$.
	Passing through the vertices of each such copy of $K^{(4)}_{3,3,3,2}$ (by starting in a vertex in~$V_1$ 
	and then passing cyclically through the other vertex classes) leads to a $4$-uniform path 
	$u_1\dots u_4x_1\dots x_4w_1w_2w_3$ in~$B$. In particular, $x_1\dots x_4$ is an edge in~$B$, and owing to the completeness of 
	$K^{(4)}_{3,3,3,2}$ we see that after removing  the vertices $x_1,\dots,x_4$, the remaining vertices still form a 
	$4$-uniform path $u_1\dots u_4w_1w_2w_3$ in~$B$.
	
	By definition of $B$ every such path $u_1\dots u_4x_1\dots x_4w_1w_2w_3$ corresponds to a walk in~$H$.
	Consequently,~$H$ contains at least  $2\xi'' n^{11}-O(n^{10})\geq \xi'' n^{11}$ $11$-tuples $u_1\dots u_4x_1\dots x_4w_1w_2w_3$ 
	that satisfy part~\ref{it:Abs-E1} of Lemma~\ref{manyelves}. Moreover, recalling that edges of $B$ correspond to $\zeta$-bridges in~$H$
	it follows that $(u_1,u_2,u_3)$, $(w_1,w_2,w_3)$ are $\zeta$-connectable in $H$ for every such $11$-tuple, i.e., part~\ref{it:Abs-E2}
	holds as well.
	\end{proof}

Next we define the \emph{absorbers}, which will be the building blocks of the absorbing path~$P_A$ in 
Proposition~\ref{prop:absorbingP}.

\begin{dfn}\label{absorber} 
Given Setup~\ref{setup:2335}, $\zeta>0$,
and $\sa=(a_1,\dots,a_4)\in V^4$, we say that a tuple $(\sb_1,\dots,\sb_4,\su,\sx,\sw)\in V^{35}$
with 
\[
	\sb_i=(b_{i1},\dots,b_{i6})\ \text{for $i\in[4]$}\,,\
	\su=(u_1,\dots,u_4)\,,\ 
	\sx=(x_1,\dots,x_4)\,,\ 
	\tand\
	\sw=(w_1,w_2,w_3)\,,
\]
is an {\it $\sa$-absorber} in~$H$, if 
\begin{enumerate}[label=\alabel]
\item\label{it:350} all its~$35$ vertices are distinct and different from 
those in $\sa$,
\item\label{it:35a} $\sb_i$ satisfies properties~\ref{it:Abs-S1} and~\ref{it:Abs-S2} of Lemma~\ref{b6} for $a_i$ and $x_i$ for every $i\in[4]$,
\item\label{it:35b} and $(\su,\sx,\sw)$ satisfies properties \ref{it:Abs-E1} and~\ref{it:Abs-E2} of Lemma~\ref{manyelves}.
\end{enumerate}
\end{dfn}
Formally, an $\sa$-absorber is defined to be a septuple. However, since it consists of $35$ vertices we may refer to it sometimes
as a $35$-tuple from $V^{35}$. Similarly, in part~\ref{it:35b} we refer to $(\su,\sx,\sw)$ as an $11$-tuple.

We note that if $\sa=(a_1,\dots,a_4)$ consists of four distinct vertices, then an $\sa$-absorber can be used 
to absorb the set $\{a_1,\dots,a_4\}$ as follows (see Figure~\ref{fig:absorberafter}).
The $35$ vertices of an $\sa$-absorber 
$(\sb_1,\dots,\sb_4,\su,\sx,\sw)$ can be partitioned into five $4$-uniform paths  
\[
	b_{i1}b_{i2}b_{i3}x_ib_{i4}b_{i5}b_{i6}\ \ \text{for $i\in[4]$}
	\qand 
	u_1\dots u_4w_1w_2w_3\,,
\]
in $H$, each of which starts and ends with a $\zeta$-connectable triple. 
If all five of these paths are segments (not necessarily consecutive) of
the absorbing path $P_A$, while all $a_1,a_2,a_3,a_4$ are not on $P_A$, 
then one can replace these five paths by
\[
	b_{i1}b_{i2}b_{i3}a_ib_{i4}b_{i5}b_{i6}\ \ \text{for $i\in[4]$}
	\qand
	u_1\dots u_4x_1\dots x_4w_1w_2w_3\,,
\]
i.e., replace $x_i$ with $a_i$ in the four ``$b$-paths'' and include $x_1,\dots,x_4$ in the middle of the fifth path.

Below we easily deduce from Lemmata~\ref{b6} and~\ref{manyelves} that there are $\Omega(n^{35})$ absorbers for every fixed 
$4$-tuple~$\sa\in V^4$. This fact will play a key r\^ole in the proof of the absorbing path lemma.
\begin{lemma}\label{AL}
	Given Setup~\ref{setup:2335}, there are constants $\zeta'_0=\zeta'_0(\alpha)$ and $\xi=\xi(\alpha)>0$
	such that for every $\zeta\in(0,\zeta'_0)$ and for every $\sa\in V^4$
	the number of $\sa$-absorbers in~$H$ is at least $\xi n^{35}$. 
\end{lemma}
\begin{proof}
	For a fixed $\alpha\in(0,1/3)$ let $\xi'(\alpha)>0$ and $\xi''>0$ be provided by Lemmata~\ref{b6} and~\ref{manyelves}.
	We set
	\[
		\zeta'_0=\min\Big\{\Big(\frac{\xi''}{23}\Big)^2\,,\ \frac{\alpha^2}{126}\Big\}
		\qqand
		\xi=\frac{1}{4}(\xi')^4\xi''
	\]
	and let $\zeta\in(0,\zeta'_0)$ and $\sa=(a_1,\dots,a_4)\in V^4$ be given. Moreover, let $X\subseteq V$ be the exceptional set 
	of vertices of size at most $\sqrt{\zeta}n$ given by Lemma~\ref{b6}.
	
	Lemma~\ref{manyelves} yields $\xi''n^{11}$ distinct $(\su,\sx,\sw)\in V^{11}$ with 
	$\su=(u_1,\dots,u_4)$, $\sx=(x_1,\dots,x_4)$, and $\sw=(w_1,w_2,w_3)$ satisfying 
	properties \ref{it:Abs-E1} and~\ref{it:Abs-E2} of the lemma. Obviously, at most 
	\[
		11(|X|+4)n^{10}
		\leq
		11\sqrt{\zeta}n^{11}+44n^{10}
	\]
	of these $11$-tuples
	share a vertex with $X\cup\{a_1,\dots,a_4\}$. Consequently, our choice of $\zeta'_0> \zeta$ guarantees that for sufficiently large $n$
	at least $\xi''n^{11}/2$ of these $11$-tuples are disjoint from~$X$ and from $\sa$. 
	
	Next, for such a fixed $11$-tuple $(\su,\sx,\sw)$ we apply 
	Lemma~\ref{b6} for every $i\in[4]$ to $a_i$ and~$x_i$. Each application yields $\xi'n^6$ sextuples $(b_{i1},\dots,b_{i6})$
	satisfying properties \ref{it:Abs-S1} and~\ref{it:Abs-S2} of that lemma. Taking into account that we insist that all 
	the vertices $b_{ij}$ for $i\in [4]$ and $j\in[6]$ need to be distinct and different from the already fixed vertices of
	$\sa$, $\sx$, $\su$, and $\sw$ this gives rise to at least 
	\[
		\frac{1}{2}\big(\xi'n^6\big)^4
	\]
	such choices of $\sb_1,\dots,\sb_4$ for every fixed $(\su,\sx,\sw)$. Summing over all possible choices of~$(\su,\sx,\sw)$
	leads to at least
	\[
		\frac{1}{2}\xi''n^{11} \times \frac{1}{2}\big(\xi'n^6\big)^4
		\geq \xi n^{35}
	\]
	$\sa$-absorbers in $H$.	
\end{proof}

After these preparations we conclude this section with the somewhat standard proof of the absorbing path lemma.
In the proof we first find a suitable selection of $\Omega(n)$ disjoint $35$-tuples that contain 
many $\sa$-absorbers for every~$\sa$. In the second and final step we then utilise the
$\zeta$-connectable end-triples to connect these $35$-tuples, 
each consisting of five disjoint paths of length  four, into one absorbing path avoiding the
given set~$\cR$.

\begin{proof}[Proof of Proposition~\ref{prop:absorbingP}]
	For $\alpha\in(0,1/3)$ let $\l\geq 3$ be given by Setup~\ref{setup:2335}
	and let $\zeta_0'=\zeta'_0(\alpha)$ and $\xi=\xi(\alpha)$ be given by Lemma~\ref{AL}. Set
	\[
		\zeta_0=\min\Big\{\zeta_0'\,,\ \frac{\xi}{12\cdot1400\l^2}\Big\}\,,
	\]
	and for $\zetas\in(0,\zeta_0)$ let $\thetas$ along with a sufficiently large $4$-uniform hypergraph 
	$H=(V,E)$ and $\cR\subseteq V$ of size at most $\thetas^2 n$ be given. Without loss of generality we 
	can assume that $\thetas<\zetas$.
	
	Applying Lemma~\ref{AL} with $\zetas$ yields for every $\sa\in V^4$ at least $\xi n^{35}$ $\sa$-absorbers in~$H$.
	However, since the absorbing path is required to be disjoint from $\cR$, only absorbers disjoint from $\cR$
	are useful here. Let $\cA(\sa)$ be the set of all $\sa$-absorbers disjoint from~$\cR$ and note
	\begin{equation}\label{eq:Aa}
		\big|\cA(\sa)\big|
		\geq
		\xi n^{35}-35|\cR|n^{34}
		\geq 
		\frac{\xi}{2}n^{35}\,.
	\end{equation}
	Let $\cA=\bigcup\cA(\sa)\subseteq (V\setminus\cR)^{35}$ 
	be the set of all absorbers outside $\cR$, where the union runs over all four tuples 
	$\sa\in V^{4}$.
 	
	We set
	\[
		p=\frac{4\zeta_0\thetas}{\xi n^{34}}
	\]
	and consider a random collection $\cA_p\subseteq \cA$, where every absorber from $\cA$ is included 
	independently with probability~$p$. Standard applications of Markov's inequality and 
	of Chernoff's inequality show that with positive probability the random 
	set~$\cA_p$ satisfies the following three properties 
	\begin{align}
		\big|\cA_p\big|
		&\leq 
		3\cdot pn^{35}\,,
		\label{eq:Ap1}\\
		\big|\{(A,A')\in\cA_p^2\colon A \tand A'\ \text{share a vertex}\}\big|
		&\leq 
		3\cdot35^2p^2n^{69}\,,\label{eq:Ap2}\\
		\text{and for every $\sa\in V^4$ we have}
		\quad
		\big|\cA_p\cap\cA(\sa)\big|
		&\geq
		\frac{1}{2}\cdot p\big|\cA(\sa)\big|\,.\label{eq:Ap3}
	\end{align}
	Consequently, there exists a collection $\cB_0\subseteq\cA$ satisfying~\eqref{eq:Ap1}--\eqref{eq:Ap3}
	with $\cB_0$ replacing~$\cA_p$. We further pass to a maximal subcollection $\cB\subseteq \cB_0$ of 
	mutually disjoint absorbers. The choices of~$p$ and $\zeta_0$ combined with~\eqref{eq:Aa} and 
	$\thetas<\zetas<\zeta_0$ allow us to transfer~\eqref{eq:Ap1} and~\eqref{eq:Ap3} 
	to the set $\cB$ as follows
	\begin{equation}\label{eq:B1}
		\big|\cB\big|
		\leq
		3\cdot pn^{35}
		=
		\frac{12\zeta_0\thetas}{\xi}n
		\leq\frac{\thetas}{1400\l^2}n
	\end{equation}
	and for every $\sa\in V^4$ we have
	\begin{equation}\label{eq:B2}
		\big|\cB\cap\cA(\sa)\big|
		\overset{\eqref{eq:Ap2}}{\geq}
		\frac{1}{2}\cdot p\big|\cA(\sa)\big|-3\cdot35^2p^2n^{69}
		\geq \zeta_0\thetas n
		-
		3\cdot35^2\frac{16\zeta_0^2\thetas^2}{\xi^2}n
		\geq 
		\frac{1}{2}\thetas^2n\,.
	\end{equation}
	
It remains to connect the absorbers from $\cB$ into a path. Recall that every $35$-tuple in $\cB$
consists of five $7$-vertex paths with $\zetas$-connectable end-triples and that all those $5|\cB|$ paths 
are mutually vertex disjoint. Let $\cP$ be the collection of all these $7$-vertex paths.

%Note that applying Proposition~\ref{lem:con} straightforwardly $|\cP|-1$ times easily
%yields a walk~$W$ containing all paths from $\cP$ and
%consisting of at most 
%\begin{equation}\label{eq:Psize}
%	|V(W)|
%	\leq
%	7\cdot |\cP|+(|\cP|-1)\cdot (8\l+10)
%	\overset{\l\geq 3}{\leq}
%	70\l|\cB|
%	\overset{\eqref{eq:B1}}{\leq}\frac{\thetas}{20\l} n
%\end{equation}
%vertices. For the proof of Proposition~\ref{prop:absorbingP} we have to make sure
%that $W$ can be indeed a path with all of its vertices outside of $\cR$. 

Finally we construct the absorbing path by connecting all paths from $\cP$.
For that we consider a maximal family of paths $\cP^{\star}\subseteq \cP$
for which there exists a path $P^{\star}_A\subseteq H-\cR$, containing all the paths from $\cP^{\star}$, 
whose end-triples are $\zetas$-connectable and such that
\begin{equation}\label{eq:sizePs}
	|V(P^{\star}_A)|
	=
	7\cdot|\cP^{\star}|+(|\cP^{\star}|-1)\cdot(8\l+10)
	\overset{\l\geq 3}{\leq}
	70\l|\cB|
	\overset{\eqref{eq:B1}}{\leq}
	\frac{\thetas}{20\l} n\,.
	%\overset{\eqref{eq:Psize}}{\leq}
	%\frac{\thetas}{20\ell} n\,.
\end{equation}

Clearly, $\cP^{\star}\ne\varnothing$ and thus,  $P^{\star}_A\ne\varnothing$. 
Assume for the sake of contradiction that there is some $P\in\cP\setminus\cP^{\star}$ 
and let $(x,y,z)$ be the starting triple of $P$. Moreover, let $(a,b,c)$ be the ending 
triple of $P^{\star}_A$.
Since both triples~$(a,b,c)$ and~$(x,y,z)$ are $\zeta$-connectable,
Proposition~\ref{lem:con} tells us that there are at least $\thetas n^{8\l+10}$
 $abc$-$xyz$-paths with $8\l+10$ inner vertices in $H$. 
Since~\eqref{eq:sizePs} combined with $|\cR|\leq \thetas^2n$ yields
\[
	\thetas n^{8\l+10}- (8\l+10)\big(\big|V(P^{\star}_A)\big|+\big|\cR\big|\big)n^{8\l+9}
	>0\,,
\]
there is at least one connecting path disjoint to $V(P^{\star}_A)\cup \cR$
giving rise to a path $P^{\star\star}_A\subseteq H-\cR$
containing~$\cP^{\star}\cup\{P\}$. This 
contradicts the maximality of $\cP^{\star}$ and consequently the
desired path $P_A$ containing all paths from $\cP$ does really exist. 

In fact, Property~\ref{it:PA1} of Proposition~\ref{prop:absorbingP} is a consequence of~\eqref{eq:sizePs}
and part~\ref{it:PA2} is also clear from the definition. For part~\ref{it:PA3} of Proposition~\ref{prop:absorbingP}, 
let $Z$ be a set outside $P_A$ of size at most $2\theta_\star^2n$ with 
$|Z|\equiv 0 \pmod 4$.  It follows from~\eqref{eq:B2}, 
that one can successively absorb quadruples of distinct vertices of $Z$ 
into the path, at least $\thetas^2n/2$ times,
always having at least one unused absorber at hand.
\end{proof}

\section{Path cover lemma}
\label{sec:long_path}

The goal of this section is to establish the following $4$-uniform variant
of Lemma~\ref{lem:2140}. 

\begin{prop}[Path cover lemma for $4$-uniform hypergraphs]\label{prop:1625}
	For every $\alpha\in (0, 1/4)$ there is a constant $\theta_0(\alpha)>0$
	such that for every positive $\theta_\star<\theta_0(\alpha)$
	there are a constant $\zeta_{\star\star}=\zeta_{\star\star}(\alpha, \theta_\star)>0$ 
	and arbitrarily large natural numbers $M$ with $M\equiv 3\pmod{4}$ such that the 
	following holds.  

	Given Setup~\ref{setup:2335} and a set $X\subseteq V$ with $|X|\le 2\theta_\star n$
	we can cover all but at most $\theta_\star^2 n$ vertices of $H-X$ by disjoint $M$-vertex
	paths that start and end with a $\zeta_{\star\star}$-connectable triple. 
\end{prop}

We would like to remark that the constants in this statement can be thought of as 
forming a hierarchy $\alpha\gg \theta_\star\gg \zeta_{\star\star}\gg M^{-1}\gg n^{-1}$.
In our intended application, the set $X$ will be the union of the reservoir and the vertex
set of the absorbing path. Moreover, it will be important that we have the liberty to take
$M$ to be substantially larger than the reciprocal of a further constant $\theta_{\star\star}$ 
obtained by applying the connecting lemma to $\zeta_{\star\star}$. 
%Finally, we require the vertex set of the hypergraph under discussion to be very large 
%in comparison to $M$. 
  
\begin{proof}
	Recall that Setup~\ref{setup:2335} involves a constant $\beta>0$ as well as a natural 
	number $\ell\ge 3$. We will assume throughout 
	that $\alpha, \beta, \ell^{-1}\gg \theta_\star\gg \zeta_{\star\star}$ without 
	calculating these dependencies explicitly.
	Let $P\subseteq 3\NN$ be the infinite arithmetic progression which the $3$-uniform
	Proposition~\ref{prop:1742} delivers for $\alpha/4$ and $\zeta_{\star\star}$ 
	here in place of $\alpha$ and $\xi$ there. Now let $M\gg \zeta_{\star\star}^{-1}$ be a 
	sufficiently large natural number with 
	$M\equiv 3\pmod{4}$ and $\frac 34(M+1)\in P$. 
	The number $M$ will play two different r\^{o}les and hoping to enhance the visibility
	of this fact we set $m=M$.  
	
	Now let a $4$-uniform hypergraph $H=(V, E)$ on $n\gg M$ vertices satisfying the minimum 
	pair degree condition $\delta_2(H)\ge \bigl(\frac59+\alpha\bigr)\frac{n^2}{2}$
	as well as a family $\bigl\{R_{uv}\colon uv\in V^{(2)}\bigr\}$ of robust subgraphs 
	of its link graphs exemplifying Setup~\ref{setup:2335} be given. Set 
	\[
		\ccP=\bigl\{P\subseteq H-X\colon P \text{ is an $M$-vertex path
			whose end-triples are $\zeta_{\star\star}$-connectable}\bigr\}
	\]
	and consider a maximum collection $\ccC\subseteq \ccP$ of vertex-disjoint paths.
	We are to establish that the set 
	\[
		U=V\setminus \Bigl(X\cup \bigcup\nolimits_{C\in \ccC}V(C)\Bigr)
	\]
	of uncovered vertices satisfies
	\begin{equation}\label{eq:U-small}
		|U| \le \theta_\star^2 n\,,
	\end{equation}
	so for the rest of the proof we can assume that \eqref{eq:U-small} is false.
	Our strategy for obtaining a contradiction is that we 
	find up to $m$ appropriate paths in $\ccC$ and show that the union of their 
	vertex sets with $U$ spans at least $m+1$ vertex-disjoint paths from $\ccP$.
	For a vertex $u\in U$ to have some chances to participate in this rerouting
	its link hypergraph should be somewhat ``typical'' and our next step is to 
	identify a set $\Ubad\subseteq U$ of bad vertices which we will not use for 
	incrementing $\ccC$.
	
	Recall that, as discussed between Setup~\ref{setup:2335} and 
	Definition~\ref{def:connectable4}, for every $u\in U$ the link hypergraph
	$\overline{H}_u$ and the family $\bigl\{R_{uv}\colon v\in V\setminus\{u\}\bigr\}$
	of $(\beta, \ell)$-robust graphs realise Setup~\ref{setup:1746}.
	Due to $\theta_{\star}\ll \ell^{-1}, \beta, \alpha$ and our assumption
	$|X|\le 2\theta_\star n$ it follows that the hypergraph $\overline{H}_u-X$
	and the family 
	\[
		\Psi=\bigl\{R_{uv}-X\colon v\in V\setminus (X\cup\{u\})\bigr\}
	\]
	of $(\beta/2, \ell)$-robust graphs exemplify Setup~\ref{setup:1746} with 
	$(\alpha/2, \beta/2, \alpha^3/9)$ here in place of $(\alpha, \beta, \mu)$ there. 
	In particular, 
	we can speak of $\zeta_{\star\star}$-connectable pairs and $\zeta_{\star\star}$-bridges
	with respect to the constellation $(\overline{H}_u-X, \Psi)$ and in the sequel we shall call 
	them {\it $(\zeta_{\star\star}, X)$-connectable pairs in $\overline{H}_u-X$} 
	and {\it $(\zeta_{\star\star}, X)$-bridges in $\overline{H}_u-X$}, respectively.   
	To clarify the relation between these concepts, we remark that 
	every $(2\zeta_{\star\star}, X)$-connectable pair of distinct vertices from $H_u-X$
	is, in particular $\zeta_{\star\star}$-connectable in $H_u$. Consequently, every
	$(2\zeta_{\star\star}, X)$-bridge in $\overline{H}_u-X$ is a 
	$\zeta_{\star\star}$-bridge in $\overline{H}_u$. 
	Now the vertices that we will not touch while refuting the maximality of $\ccC$ are those in 
	the set
	\begin{multline*}
		\Ubad = \bigl\{u\in U\colon \text{the number 
			of $(2\zeta_{\star\star}, X)$-bridges in $\overline{H}_u-X$
			which are } \\
			\text{ $\zeta_{\star\star}$-connectable in $H$ is at most $n^3/8$}\bigr\}\,.
	\end{multline*}

	\begin{claim}\label{clm:1717}
		We have $|\Ubad| \le 8\zeta_{\star\star} n$.
	\end{claim}
	
	\begin{proof}
		Consider the set 
		\begin{multline*}
			\Pi=\bigl\{(u, e)\in \Ubad\times V^3\colon \text{the triple $e$ is a 
					$(2\zeta_{\star\star}, X)$-bridge in $\overline{H}_u-X$,}\\
					\text{ but not $\zeta_{\star\star}$-connectable in $H$}\bigr\}\,.
		\end{multline*}
		Since for $u\in\Ubad$ every $(2\zeta_{\star\star}, X)$-bridge in $\overline{H}_u-X$
		is a $\zeta_{\star\star}$ bridge in $\overline{H}_u$, Lemma~\ref{F41analog} tells us that
		\[
			|\Pi|\le \zeta_{\star\star} n^4\,.
		\]

		On the other hand, for every $u\in\Ubad$ the number 
		of $(2\zeta_{\star\star}, X)$-bridges in $\overline{H}_u-X$ is 
		at least $(n-|X|)^3/3$ by Lemma~\ref{NB3} and by the 
		definition of $\Ubad$ all but at most $n^3/8$ of them fail to be 
		$\zeta_{\star\star}$-connectable in $H$, whence 
		\[
			|\Pi|
			\ge 
			|\Ubad|\left(\frac{(n-|X|)^3}{3}-\frac{n^3}{8}\right)
			\ge 
			\frac{|\Ubad| n^3}8\,.
		\]
		Comparing our estimates on $|\Pi|$ we obtain indeed 
		that $|\Ubad| \le 8\zeta_{\star\star} n$.
	\end{proof}

	\subsection*{Useful societies} Denote the vertex sets of the paths in our 
	maximum collection $\ccC$ by~$B_1, \ldots, B_{|\ccC|}$ and fix an arbitrary  partition 
	\[
		U
		=
		B_{|\ccC|+1}\dcup \ldots \dcup B_{\nu}\dcup B'
	\]
	with
	\[
		|B_{|\ccC|+1}| = \ldots = |B_\nu| = M > |B'|\,.
	\]
	The sets belonging to the family
	\[
		\ccB=\{B_1, \ldots, B_\nu\}
	\]
	will be referred to as \textit{blocks}. The size of their union
	\begin{equation*}
		B=B_1\dcup B_2\dcup \ldots \dcup B_\nu
	\end{equation*}
	is bounded from below by
	\begin{equation}\label{eq:Mnun}
		|B|= n-|X|-|B'|\ge (1-2\theta_\star)n -M \ge (1-3\theta_\star)n\,.
	\end{equation}
	By a {\it society} we mean a set consisting of $m$ blocks and we shall 
	write $\gS$ for the collection of all $\binom{\nu}{m}$ societies. 
	
	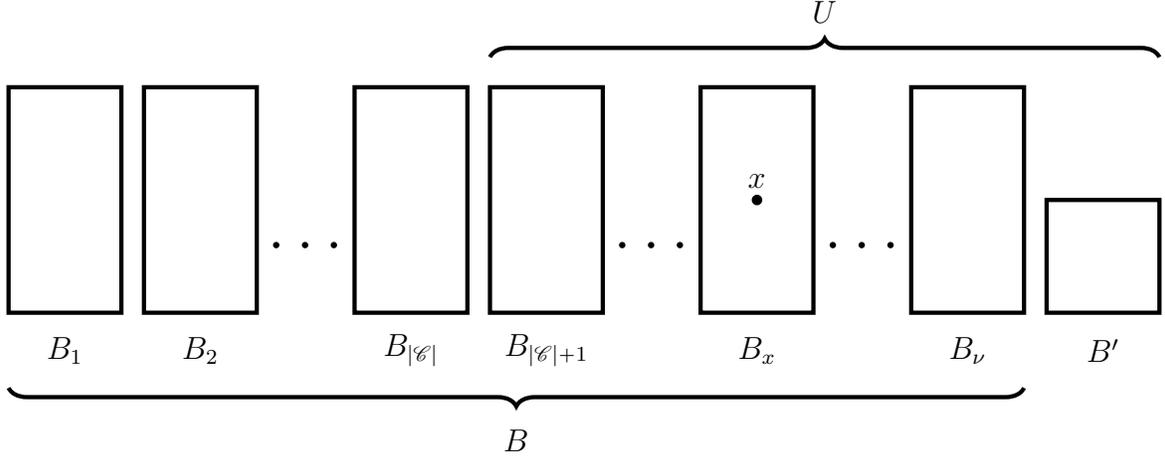
\begin{figure}
	\begin{tikzpicture}[scale=1]

	\def \w{1.5};
	\def \h{3}; 
	\def \ms{.3};
	\def \bs{1.3};  
	
	\coordinate (b1) at (-8,0);
	\coordinate (b2) at ($(b1) + (\w,0) + (\ms,0)$);
	\coordinate (b3) at ($(b2) + (\w,0) + (\bs,0)$);
	\coordinate (b4) at ($(b3) + (\w,0) + (\ms,0)$);
	\coordinate (b5) at ($(b4) + (\w,0) + (\bs,0)$);
	\coordinate (b6) at ($(b5) + (\w,0) + (\bs,0)$);
	\coordinate (b7) at ($(b6) + (\w,0) + (\ms,0)$);
	
	\coordinate (x) at ($(b5)+(.5*\w,.5*\h)$);
	
	\foreach \i in {1,...,6}{
		\draw [ultra thick] (b\i) rectangle +(\w,\h);
		}
	
		\draw [ultra thick] (b7) rectangle +(\w,.5*\h);

			\fill (x) circle (2pt);
			\node at ($(x)+(0,.25)$) {$x$};
			
			\node at ($(b1) + (.5*\w,-.5)$) {$B_1$};
			\node at ($(b2) + (.5*\w,-.5)$) {$B_2$};
			\node at ($(b3) + (.5*\w,-.5)$) {$B_{|\mathscr{C}|}$};
			\node at ($(b4) + (.5*\w,-.5)$) {$B_{|\mathscr{C}|+1}$};
			\node at ($(b5) + (.5*\w,-.5)$) {$B_x$};
			\node at ($(b6) + (.5*\w,-.5)$) {$B_\nu$};
			\node at ($(b7) + (.5*\w,-.5)$) {$B'$};
			
			\node at ($($(b2)+(\w,0)$)!.55!(b3)+(0,.3*\h)$) {{\Huge $\dots$}};
			\node at ($($(b4)+(\w,0)$)!.55!(b5)+(0,.3*\h)$) {{\Huge $\dots$}};
			\node at ($($(b5)+(\w,0)$)!.55!(b6)+(0,.3*\h)$) {{\Huge $\dots$}}; 
	
		\draw[ultra thick, decorate,decoration={brace,amplitude=7pt,mirror}] 
			($(b1)-(0,1)$) coordinate (a)  -- ($(b6)+(\w,-1)$) coordinate (c) ; 
			
			\draw[ultra thick, decorate,decoration={brace,amplitude=7pt}] 
			($(b4)+(0,\h)+(0,.4)$) coordinate (aa)  -- ($(b7)+(\w,\h)+(0,.4)$) coordinate (ca) ; 
	
	\node at ($($(b1)-(0,1)$)!.5! ($(b6)+(\w,-1)$)-(0,.7)$){$B$};
	\node at ($($(b4)+(0,\h)$) !.5!($(b7)+(\w,\h)$)+(0,1)$){$U$};

	\end{tikzpicture}
	\caption{Block partition of $V\setminus X$ for given $\ccC$.}
	\label{fig:block}
\end{figure}
	
	\begin{dfn}\label{dfn:1754}
		A society $\cS\in\gS$ with $S=\bigcup\cS$ is {\it useful}
		for a vertex $u\in U\setminus S$ if 
		\begin{enumerate}[label=\rmlabel]
				\item\label{it:1754a} $\delta_1 \left(H_u\left[ S\right]\right) 
						\ge\left(\frac{5}{9}+\frac\alpha 4\right)\frac{(mM) ^2}{2}$,
				\item\label{it:1754b} the family of graphs $\bigl\{R_{ux}[S]\colon x\in S\bigr\}$ 
					exemplifies Setup~\ref{setup:1746} for~$H_{u}[S]$ 
					with $(\alpha/4, \beta/2, \alpha/16)$
					here in place of $(\alpha, \beta, \mu)$ there,
				\item\label{it:1754c} and there are at least $\zeta_{\star\star} m^3M^3$ 
					triples in $S^3$ that are $\zeta_{\star\star}$-connectable in $H$ and
					$\zeta_{\star\star}$-bridges in $H_u[S]$ with respect to the robust 
					graphs in $\bigl\{R_{ux}[S]\colon x\in S\bigr\}$.
		\end{enumerate}
	\end{dfn}
	
	We shall argue that if a society $\cS$ is useful for many vertices in $U$, then 
	$U\cup\bigcup\cS$ spans $m+1$ disjoint paths from $\ccP$, contrary to the maximality
	of $|\ccC|$. The following claim provides a first step in this direction. 
	
 	\begin{claim}\label{clm:2053}
		If a society $\cS \in \gS$ is useful for a vertex $u\in U$ and $S=\bigcup\cS$, 
		then there exist $m+1$ vertex-disjoint $3$-uniform paths in $H_u[S]$ each of which has 
		$\frac{3}{4}(M+1)$ vertices and starts and ends with a triple which 
		is $\zeta_{\star\star}$-connectable in $H$.
	\end{claim}
	
	\begin{proof}
		We can apply Proposition~\ref{prop:1742} with $(\alpha/4, \zeta_{\star\star})$
		here in place of $(\alpha, \xi)$ there to the hypergraph $H_u[S]$,
		the family $\bigl\{R_{ux}[S]\colon x\in S\bigr\}$ of robust graphs, the set 
		\[
			\Xi=\bigl\{e\in S^3\colon \text{$e$ is a $\zeta_{\star\star}$-bridge in $H_u[S]$
					and $\zeta_{\star\star}$-connectable in $H$}\bigr\}
		\]
		of bridges and to $\frac34(M+1)$ here in place of $M$ there. This yields a 
		collection $\ccW$ of vertex-disjoint $3$-uniform $\frac34(M+1)$-vertex paths 
		in $H_u[S]$ with 
		\[
			\Big|S\setminus\bigcup\nolimits_{W\in\ccW}V(W)\Big|
			\le 
			\zeta_{\star\star}Mm+M
		\]
		such that every path in $\ccW$ starts and ends with a triple from $\Xi$.  
		In particular, the paths in $\ccW$ start and end with triples which are 
		$\zeta_{\star\star}$-connectable in $H$. 
		It remains to show $|\ccW|\ge m+1$, which follows from the fact that due 
		to $M=m\ge 15$ we have 
		\[
			|\ccW|
			\ge 
			\frac{(1-\zeta_{\star\star})Mm-M}{\frac34(M+1)}
			\ge 
			\frac{\frac9{10}Mm-M}{\frac45M}
			=
			\frac98m-\frac54
			>m\,. \qedhere
		\]
	\end{proof}
	
	To conclude the proof of Proposition~\ref{prop:1625} we need another result on useful 
	societies whose proof we postpone. 
	 
	\begin{claim}\label{clm:2042}
 		For every $u\in U\setminus U_{bad}$ there are at 
		least $\frac{2}{3}\left\vert \mathfrak{S}\right\vert$ useful societies.
		% that are  for $u$.
	\end{claim}	
	
	Since we assume that~\eqref{eq:U-small} is false, Claim~\ref{clm:1717} yields 
	\[
		|U\setminus \Ubad|\ge (\theta_\star^2-8\zeta_{\star\star})n\ge \frac{\theta_\star^2}{2}n\,.
	\]
	By Claim~\ref{clm:2042} and double counting there exists a society $\cS$ which is useful
	for at least $\frac23|U\setminus \Ubad|\ge \theta_\star^2n/3$ vertices from $U$.  
	Next, Claim~\ref{clm:2053} allows us to choose for every such vertex $u$ a 
	collection $\ccW_u$ of $m+1$ paths in $H_u[S]$ each of which consists of $\frac34(M+1)$
	vertices and starts and ends with a triple that is $\zeta_{\star\star}$-connectable in $H$. 
	As there are no more than $(Mm)!$ possibilities for $\ccW_u$, there exist a collection 
	$\ccW$ of $3$-uniform paths on $S$ and a set $U'\subseteq U$ such that 
	$\ccW_u=\ccW$ for every $u\in U'$ and 
	\[
		 |U'|
		 \ge  
		 \frac{\theta_\star^2n}{3(Mm)!}\ge \frac{1}{4}(M-3)(m+1)\,,
	\]
	where the second inequality uses $n\gg M=m$. Now we augment every path in $\ccW$
	by inserting $(M-3)/4$ vertices from $U'$ in every fourth position 
	(see Figure~\ref{fig:aug}),
	thus obtaining~$m+1$ mutually disjoint $4$-uniform $M$-vertex paths. 
	As the $m+1$ paths obtained 
	in this way start and end with $\zeta_{\star\star}$-connectable triples, 
	the new paths are in $\ccP$. Thus, if we remove from $\ccC$ the paths whose 
	vertex sets belong to the useful society $\cS$ and add the newly constructed paths, 
	we obtain a collection of paths contradicting the maximality of $\ccC$. This 
	contradiction proves the validity of~\eqref{eq:U-small} and, hence, 
	concludes the proof of Proposition~\ref{prop:1625} based on Claim~\ref{clm:2042}.
	
	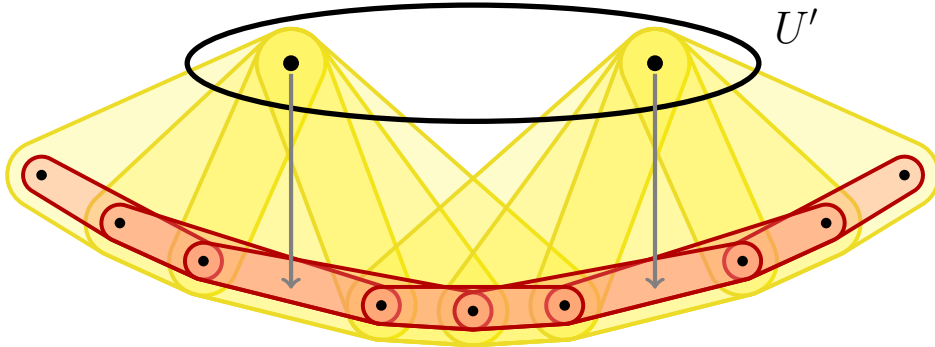
\begin{figure}[h!]
	\begin{tikzpicture}[scale=1]
		
	\def\an{3.5};
	\def\s{270-10*\an};
	\def\ra{10cm};
	
	\foreach \i in {0,...,10}{
		\coordinate (a\i) at (\s+2*\i*\an:\ra);
	}
	
	\coordinate (u1) at ($(a3) + (0,3)$);
	\coordinate (u2) at ($(a7) + (0,3)$);
				
	\redge{(a2)}{(a1)}{(a0)}{(u1)}{13pt}{1.5pt}{yellow!90!black}{yellow,opacity=0.2}
	\redge{(a4)}{(a2)}{(a1)}{(u1)}{13pt}{1.5pt}{yellow!90!black}{yellow,opacity=0.2};
	\redge{(a5)}{(a4)}{(a2)}{(u1)}{13pt}{1.5pt}{yellow!90!black}{yellow,opacity=0.2};
	\redge{(a6)}{(a5)}{(a4)}{(u1)}{13pt}{1.5pt}{yellow!90!black}{yellow,opacity=0.2};			
				
	\redge{(a6)}{(a5)}{(a4)}{(u2)}{13pt}{1.5pt}{yellow!90!black}{yellow,opacity=0.2};
	\redge{(a8)}{(a6)}{(a5)}{(u2)}{13pt}{1.5pt}{yellow!90!black}{yellow,opacity=0.2}
	\redge{(a9)}{(a8)}{(a6)}{(u2)}{13pt}{1.5pt}{yellow!90!black}{yellow,opacity=0.2};
	\redge{(a10)}{(a9)}{(a8)}{(u2)}{13pt}{1.5pt}{yellow!90!black}{yellow,opacity=0.2};
	
	\qedge{(a2)}{(a1)}{(a0)}{7pt}{1.5pt}{red!70!black}{red!50!white,opacity=0.3};
	\qedge{(a10)}{(a9)}{(a8)}{7pt}{1.5pt}{red!70!black}{red!50!white,opacity=0.3};
	
	\qedge{(a4)}{(a2)}{(a1)}{7pt}{1.5pt}{red!70!black}{red!50!white,opacity=0.3};
	\qedge{(a9)}{(a8)}{(a6)}{7pt}{1.5pt}{red!70!black}{red!50!white,opacity=0.3};
	
	\qedge{(a5)}{(a4)}{(a2)}{7pt}{1.5pt}{red!70!black}{red!50!white,opacity=0.3};
	\qedge{(a8)}{(a6)}{(a5)}{7pt}{1.5pt}{red!70!black}{red!50!white,opacity=0.3};
	
	\qedge{(a6)}{(a5)}{(a4)}{7pt}{1.5pt}{red!70!black}{red!50!white,opacity=0.3};				

	\node at (4.3,-6.2) {\Large $U'$};
	
	\draw[black, line width=2pt] ($(u1)!.5!(u2)$) ellipse (3.8cm and 22pt);	
	
	\draw [ultra thick, black!50!white, shorten <= 4pt, ->] (u1) -- (a3); 
	\draw [ultra thick, black!50!white, shorten <= 4pt, ->] (u2) -- (a7); 				
				
	\foreach \i in {1,2}{
		\fill (u\i) circle (3pt);
	}			
	
	\foreach \i in {0,1,2,4,5,6,8,9,10}{
		\fill (a\i) circle (2pt);
	}

	\end{tikzpicture}
	\caption {Augmenting a $\frac 34 (M+1)$-vertex $3$-uniform path to an $M$-vertex $4$-uniform path.}
	\label{fig:aug}
\end{figure}

	\subsection*{Proof of Claim~\ref{clm:2042}.}
	Fix a vertex $u\in U\setminus \Ubad$. We shall prove that the probability that 
	a society $\cS\in\gS$ chosen uniformly at random fails to be useful for $u$ is 
	$\exp\bigl(-\Omega(m)\bigr)$, where the implicit constant only depends 
	on $\alpha$, $\beta$, $\ell$, $\theta_\star$ and $\zeta_{\star\star}$. 
	So a sufficiently large choice of~$M=m$ allows us to push this probability below $1/3$, 
	as desired. 
	
	We will apply Corollary~\ref{cor:1041} several times to 
	the partition 
	\[
		V=B_1\dcup \ldots \dcup B_{\nu} \dcup (B'\dcup X)
	\]
	or to a partition derived from it by relocating up to three of the blocks $B_1, \ldots, B_\nu$
	to the exceptional set. By~\eqref{eq:Mnun} we can take~$\eta=4\theta_\star$ in all these applications.
	It will be convenient to write $B_x$ for the block containing a vertex $x\in B$. 
	
	We begin by estimating the probability of the unfortunate event $\gE_1$ that the minimum 
	vertex-degree condition in Definition~\ref{dfn:1754}\,\ref{it:1754a} fails for our random 
	society, i.e.,
	\[
		\gE_1=\left\{\cS\in \gS\colon 
			\delta_1(H_u[S]) < \left(\frac{5}{9}+\frac\alpha 4\right)\frac{M^2m^2}{2}\right\}\,.
	\]
	Since $u$ is isolated in $H_u$, this event occurs whenever $u\in S$ and we have 
	\begin{align}
		\PP(\gE_1)& \le \PP(u\in S) \notag \\
		&+
		\sum_{x\in B\setminus\{u\}} 
		\PP(x\in S)\,
		\PP\left(e_{H_{ux}}(S\setminus B_x) <
			\left(\frac59+\frac\alpha2\right)\frac{M^2(m-1)^2}{2}\copr x\in S\right)\,, \label{eq:1805}
	\end{align}
	where the reason for excluding the set $B_x$ is that conditioned on $x\in S$ 
	the random variable $e_{H_{ux}}(S\setminus B_x)$ is more pleasant to work with 
	than $e_{H_{ux}}(S)$.
	For a fixed vertex $x\in B\setminus \{u\}$ we want to derive an upper bound 
	on the probability summed in~\eqref{eq:1805} by applying Corollary~\ref{cor:1041}\,\ref{it:1041b}
	with $k=2$ to the graph $H_{ux}$. Our assumption on $H$ yields  
	\[
		e(H_{ux})\ge \left(\frac 59+\alpha \right)\frac{n^2}2
	\]
	and given the event $x\in S$, or equivalently $B_x\in \cS$, 
	the variable $e_{H_{ux}}(S\setminus B_x)$ is determined by a random selection of $m-1$
	blocks from $\ccB\setminus\{B_x\}$. So by Corollary~\ref{cor:1041}\,\ref{it:1041b} with 
	$m-1$ in place of $m$ and $\xi=\alpha/2$ we obtain
	\[
		\PP\left(e_{H_{ux}}(S\setminus B_x) <
			\left(\frac59+\frac\alpha2\right)\frac{M^2(m-1)^2}{2}\copr x\in S \right)
		\le
		\exp\bigl(-\Omega(m)\bigr)\,.
	\]
	Together with~\eqref{eq:1805} this yields 
	\[
		\PP(\gE_1)
		\le 
		\frac{m}{\nu}+ \Bigl(\sum_{x\in B\setminus\{u\}} \PP(x\in S)\Bigr)\exp\bigl(-\Omega(m)\bigr)
		\le 
		\frac{m}{\nu} + Mm \exp\bigl(-\Omega(m)\bigr)
	\]
	and for sufficiently large $n\gg M=m$ this shows that 
	\begin{equation}\label{eq:2255}
		 \PP(\gE_1) \le \exp\bigl(-\Omega(m)\bigr)\,.
	\end{equation}

	Proceeding with the second item in Definition~\ref{dfn:1754} we let $\gE_2$ be the bad 
	event that, for our fixed vertex $u\in U\setminus \Ubad$, the family of 
	graphs $\bigl\{R_{ux}[S]\colon x\in S\bigr\}$ fails to 
	exemplify Setup~\ref{setup:1746} for~$H_{u}[S]$ with $(\alpha/4, \beta/2, \alpha/16)$ 
	here in place of $(\alpha, \beta, \mu)$ there.
	We analyse~$\gE_2$ by considering for every fixed $x\in B\setminus \{u\}$ the event 
	$\gE'_2(x)$ that $R_{ux}[S]$ fails to be $(\beta/2, \ell)$-robust and the event $\gE''_2(x)$
	that one of the estimates required by Setup~\ref{setup:1746} fails. Observe that in the 
	present context these estimates read as follows:
	\begin{enumerate}[itemsep=.7em]
			\item[$\bullet$] $|V(R_{ux})\cap S|\ge \bigl(\frac23+\frac\alpha 8\bigr)Mm$,
			\item[$\bullet$] $e_{H_{ux}}\big(V(R_{ux})\cap S, S\setminus V(R_{ux})\big)\le 
				\frac 1{16}\alpha M^2m^2$, and 		
			\item[$\bullet$] $e(R_{ux})\ge 
				\left(\frac{5}{9}+\frac{\alpha}{8}\right)\frac{M^2m^2}{2}
					-\frac{(|S\setminus V(R_{ux})|)^2}{2}
				\ge\left(\frac49+\frac16\alpha\right)\frac{M^2m^2}{2}$.
	\end{enumerate}
	
	Consider a fixed vertex $x\in B\setminus\{u\}$. For any two distinct 
	vertices $y, z\in V(R_{ux})$ we let $P_{xyz}\subseteq V^{\ell-1}$ 
	be a set of $(\ell-1)$-tuples encoding the inner vertices of the $\ell$-edge 
	paths in~$R_{ux}$ from~$y$ to~$z$ and we let $\gP_{xyz}$ be the event that
	$|P_{xyz}\cap S^{\ell-1}|\le \frac12\beta|V(R_{ux})\cap S|^{\ell-1}$.
	By the law of total probability we have 
	\begin{equation}\label{eq:2120}
		\PP(\gE'_2(x) \coprn x\in S)
		\le 
		\sum_{yz\in V(R_{ux})^{(2)}} \PP(y, z\in S)\, \PP(\gP_{xyz} \coprn x, y, z\in S)\,.
	\end{equation}

	Let us look at a fixed pair $yz\in V(R_{ux})^{(2)}$. Since $R_{ux}$ 
	is $(\beta, \ell)$-robust, we know that $|P_{xyz}|\ge \beta |V(R_{ux})|^{\ell-1}$
	and, therefore, the set 
	\[
		P'_{xyz}=P_{xyz}\cap \bigl(V\setminus (B_x\cup B_y\cup B_z)\bigr)^{\ell-1}
	\]
	corresponding to those paths in $P_{xyz}$ that avoid $B_x\cup B_y\cup B_z$ satisfies   
	\[
		|P'_{xyz}|\ge \frac56\beta (\rho n)^{\ell-1}\,,
	\]
	where $\rho=|V(R_{ux})|/n>2/3$. For $d=\big|\{B_x, B_y, B_z\}\big|\in [3]$
	we deduce from Corollary~\ref{cor:1041}\,\ref{it:1041a} (by moving $B_x$, $B_y$, and $B_z$
	into the exceptional set) that 
	\begin{multline}\label{eq:2203}
		\PP\bigl(|P'_{xyz}\cap S^{\ell-1}|\le \tfrac23\beta (\rho Mm)^{\ell-1} 
			\bcopr x, y, z\in S\bigr)\\
		\le
		\PP\bigl(|P'_{xyz}\cap S^{\ell-1}|\le \tfrac34\beta \bigl(\rho M(m-d)\bigr)^{\ell-1} 
			\bcopr x, y, z\in S\bigr)
		\le
		\exp\bigl(-\Omega(m)\bigr)\,. 
	\end{multline}
	Similarly, Corollary~\ref{cor:1041}\,\ref{it:1041a}
	applied with $k=1$ to the set $\Lambda=V(R_{ux})\setminus (B_x\cup B_y\cup B_z)$ yields 
	\[
			\PP\left(\bigg|\frac{|\Lambda\cap S|}{M(m-d)}-\frac{|\Lambda|}n\bigg|
				\ge \frac 1{9\ell} \copr x, y, z\in S \right) 
		\le 
		\exp\bigl(-\Omega(m)\bigr)\,.
	\]
	In particular, the random variable $\rho_S=|V(R_{ux})\cap S|/(Mm)$ satisfies 
	\begin{equation}\label{eq:2213}
		\PP\left(|\rho_S-\rho|\ge \frac1{8\ell}\copr x, y, z\in S\right)
		\le 
		\exp\bigl(-\Omega(m)\bigr)\,.
	\end{equation}
	Now if both of the likely events 
	$|P'_{xyz}\cap S^{\ell-1}| > \tfrac23\beta (\rho Mm)^{\ell-1}$
	and $|\rho-\rho_S|<1/(8\ell)$ hold, then $\rho_S>1/2$ and  

	\begin{equation}\label{eq:1450}
		\frac{|P'_{xyz}\cap S^{\ell-1}|}{\beta (Mm)^{\ell-1}}
		>
		\frac 23\rho^{\ell-1}
		>
		\frac 23\left(\rho_S-\frac 1{8\ell}\right)^{\ell-1}
		\ge
		\frac 23\rho_S^{\ell-1}\left(1-\frac 1{4\ell}\right)^{\ell}
		\ge
		\frac 12\rho_S^{\ell-1}\,.
	\end{equation}
	Adding~\eqref{eq:2203} and~\eqref{eq:2213} we deduce from~\eqref{eq:1450} that
	\[
		\PP\bigl(|P'_{xyz}\cap S^{\ell-1}|\le \tfrac12\beta(\rho_S Mm)^{\ell-1}
			\bcopr x, y, z\in S\bigr)
		\le
		\exp\bigl(-\Omega(m)\bigr)\,,
	\]
	whence
	\[
		\PP\bigl(\gP_{xyz})\le \exp\bigl(-\Omega(m)\bigr)\,.
	\]

	As this holds for every pair $yz\in V(R_{ux})^{(2)}$ we conclude from~\eqref{eq:2120}
	that
	\[ 
		\PP(\gE'_2(x)\coprn x\in S)
		\le 
		\Bigl(\sum_{yz\in V(R_{ux})^{(2)}} \PP(y, z\in S)\Bigr) \exp\bigl(-\Omega(m)\bigr)
		\le 
		\binom{Mm}2\exp\bigl(-\Omega(m)\bigr) \,.
	\]%end{align*}
	Summarising the argument so far, we have proved 
	\[
		\PP\bigl(\gE'_2(x)\bcopr x\in S\bigr)\le \exp\bigl(-\Omega(m)\bigr)
	\]
	for every $x\in B\setminus \{u\}$. 
	Similar but easier considerations based on Corollary~\ref{cor:1041} show that 
	\[
		\PP\bigl(\gE''_2(x)\bcopr x\in S\bigr)\le \exp\bigl(-\Omega(m)\bigr)
	\]
	holds as well and we leave the details of this derivation to the reader. Returning 
	now to the event $\gE_2$ that the family $\bigl\{R_{ux}[S]\colon x\in S\bigr\}$ fails to 
	exemplify Setup~\ref{setup:1746} for~$H_{u}[S]$ with $(\alpha/4, \beta/2, \alpha/16)$ 
	here in place of $(\alpha, \beta, \mu)$ there we obtain
	\[
		\PP(\gE_2)
		\le 
		\PP(u\in S)+\sum_{x\in B\setminus\{u\}}\PP\bigl(\gE'_2(x)\cup\gE''_2(x)\bcopr x\in S\bigr)
		\le 
		\frac m\nu+Mm\exp\bigl(-\Omega(m)\bigr)\,,
	\]
	i.e.,
	\begin{equation}\label{eq:2257}
		 \PP(\gE_2) \le \exp\bigl(-\Omega(m)\bigr)\,.
	\end{equation}

	It remains to analyse the adverse event $\gE_3$ that the third clause of 
	Definition~\ref{dfn:1754} fails. 
	Consider any pair of vertices $yz\in (B\setminus\{u\})^2$ which is 
	$(2\zeta_{\star\star}, X)$-connectable in $\overline{H}_u-X$. Recall that 
	this means that a certain set $U_{yz}\subseteq V\setminus X$ of witnesses 
	definable from the family of robust 
	graphs $\bigl\{R_{uv}-X\colon v\in V\setminus (X\cup\{u\})\bigr\}$
	satisfies $|U_{yz}|\ge 2\zeta_{\star\star}|V\setminus X|$ and, 
	hence, $|U_{yz}\setminus (B_y\cup B_z)|\ge (3/2)\zeta_{\star\star}|V\setminus X|$. 
	Corollary~\ref{cor:1041}\,\ref{it:1041a} applied to $V\setminus X$, the block 
	partition with exceptional set $B'\cup B_y\cup B_z$,
	and with the constants 
	$\eta=\zeta_{\star\star}^2$, $\xi=\zeta_{\star\star}/4$
	shows that 
	\[
		\PP\bigl(|(U_{yz}\cap S)\setminus (B_y\cup B_z)|
			\le \zeta_{\star\star}Mm \bcopr y, z\in S \bigr)
		\le
		\exp\bigl(-\Omega(m)\bigr)\,.
	\]
	As this holds for every $(2\zeta_{\star\star}, X)$-connectable pair $yz$, it follows 
	in the usual way that 
	\begin{multline*}
		\PP(\neg \gE_2 \text{ and some $(2\zeta_{\star\star}, X)$-connectable pair
			belonging to $S^{(2)}$} \\
			\text{ is not $\zeta_{\star\star}$-connectable in $H_u[S]$}) 
		\le
		\exp\bigl(-\Omega(m)\bigr)\,,
	\end{multline*}
	where the reason for adding the conjunct $\neg \gE_2$ is that it makes the notion 
	of connectable pairs in $H_u[S]$ meaningful. Due to the definition of bridges in 
	terms of connectable pairs it follows that 
	\begin{multline}\label{eq:2246}
		\PP(\neg \gE_2 \text{ and some $(2\zeta_{\star\star}, X)$-bridge belonging to $S^3$} \\
		\text{ is not a $\zeta_{\star\star}$-bridge in $H_u[S]$}) 
		\le
		\exp\bigl(-\Omega(m)\bigr)\,.
	\end{multline}
	Since $u\not\in\Ubad$ the set 
	\[
		\Phi_u=\bigl\{e\in V^3\colon \text{$e$ is a $\zeta_{\star\star}$-connectable 
			$(2\zeta_{\star\star}, X)$-bridge}\bigr\}
	\]
	has size $|\Phi_u|\ge n^3/8$ and a final application of 
	Corollary~\ref{cor:1041}\,\ref{it:1041a} with $k=3$ 
	shows that this set scales appropriately to $S$ in the sense that 
	\[
		\PP\left(|\Phi_u\cap S^3|\le M^3m^3/16\right)
		\le
		\exp\bigl(-\Omega(m)\bigr)\,.
	\]
	Together with~\eqref{eq:2246} this proves 
	\[
		\PP(\neg \gE_2\,\,\&\,\,\gE_3)
		\le
		\exp\bigl(-\Omega(m)\bigr)
	\]
	and by adding~\eqref{eq:2255} as well as~\eqref{eq:2257} we finally obtain
	\[
		\PP(\text{$\cS$ is not useful for $u$})
		\le 
		\exp\bigl(-\Omega(m)\bigr)\,.
	\]
	This concludes the proof of Claim~\ref{clm:2042} and, hence, the proof of 
	Proposition~\ref{prop:1625}.
\end{proof}

\section{The proof of the main result}
\label{sec:main-pf}

In this section we give the routine derivation of Theorem~\ref{thm:main} from the 
results in \hbox{Sections~\ref{conn}\,--\,\ref{sec:long_path}}. 
\begin{proof}[Proof of Theorem~\ref{thm:main}] We can assume 
that $\alpha>0$ is sufficiently small. Now we choose 
an appropriate hierarchy of constants 
\[
	\alpha\gg \beta, \ell^{-1} \gg \zeta_{\star}\gg \theta_\star\gg \zeta_{\star\star} 
	\gg \theta_{\star\star} \gg M^{-1} \gg n_0^{-1}\,.
\]
We recall that Corollary~\ref{all4} yields four natural numbers 
$\ell_1, \ell_2, \ell_3, \ell_4\le 50\ell$.

Let $H=(V, E)$ be a $4$-uniform hypergraph on $|V|=n\ge n_0$ vertices satisfying the
minimum pair degree condition
$
	\delta_2(H)\ge\big(\frac59+\alpha\big)\frac{n^2}2\,.
$ 
We need to construct a Hamiltonian cycle in $H$. 
Appealing to Proposition~\ref{prop:robust} with $\mu=\alpha^3/18$ we choose
for every pair $uv\in V^{(2)}$ a $(\beta, \ell)$-robust subgraph $R_{uv}\subseteq H_{uv}$   
of its link graph. Notice that $H$ and the family of robust 
graphs 
\[
	\{R_{uv}\colon uv\in V^{(2)}\}
\]
realise Setup~\ref{setup:2335}. Proposition~\ref{prop:reservoir} allows us to choose a 
reservoir set $\cR$ with $|\cR|\le \theta_{\star}^2n$ which by 
Corollary~\ref{lem:use-reservoir} has the property that if a subset~$\cR'\subseteq \cR$
with $|\cR'|\le \theta_{\star\star}^2n$ has ``already been used'', then for every $i\in [4]$
we can still connect any two disjoint $\zeta_{\star\star}$-connectable triples 
by a path through $\cR\setminus \cR'$ having $\ell_i$ inner vertices.  
Next we apply Proposition~\ref{prop:absorbingP} to obtain  
an (absorbing) path $P_A\subseteq H-\cR$ such that
\begin{enumerate}[label=\rmlabel]
	\item\label{it:QA1} $ |V(P_A)|\le \theta_\star n$,
	\item\label{it:QA2} the end-triples of $P_A$ are $\zeta_\star$-connectable, 
	\item\label{it:QA3} and for every subset $Z\subseteq V\setminus V(P_A)$
		with $|Z|\le 2\theta_\star^2n$ and $|Z|\equiv0 \pmod 4$, there is a path $Q\subseteq H$ 
		with the same end-triples as $P_A$ and $V(Q)=V(P_A)\cup Z$.
\end{enumerate}
As the set $X=\cR\cup V(P_A)$ satisfies 
$|X|\le (\theta_\star+\theta_{\star}^2)n\le 2\theta_\star n$, Proposition~\ref{prop:1625}
yields a collection $\ccC$ of $M$-vertex paths starting and ending 
with $\zeta_{\star\star}$-connectable triples such that the set 
\[
	J=V\setminus \Bigl(V(P_A)\cup \cR\cup \bigcup\nolimits_{P\in \ccC}V(P)\Bigr)
\]
of uncovered vertices satisfies $|J|\le \theta_{\star}^2n$. 

Now we want to form an almost 
spanning cycle in $H$ by connecting the paths in $\ccC$ and~$P_A$ through the reservoir.
For each of the first $|\ccC|$ of these connections we want use $\ell_1$ vertices 
from the reservoir, which altogether requires 
\[
	\ell_1|\ccC|\le \frac{50\ell n}{M}\le \theta_{\star\star}^2n
\]
vertices from the reservoir. In other words, there arises no problem if we choose these 
connections one by one, thus creating a path $T$ possessing $|\ccC|(\ell_1+M)+|V(P_A)|$
vertices. Moreover, the set $V(T)\cap \cR$ of used vertices is so small that we can still 
make a last connection to close the desired cycle. For this last connection we use $\ell_i$
inner vertices, where $i\in [4]$ is determined in such a way 
that $i\equiv n-|V(T)|\pmod{4}$. In this manner, we obtain a cycle $C$ containing the absorbing 
path $P_A$ such that the set $Z=V\setminus V(C)$ of left-over vertices satisfies 
\[
	|Z|\equiv n-|V(C)| \equiv n-|V(T)|-\ell_i\equiv i-\ell_i\equiv 0 \pmod{4}
\]
as well as 
\[
	|Z|
	=
	|Z\setminus \cR|+|Z\cap \cR|
	\le 
	|J|+|\cR| 
	\le 
	2\theta_\star^2 n\,.
\]
So by property~\ref{it:QA3} of the absorbing path we can absorb $Z$ into $P_A$, 
thus arriving at the desired Hamiltonian cycle. Thereby Theorem~\ref{thm:main} is proved. 
\end{proof}

\appendix

\section{Weighted Janson Inequality}

In the proof of Claim~\ref{clm:2042} we use a probabilistic concentration result 
that follows from the following weighted variant of Janson's inequality.

%\vbox{
\begin{lemma}[Weighted Janson Inequality]\label{lem:2726}
	For a nonempty set $V$ and $p\in [0, 1]$ let $V_p$ be the binomial 
	subset of $V$ including every element of $V$ independently and uniformly at random 
	with probability $p$. Let $w\colon \powerset(V)\longrightarrow\RR_{\ge 0}$ be a 
	weight function and let 
	\[
		X=\sum_{A\in\powerset(V)}w(A)\Ind_{A\subseteq V_p}
	\]
	be the random variable giving the total weight of $\powerset(V_p)$. Setting
	\[
		\Delta
		=
		\sum_{\substack{A, B\in\powerset(V)\\ A\cap B\ne \emptyset}}w(A)\,w(B)\,\PP(A\cup B\subseteq V_p)
	\]
	we have 
	\[
		\PP\bigr(X\le \EE X - t\bigr) 
		\le 
		\exp\left(-\frac{t^2}{2\Delta}\right)
	\]
	for every $t\in [0, \EE X]$. 
\end{lemma}%}

It is straightforward to check that Janson's original proof (see e.g.~\cite{JLR00}) 
extends to this weighted setting but for the sake of completeness we give the details.

\begin{proof}
		Let $\Psi\colon \RR_{\ge_0}\longrightarrow \RR_{>0}$ be the 
		function $s\longmapsto \EE[\eu^{-sX}]$. Clearly, $\Psi$ is differentiable with 
		the derivative 
		\begin{equation}\label{eq:1407}
			-\Psi'(s)
			=
			\EE[X\eu^{-sX}]
			=
			\sum_{A\subseteq V} w(A)\,\PP(A\subseteq V_p)\,\EE[\eu^{-sX}\coprn A\subseteq V_p]\,.
		\end{equation}
		For every $A\subseteq V$ we split $X=Y_A+Z_A$, where
		\[
			Y_A=\sum_{A\cap B\ne \varnothing}w(B)\Ind_{B\subseteq V_p}
			\qand
			Z_A=\sum_{A\cap B = \varnothing}w(B)\Ind_{B\subseteq V_p}\,.
		\]
		Now the FKG inequality yields 
		\[
			\EE[\eu^{-sX}\coprn A\subseteq V_p]
			\ge
			\EE[\eu^{-sY_A}\coprn A\subseteq V_p]\cdot\EE[\eu^{-sZ_A}\coprn A\subseteq V_p]\,,
		\]
		where in view of the independence of $A\subseteq V_p$ and $Z_A$ the second factor
		is at least $\Psi(s)$. Applying the trivial estimate $\eu^{-x}\ge 1-x$ to the first factor 
		we obtain
		\[
			\EE[\eu^{-sX}\coprn A\subseteq V_p]
			\ge 
			\EE[1-sY_A\coprn A\subseteq V_p]\cdot\Psi(s)
		\]
		for every $A\subseteq V$ and by plugging this into~\eqref{eq:1407} we arrive at 
		\begin{align*}
			-\frac{\Psi'(s)}{\Psi(s)}
			&\ge 
			\sum_{A\subseteq V} w(A)\,\PP(A\subseteq V_p)\,\EE[1-sY_A\coprn A\subseteq V_p] \\
			&=
			\sum_{A\subseteq V} w(A)\,\PP(A\subseteq V_p)
			-
			s\sum_{A\cap B\ne \varnothing} w(A)\,w(B)\,\PP(A\cup B\subseteq V_p) \\
			&=
			\EE X-s\Delta\,.
		\end{align*}
		Integrating over $s$ and taking $\Psi(0)=1$ into account we conclude 
		\[
			\log\bigl(\Psi(u)\bigr)
			=
			\int_0^u \frac{\Psi'(s)}{\Psi(s)}\,\mathrm{d} s
			\le
			\int_0^u \bigl(s\Delta-\EE X\bigr)\,\mathrm{d} s
			=
			\frac{u^2\Delta}2 - u\EE X
		\]
		for every $u\in \RR_{\ge 0}$. Finally, Markov's inequality implies 
		\begin{align*}
			\PP\bigl(X\le \EE X-t\bigr)
			&=
			\PP\bigl(\eu^{-uX}\ge \eu^{u(t-\EE X)}\bigr)
			\le
			\exp\bigl(u(\EE X-t)\bigr)\,\EE[\eu^{-uX}] \\
			&\le
			\exp\bigl(u(\EE X-t)+u^2\Delta/2 - u\EE X\bigr)
			=
			\exp(u^2\Delta/2-tu)
		\end{align*}
		for every $u\in\RR_{\ge 0}$ and the optimal choice $u=\frac t\Delta$
		discloses 
		\[
			\PP\bigl(X\le \EE X-t\bigr)
			\le 
			\exp\left(-\frac{t^2}{2\Delta}\right)\,. \qedhere
		\]
\end{proof}

For bounded weight functions we deduce the following version. 

\begin{cor} \label{cor:1840}
	Suppose that $|V|\ge m\ge k\ge 1$, where $V$ is a finite set and $k$ is an integer.
	For $p=m/|V|$ let $V_p\subseteq V$ be the binomial subset of $V$ including 
	every element independently and uniformly at random with probability $p$.
	If $w\colon V^{(k)}\longrightarrow [0, 1]$ denotes a bounded weight function, then 
	the random variable $X=\sum_{A\in V^{(k)}}w(A)\Ind_{A\subseteq V_p}$ satisfies 
	\[
		\PP\bigl(|X-\EE X|\ge \xi m^k\bigr)\le 3\exp\left(-\frac{\xi^2m}{12k^2}\right)
	\]
	for every $\xi\in (0, 1)$.
\end{cor}	
	
\begin{proof}
	In order to make Lemma~\ref{lem:2726} applicable we set $w(A)=0$ for 
	every $A\in\powerset(V)\setminus V^{(k)}$. Now for $t=\xi m^k$ we obtain
	\begin{equation}\label{eq:2734}
		\PP\bigl(X\le \EE X-\xi m^k\bigr) 
		\le 
		\exp\left(-\frac{\xi^2m^{2k}}{2\Delta}\right)\,,
	\end{equation}
	where 
	\[
		\Delta
		=
		\sum_{\substack{A, B\in V^{(k)}\\ A\cap B\ne \emptyset}}w(A)\,w(B)\,\PP(A\cup B\subseteq V_p)
		\le 
		\sum_{\substack{A, B\in V^{(k)}\\ A\cap B\ne \emptyset}} p^{|A\cup B|}\,.
	\]
	Since for every $i\in [k]$ there are at most $|V|^{2k-i}$ 
	pairs $(A, B)\in V^{(k)}\times V^{(k)}$ with the property $|A\cup B|=2k-i$, 
	we are thus lead to the upper bound 
	\[
		\Delta
		\le
		\sum_{i=1}^k |V|^{2k-i} p^{2k-i}
		\le 
		km^{2k-1}\,.
	\]
	Therefore~\eqref{eq:2734} implies 
	\begin{equation}\label{eq:2737}
		\PP\bigl(X\le \EE X-\xi m^k\bigr) 
		\le 
		\exp\left(-\frac{\xi^2m}{2k}\right)
	\end{equation}
	and to conclude the argument it suffices to prove 
	\begin{equation}\label{eq:2739}
		\PP\bigl(X\ge \EE X+\xi m^k\bigr) 
		\le 
		2\exp\left(-\frac{\xi^2m}{12k^2}\right)\,.
	\end{equation}

	To this end we apply~\eqref{eq:2737} to the weight function $\wh{w}(A)=1-w(A)$
	and to $\xi/2$ instead of~$\xi$, thus learning that the random variable
	\[
		Y
		=
		\sum_{A\in V^{(k)}}\bigl(1-w(A)\bigr)\Ind_{A\subseteq V_p}
		=
		\binom{|V_p|}k - X
	\]
	satisfies
	\[
		\PP\bigl(Y\le \EE Y-\tfrac 12\xi m^k\bigr) 
		\le 
		\exp\left(-\frac{\xi^2m}{8k}\right)\,.
	\]
	Rewriting this in terms of $X$ and taking into account that the expected value 
	of $\binom{|V_p|}k$ is~$p^k\binom{|V|}k$ we obtain
	\begin{equation}\label{eq:2757}
		\PP\left(X\ge \EE X + \binom{|V_p|}k - p^k\binom{|V|}k + \frac \xi2 m^k\right)
		\le
		\exp\left(-\frac{\xi^2m}{8k}\right)\,.
	\end{equation}
	As we shall prove below, the number 
	\[
		m^+=m\left(1+\frac\xi{2k}\right)
	\]
	satisfies 
	\begin{equation}\label{eq:2809}
		\binom{m^+}{k}\le p^k\binom{|V|}k+\frac \xi2 m^k\,.
	\end{equation}
	Assuming this estimate for a moment, we conclude from~\eqref{eq:2757} that 
	\begin{multline*}
		\PP\bigl(X\ge \EE X+\xi m^k \text{ and } |V_p|\le m^+\bigr)\\
		\le
		\PP\left(X\ge \EE X+\binom{m^+}{k}-p^k\binom{|V|}k+\frac\xi2 m^k 
			\text{ and } |V_p|\le m^+\right)
		\le 
		\exp\left(-\frac{\xi^2m}{8k}\right)\,.
	\end{multline*}
	Together with Chernoff's inequality this yields 
	\begin{align*}
			\PP\bigl(X\ge \EE X+\xi m^k\bigr)
			&\le 
			\PP(|V_p|>m^+)
			+
			\PP\bigl(X\ge \EE X+\xi m^k \text{ and } |V_p|\le m^+\bigr) \\
			&\le 
			2\exp\left(-\frac{\xi^2m}{12k^2}\right)\,,
	\end{align*}
	which concludes the proof of~\eqref{eq:2739} and, hence, of Corollary~\ref{cor:1840}.	
	
	Now it remains to deal with~\eqref{eq:2809}. Since $p=m/|V|\le 1$ we have $p(|V|-j)\ge m-j$
	for every $j\in [0, k-1]$ and multiplying these estimates we 
	infer $p^k\binom{|V|}k\ge \binom mk$. Thus it suffices to prove 
	\begin{equation}\label{eq:1852}
		\binom{m^+}k-\binom mk \le \frac\xi2 m^k\,.
	\end{equation}
	Applying the mean value theorem to the increasing function $x\longmapsto\binom xk$
	we obtain
	\[
		\binom{m^+}k-\binom mk 
		\le 
		\frac{(m^+-m)(m^+)^{k-1}}{(k-1)!}\,,
	\]
	so~\eqref{eq:1852} is a consequence of 
	\[
		\left(1+\frac\xi{2k}\right)^{k-1}\le k!\,,
	\]
	which is clear for $k=1$ and which for $k\ge 2$ follows from 
	\[
		\left(1+\frac\xi{2k}\right)^{k-1}
		\le 
		\eu^{(k-1)/2k}
		\le 
		\sqrt{\eu}
		\le 
		2\,. \qedhere
	\]
\end{proof}	

The following consequence of this result is utilised multiple times 
in Section~\ref{sec:long_path}.

\begin{cor}\label{cor:1041}
	Let $m\ge k$ and $M$ be positive integers, and let $\eta\in \big(0, \frac{1}{2k}\big)$. Suppose that~$V$ 
	is a finite set and that 
	\[
		V=B_1\dcup\ldots\dcup B_\nu\dcup Z
	\]
	is a partition with $|B_1|=\ldots =|B_\nu|=M<\eta |V|$, $|Z|<\eta |V|$, and $\nu\ge m$. 
	Let $\cS\subseteq \{B_1, \ldots, B_\nu\}$ be an $m$-element subset chosen uniformly
	at random and set $S=\bigcup \cS$.
	\begin{enumerate}[label=\alabel]
		\item\label{it:1041a} If $Q\subseteq V^k$ has size $|Q|=d |V|^k$, then 
			\[
				\PP\bigl(\big||Q\cap S^k|-d(Mm)^k\big|\ge \xi (Mm)^k\bigr) 
				\le 
				12\sqrt{m}\exp\left(-\frac{\xi^2m}{48k^{2k+2}}\right)
			\]
			holds for every real $\xi$ with $\max(8k^2\eta, 16k^2/m) < \xi < 1$.
		\item\label{it:1041b} Similarly, if $G$ denotes a $k$-uniform hypergraph with vertex 
			set $V$ and $d|V|^k/k!$ edges, then 
			\[
				\PP\bigl(\big|e_G(S)-d(Mm)^k/k!\big|\ge \xi (Mm)^k/k!\bigr) 
				\le 
				12\sqrt{m}\exp\left(-\frac{\xi^2m}{48k^{2k+2}}\right)
			\]
			holds for every $\xi$ with $\max(8k^2\eta, 16k^2/m) < \xi < 1$.
	\end{enumerate} 
\end{cor}

\begin{proof}
	Notice that~\ref{it:1041a} implies~\ref{it:1041b}. Indeed given a $k$-uniform hypergraph $G$
	we apply~\ref{it:1041a} to the ordered version of its set of edges defined by
	\[
		Q=\bigl\{(x_1, \ldots, x_k)\in V^k\colon \{x_1, \ldots, x_k\}\in E(G)\bigr\}
	\]
	and we obtain~\ref{it:1041b} immediately. 
	
	So it remains to verify~\ref{it:1041a}. Intending to invoke Corollary~\ref{cor:1840}
	we move from the hypergeometric distribution involved in choosing the set $\cS$
	to a binomial distribution, where we include every block $B_i$ independently from 
	the other ones with probability $p=m/\nu$. For this transition we introduce the 
	following notation.  
	
	Write $\ccB=\{B_1, \ldots, B_\nu\}$ for the set of {\it blocks} and consider the event
	\[
		\gX=\left\{\cS\subseteq \ccB\colon \text{the set $S=\bigcup \cS$ satisfies 
			$\big||Q\cap S^k|-d(Mm)^k\big|\ge \xi (Mm)^k$}\right\}\,.
	\]
	Let $\cS_m\subseteq \ccB$ be an $m$-element set chosen uniformly at random and 
	let $\cS_{p}\subseteq \ccB$ be a binomial subset containing every block independently and 
	uniformly at random with probability $p=m/\nu$. 
	Pittel's inequality (see, e.g., \cite{JLR00}*{eq.\ (1.6)}) informs us that -- without any 
	assumptions on the event $\gX\subseteq \powerset(\ccB)$ -- we have  
	\[
			\PP\left(\cS_m\in \gX\right)
			\le 
			3\sqrt{m}\,\,\PP\left(\cS_{p}\in \gX\right)\,,
	\]
	so it suffices to show 
	\begin{equation}\label{eq:2157}
		 \PP\left(\cS_{p}\in \gX\right)\le 4\exp\left(-\frac{\xi^2m}{48k^{2k+2}}\right)\,.
	\end{equation}
	We will exploit that most $k$-tuples in $Q$ are {\it crossing} in the sense that their 
	entries belong to $k$ distinct blocks. More precisely, if $Q_\circ\subseteq Q$ denotes
	the set of theses crossing $k$-tuples, we contend that 
	\begin{equation}\label{eq:1201}
		\left(d-\frac\xi4\right)(M\nu)^k 
		\le 
		|Q_\circ|
		\le 
		\left(d+\frac\xi4\right)(M\nu)^k\,.
	\end{equation}

	To justify the lower bound we remark that at most $k|Z||V|^{k-1}$ members
	of $Q$ can have an entry in $Z$ and at most $k^2M|V|^{k-1}$ members of $Q$ 
	can have two entries from the same block, whence 
	\[
		|Q_\circ|
		\ge 
		|Q|-k|Z||V|^{k-1}-k^2M|V|^{k-1} 
		\ge
		(d-k\eta-k^2\eta)|V|^k
		\ge
		\left(d-\frac\xi4\right)(M\nu)^k\,.
	\]
	For the upper bound we exploit 
	\[
		(M\nu)^k
		=
		|V\setminus Z|^k\ge (1-\eta)^k|V|^k
		\ge 
		(1-k\eta)|V|^k\,,
	\]
	which yields 
	\[
		|Q_\circ|
		\le 
		|Q|
		=
		d |V|^k
		\le
		\left(d+\frac\xi4\right)(1-k\eta)|V|^k
		\le
		\left(d+\frac\xi4\right)(M\nu)^k\,.
	\]
	Thereby~\eqref{eq:1201} is proved. Now we decompose 
	\[
		|Q_\circ|=\sum_{\{i(1), \ldots, i(k)\}\in [\nu]^{(k)}} W\bigl(i(1), \ldots, i(k)\bigr)\,,
	\]
	where for every $k$-element set $\{i(1), \ldots, i(k)\}\in [\nu]^{(k)}$ the number 
	of $k$-tuples in $Q_\circ$ with one entry from each of the blocks $B_{i(1)}, \ldots, B_{i(k)}$ 
	is denoted by $W\bigl(i(1), \ldots, i(k)\bigr)$. These numbers are bounded by  
	\[
		 0\le W\bigl(i(1), \ldots, i(k)\bigr)\le k!M^k\,.
	\]
	Bearing in mind that Corollary~\ref{cor:1840} requires normalised weights, we set  
	\[
		 w\bigl(i(1), \ldots, i(k)\bigr)=\frac{W\bigl(i(1), \ldots, i(k)\bigr)}{k!M^k}
	\]
	for every $k$-set $\{i(1), \ldots, i(k)\}\in [\nu]^{(k)}$. As a consequence 
	of~\eqref{eq:1201}, the expectation of the random variable 
	\[
		X
		=
		\sum_{A\in [\nu]^{(k)}} w(A)\Ind_{A\subseteq [\nu]_p}
		=
		\frac{|S_p^k\cap Q_\circ|}{k!M^k}\,,
	\]
	where $S_p=\bigcup \cS_p$, is 
	\[
		\EE X=\frac{|Q_\circ|p^k}{k!M^k}=\frac{(d\pm \xi/4)m^k}{k!}\,.
	\]
	Therefore, Corollary~\ref{cor:1840} applied to $\xi/(2k!)$ here in place of $\xi$
	there yields 
	\[
		\PP\left(\big||S_p^k\cap Q_\circ|-d(Mm)^k\big|\ge (3/4)\xi(Mm)^k\right)
		\le
		3\exp\left(-\frac{\xi^2m}{48k^{2k+2}}\right)\,.
	\]
	So to conclude the proof of \eqref{eq:2157} is certainly suffices to show 
	\[
		\PP\left(\big|S_p^k\cap (Q\setminus Q_\circ)\big| \ge (1/4)\xi(Mm)^k\right)
		\le
		\exp\left(-\frac{m}{48k^2}\right)\,.
	\]
	Now Chernoff's inequality yields
	\[
		\PP\bigl(|\cS_p| > (1+1/k)m\bigr)
		\le 
		\exp\left(-\frac{m}{48k^2}\right)
	\]
	and for this reason it suffices to prove the deterministic statement that for 
	every $\ccA\subseteq \ccB$ with $|\ccA|\le m(1+1/k)$ the set $A=\bigcup \ccA$ satisfies 
	\[
		 |A^k\cap (Q\setminus Q_\circ)| 
		 <
		 (1/4)\xi(Mm)^k\,.
	\]
	Since the $k$-tuples counted on the left side contain two entries from the same block, 
	we have indeed
	\begin{align*}
		 |A^k\cap (Q\setminus Q_\circ)| 
		 &\le
		 k^2M|A|^{k-1}
		 \le
		 (1+1/k)^{k-1}k^2M(Mm)^{k-1} \\
		 &<
		 (4k^2/m)(Mm)^k
		 <
		 (\xi/4)(Mm)^k\,,
	\end{align*}
	where the last inequality uses our assumed lower bound on $\xi$.
\end{proof}

\subsection*{Acknowledgement} We would like to thank the referee for 
dealing with this article very efficiently and for many constructive 
suggestions.

\end{document}